%% file: main.tex
\documentclass[usletter,10pt]{article}

\pdfoutput=1

\input{header}

\date{}

\title{Wavelet-Fourier \corsing techniques for multi-dimensional advection-diffusion-reaction equations}
\author{S. Brugiapaglia\footnote{Department of Mathematics and Statistics, Concordia University, Montr\'{e}al, QC, Canada.\newline E-mail address (corresponding author): \texttt{simone.brugiapaglia@concordia.ca}}, S. Micheletti\footnote{MOX, Dipartimento di Matematica, Politecnico di Milano, 20133 - Milano, Italy.\newline E-mail address: \texttt{stefano.micheletti@polimi.it}}, F. Nobile\footnote{MATHICSE-CSQI, \'{E}cole Polytechnique F\'{e}d\'{e}rale de Lausanne, Lausanne, CH-1015, Switzerland.\newline E-mail address: \texttt{fabio.nobile@epfl.ch}},  S. Perotto\footnote{MOX, Dipartimento di Matematica, Politecnico di Milano, 20133 - Milano, Italy.\newline 
E-mail address: \texttt{simona.perotto@polimi.it}}}

\begin{document}

%\subjclass[2010]{Primary 65N30, 65Y20, 94A20; Secondary 65T40, 65K10, 42A61}

%\keywords{Compressed sensing, Petrov-Galerkin formulation, advection-diffusion-reaction equation, inf-sup property, local coherence}

\maketitle

\begin{abstract}
We present and analyze a novel wavelet-Fourier technique for the numerical treatment of multi-dimensional advection-diffusion-reaction equations based on the \textsf{CORSING} (\textsf{COmpRessed SolvING}) paradigm. Combining the Petrov-Galerkin technique with the compressed sensing approach, the proposed method is able to approximate the largest coefficients of the solution with respect to a biorthogonal wavelet basis. Namely, we assemble a compressed discretization based on randomized subsampling of the Fourier test space and we employ sparse recovery techniques to approximate the solution to the PDE. In this paper, we provide the first rigorous recovery error bounds and effective  recipes for the implementation of the \textsf{CORSING} technique in the multi-dimensional setting. Our theoretical analysis relies on new estimates for the local $a$-coherence, which measures interferences between wavelet and Fourier basis functions with respect to the metric induced by the PDE operator. The stability and robustness of the proposed scheme is shown by numerical illustrations in the one-, two-, and three-dimensional case. 
\end{abstract}

%\tableofcontents

%%%%%%%%%%%%%%%%%%%%%%%%%%%%%%%%%%%%%%%%%%%%%%%%%%%%%%%%%%%%%%%%%%%%%%%%%%%%
%%%%%%%%%%%%%%%%%%%%%%%%%%%%%%%%%%%%%%%%%%%%%%%%%%%%%%%%%%%%%%%%%%%%%%%%%%%%
\section{Introduction}
%%%%%%%%%%%%%%%%%%%%%%%%%%%%%%%%%%%%%%%%%%%%%%%%%%%%%%%%%%%%%%%%%%%%%%%%%%%%
%%%%%%%%%%%%%%%%%%%%%%%%%%%%%%%%%%%%%%%%%%%%%%%%%%%%%%%%%%%%%%%%%%%%%%%%%%%%

This paper deals with the theoretical analysis and the numerical implementation of a recently-introduced paradigm in numerical approximation of Partial Differential Equations (PDEs), named \textsf{COmpRessed SolvING} (in short, \corsing). The \corsing method has been  proposed and studied in \cite{PhDThesis,Brugiapaglia2015,Brugiapaglia2018} for the solution of linear PDEs set in Hilbert spaces, and combines the Petrov-Galerkin discretization techniques with compressed sensing \cite{Candes2006,Donoho2006}. Assuming the sparsity of the solution with respect to a suitable trial function basis, the idea is to build a reduced Petrov-Galerkin discretization of the weak formulation of the problem by considering a randomly subsampled test subspace, and then to recover a sparse approximation to the solution via sparse recovery techniques, such as $\ell^1$-minimization or the greedy Orthogonal Matching Pursuit (OMP) algorithm. As discussed in \cite{Brugiapaglia2018}, the main advantages of \corsing with respect to adaptive finite element techniques for PDEs are that no \emph{a posteriori} error estimators are needed and that assembly and recovery via OMP are fully parallelizable.

In this paper, we focus on the scenario of multi-dimensional advection-diffusion-reaction (ADR) equations on a torus, and we employ biorthogonal wavelets as trial functions basis and a Fourier basis as test functions. In this way, we introduce a hybrid  Wavelet-Fourier technique named \corsing \WF, which is able to approximate the largest wavelet coefficients of the solution to the PDE by sampling randomly the Fourier test space, using a suitable probability measure. 

It is worth noticing that the results showed here can be extended to the nonperiodic case by considering biorthogonal wavelets on the interval (and on the hypercube) (see \cite{Dahmen1997,Pabel2015,Urban2008}) as the trial functions basis and the sine basis $\{\sin(\pi k x)\}_{k \in \mathbb{N}}$ (or a tensorized version of it) as the test functions basis. However, we decided to stick to the periodic case to make the theoretical exposition free of an excessive quantity of technical details regarding the construction of wavelets at the boundary. In this respect, the present work should be considered as  a first step towards the setup of \corsing in practical applications. We also observe that the construction of wavelets over general domains is not a straightforward task. Hence, the \textsf{CORSING} $\mathcal{WF}$ shares this difficulty with classic wavelets methods for PDEs. We refer to \cite{CossIGA} for an extension and implementation of the \textsf{CORSING} method to domains with more general geometries via the IsoGeometric Analysis principle \cite{hughes2005isogeometric}.

The choice of biorthogonal wavelet functions is motivated by the need of working with trial and test functions bases that satisfy the \emph{Riesz basis} property \cite{Brugiapaglia2018}. Roughly speaking, this assumption guarantees a control on the condition number of the Petrov-Galerkin discretization matrix, which is crucial for a successful application of the compressed sensing paradigm (see, e.g., \cite[Theorem 1.21]{PhDThesis} or \cite[Theorem 3.6]{adcock2019uniform}). In \cite{Brugiapaglia2018}, the authors considered the hierarchical basis of hat functions and the sine functions basis as trial and test basis, respectively. While this choice ensures the Riesz basis property in the one-dimensional setting, tensorizing these two bases breaks the Riesz basis property in the multi-dimensional framework. Indeed, the tensorization of the sine functions basis is a Riesz basis (up to suitable diagonal rescaling) in any dimension, while the Riesz property is broken for the tensorized hierarchical basis of hat functions. However, thanks to the so-called \emph{norm equivalence} property, tensorized biorthogonal wavelets form a Riesz basis in any dimension, hence providing a remedy for this issue.

Although compressed sensing is becoming a standard paradigm for signal processing applications, understanding its full potential and limitations in scientific computing is still object of active work. This paper moves a step forward in this direction. In a fast-growing literature, it is worth mentioning here the applications of compressed sensing to numerical approximation of high-dimensional functions (see, e.g., \cite{Adcock2017correcting, Adcock2017,Chkifa2018,Rauhut2012}) and of parametric PDEs, with special emphasis on uncertainty quantification (see, e.g., \cite{Bouchot2017multi,Doostan2011,Rauhut2017,Yang2013}). In these cases, a smooth function, which can be the quantity of interest of a parametric PDE, is approximated with respect to a global sparsity basis like orthogonal polynomials by means of random pointwise observations. The PDE solver is a black box used to evaluate the quantity of interest for different values of the parameters, and  compressed sensing is performed \emph{outside} the black box. Our focus is different, since the \corsing method takes advantage of the compressed sensing paradigm \emph{inside} the black box, i.e., to solve the PDE itself given a particular choice of the parameters. 

%Recovering the best $s$-term approximation to the solution of a PDE with respect to a biorthogonal wavelet family is also the goal of adaptive wavelet methods for PDEs (see \cite{Cohen2003,Dahmen1997,Urban2008} and references therein). The theoretical analysis of \cite{Cohen2001} shows that adaptive wavelet methods can perform this task in $O(s)$ flops. From a theoretical viewpoint, such a low budget seems to be out of reach for \corsing. Indeed, a straightforward implementation of \corsing with OMP recovery requires $O(s m N)$ flops, where $m$ is the number of random test functions employed and $N$ is the trial space dimension. However, implementing adaptive wavelet methods is a challenging software engineering exercise (several technical aspects are discussed in detail in \cite{Pabel2015}). On the other hand, the implementation of the \corsing approach is less demanding (although nontrivial) from the software engineering perspective  and can be easily parallelized in order to lessen the impact of $N$ on the overall computational cost.

The compressed sensing principle has also been recently employed for the efficient numerical approximation of diffusion equations via Sturm-Liouville spectral collocation  in \cite{brugiapaglia2018compressive}. Finally, we observe that wavelet-Fourier techniques are widely used in signal processing applications (see, e.g., \cite{Adcock2015,BreakingBarrier,Krahmer2014} for theoretical contributions in this direction). Yet, to the best of our knowledge, this paper is the first comprehensive study of this type of techniques in the context of numerical approximation of multidimensional PDEs.

%Being the \corsing approach nonadaptive by construction, there is no need for \emph{a posteriori} error estimators (which may be hard to find) and the method is easily parallelizable when combined with OMP. Fixed a trial function basis of dimension $N$ and a sparsity level $s \ll N$, under suitable conditions, the method can recover the best $s$-term approximation to the solution by means using $m \ll N$ test functions and with an overall computational cost $O(smN)$ with OMP. The theory presented in \cite{Brugiapaglia2018} based on the so-called restricted inf-sup property gives the condition $m \sim s^2 \times \text{(log factors)}$, whereas an analysis based on the restricted isometry property shows that $m \sim s \times \text{(log factors)}$ is sufficient (see \cite{PhDThesis}). A first theoretical analysis of \corsing for 1D ADR problems is carried out in \cite{Brugiapaglia2018} and in \cite[Section 3.2.5]{PhDThesis} the method is successfully applied to the 2D Stokes problem and to the 3D Laplace equation.

%In this paper, we extend the discussion and the analysis of \corsing to the multi-dimensional case, employing wavelets as a trial basis and Fourier as a test basis. The choice of wavelets allows us to generate sparsity in solutions that exhibit local features or sharp transitions such as boundary layers, as it is common in fluid-dynamics applications. For the sake of simplicity, we focus on the case of periodic boundary conditions. 

\subsection{Main contributions}
\label{sec:main_contr}

The main contribution of this paper is threefold. First, we present and study the first hybrid wavelet-Fourier discretization for ADR equations in arbitrary dimension on the torus based on Petrov-Galerkin discretization and on compressed sensing, named \corsing \WF\ (see Section~\ref{sec:CORSING}). Moreover, in Section~\ref{sec:localcoherence} we show the applicability of the theoretical analysis in \cite{Brugiapaglia2018} to $n$-dimensional ADR equations with constant coefficients for $n \geq 1$ and with nonconstant coefficients for $n = 1$.
Finally, in Section~\ref{sec:numerics} we provide a Matlab$^\text{\textregistered}$ implementation of the \corsing \WF method for $n$-dimensional ADR equations, with $n = 1,2,3$.

In view of the aforementioned contributions, we will focus on the following three key technical issues necessary to implement and quantify the performance of the \corsing \WF method (more details on these three issues are given in Section~\ref{sec:i,ii,iii}):
\begin{itemize}
\item [(i)] Find  a suitable truncation condition on the Fourier test space in order to guarantee stability of the resulting Petrov-Galerkin discretization.
\item [(ii)] Give lower bounds to the sampling complexity, i.e., the minimum number of randomly selected Fourier test functions needed to recover the $s$ dominant coefficients of the solution in the wavelet expansion.
\item [(iii)] Provide explicit expressions for the probability distribution on the Fourier test space needed for the random selection of the basis functions.
\end{itemize}
In order to address issues (i), (ii), and (iii), we will take advantage of the general framework given in \cite{Brugiapaglia2018} for the analysis of \corsing, based on the so-called local $a$-coherence, which can be interpreted as a measure of the angle between the wavelet trial functions and the Fourier test functions with respect to the metric induced by the sesquilinear form associated with the ADR problem (see Definition~\ref{def:mu}). In view of this, our main  efforts will be aimed at producing upper bounds to the local $a$-coherence in Section~\ref{sec:localcoherence} for the specific ADR problems addressed. In particular, we provide local $a$-coherence upper bounds for the 1D case with nonconstant coefficients in Theorem~\ref{thm:mu_AReq1D}, and for the multi-dimensional case with constant coefficients, when employing anisotropic and isotropic tensor product wavelets in Theorems~\ref{thm:mu_bound_multi_ani} and \ref{thm:CORSING_rec_multi_iso}, respectively. As a consequence, we derive explicit and computable answers to  (i), (ii), and (iii) and provide recovery error guarantees for the \corsing \WF method in Theorem~\ref{thm:CORSING_rec_1D} (1D case), Theorem~\ref{thm:CORSING_rec_multi_ani} (multi-dimensional anisotropic case), and Theorem~\ref{thm:CORSING_rec_multi_iso} (multi-dimensional isotropic case). Numerical results in Section~\ref{sec:numerics} confirm the theoretical findings.

%%%%%%%%%%%%%%%%%%%%%%%%%%%%%%%%%%%%%%%%%%%%%%%%%%%%%%%%
%%%%%%%%%%%%%%%%%%%%%%%%%%%%%%%%%%%%%%%%%%%%%%%%%%%%%%%%
\section{Problem setting}
\label{sec:problem}

In this section, we recall some basics on periodic Sobolev spaces and discuss ADR problems in weak form with periodic boundary conditions.

\paragraph{Notation}
We denote by $\NN := \{1,2,3,\ldots\}$,  $\NN_0 := \{0\} \cup  \NN $, and $\ZZ:=\NN_0 \cup (-\NN)$.  We define $[k]:= \{1,\ldots,k\}$ and $[k]_0 := \{0\} \cup [k]$. The notation $X \lesssim Y$ means $X \leq C Y$, with $C >0$ a constant independent of $X$ and $Y$; $X \sim Y$ means that $X \lesssim Y$ and $X \gtrsim Y$ hold simultaneously. By $X \propto Y$, we understand that there exists a constant $C > 0$ independent of $X$ and $Y$ such that $X = CY$. Given a multi-index $\bm{r}$, we denote its 2-norm by $|\bm{r}|$. For every $z \in \mathbb{C}$, $|z|$ is its modulus, $\overline{z}$ is its complex conjugate, and $\Real(z)$ and $\Imag(z)$ denote its real and imaginary part, respectively. Given a vector $\bm{x} \in \mathbb{C}^n$ and $1\leq p < +\infty$, then $\|\bm{x}\|_p := (\sum_{j =1}^n |x_j|^p)^{1/p}$, $\supp(\bm{x}):= \{j \in [n] : x_j \neq 0\}$, and $\|\bm{x}\|_0 := |\supp(\bm{x})|$. Inequalities between vectors in $\mathbb{R}^n$ have to be read componentwise; for example, $\bm{x} \leq \bm{y}$ means $x_i \leq y_i$ for every $i$. The vectors of the canonical basis of $\mathbb{C}^n$ are denoted by $\bm{e}_1,\ldots,\bm{e}_n$ and $\bm{x}\cdot\bm{y} := x_1\overline{y}_1 + \cdots + x_n \overline{y}_n$ is the standard inner product of $\mathbb{C}^n$. Given a matrix $A \in \mathbb{C}^{m \times n}$, $A^*$ denotes its conjugate transpose. The set of sequences indexed by integer multi-indices is denoted by $\ell(\mathbb{Z}^n) := \{(x_{\bm{j}})_{\bm{j} \in \mathbb{Z}^n} : x_{\bm{j}} \in \mathbb{C}, \, \forall \bm{j} \in \mathbb{Z}^n\}$.

\subsection{Sobolev spaces} 
\label{sec:Sobolev}
We start by recalling some standard notions about periodic Sobolev spaces. Let $\Dim \in \mathbb{N}$ and consider the domain 
$$
\mathcal{D} = (0,1)^\Dim \subseteq \mathbb{R}^\Dim. 
$$
Given $k \in \mathbb{N}_0$, let $H^k(\mathcal{D}) = H^k(\mathcal{D}; \mathbb{C})$ be the Sobolev space of order $k$ of complex-valued functions over $\mathcal{D}$, being understood  that $H^0(\mathcal{D}) = L^2(\mathcal{D})$. Moreover, we denote the $H^k(\mathcal{D})$-inner product by
\begin{equation}
(u,v)_{k} 
:= \sum_{\substack{\bm{\alpha} \in [k]^\Dim_0 \\ \alpha_1 + \cdots + \alpha_\Dim \leq k}} \int_{\mathcal{D}} D^\bm{\alpha} u(\bm{x}) \overline{D^\bm{\alpha} v(\bm{x})} \de\bm{x}, 
\end{equation}
where $D^{\bm{\alpha}} 
:= \frac{\partial^{\alpha_1 +\cdots +\alpha_\Dim}}{\partial x_1^{\alpha_1} \cdots \partial x_\Dim^{\alpha_\Dim}}$ is the derivative in the sense of distributions and $[k]_0^n:=[k]_0 \times \cdots \times [k]_0$ ($n$ times). For the sake of simplicity, we use the shorthand notation $(\cdot,\cdot):=(\cdot,\cdot)_{0}$ for the $L^2(\mathcal{D})$-inner product. The $H^k(\mathcal{D})$-norm is defined as $\|\cdot\|_{H^k(\mathcal{D})}^2 = (\cdot,\cdot)_{k}$ and the $H^k(\mathcal{D})$-seminorm is given by 
$$
|u|_{H^k(\mathcal{D})} ^2
:= \sum_{\substack{\bm{\alpha} \in [k]^\Dim_0 \\ \alpha_1 + \cdots + \alpha_\Dim = k}} \|D^{\bm{\alpha}}u\|^2_{L^2(\mathcal{D})}.
$$
The periodic Sobolev space of order $k$, $H^k_\per(\mathcal{D}) \subseteq H^k(\mathcal{D})$, is then defined as
$$
H^k_\per(\mathcal{D}) 
:= \clos_{\|\cdot\|_{H^k(\mathcal{D})}}(C^\infty_\per(\mathcal{D})),
$$
where 
$$
C^\infty_\per(\mathcal{D}) 
:= \{v|_{\mathcal{D}}:  v\in C^\infty(\mathbb{R}^\Dim), \; v(\bm{x} + \bm{e}_i) = v(\bm{x}), \; \forall \bm{x} \in \mathbb{R}^n, \; \forall i \in [\Dim] \}.
$$
Notice that $H^0_\per(\mathcal{D}) \equiv L^2(\mathcal{D})$ since $C^\infty_\per(\mathcal{D})$ is dense in $L^2(\mathcal{D})$. Now, let  $u \in H^k_\per(\mathcal{D})$ and consider its Fourier series expansion
\begin{equation}
u(\bm{x}) = \sum_{\bm{r} \in \ZZ^\Dim} c_{\bm{r}} e^{2\pi\iunit \bm{r} \cdot \bm{x}}, 
\quad \text{with }
c_{\bm{r}} := \int_{\mathcal{D}}u(\bm{x}) e^{-2\pi \iunit \bm{r} \cdot \bm{x}} \de \bm{x}.
\end{equation}
Then, the following norm equivalence holds: 
\begin{equation}
\label{eq:equivHmFourier}
\|u\|_{H^k(\mathcal{D})}^2 \sim \sum_{\bm{r} \in \ZZ^\Dim}  (1+|\bm{r}|^{2k}) |c_\bm{r}|^2, \quad \forall u \in H^k_{\per}(\mathcal{D}), \quad \forall k \in \mathbb{N}_0. 
\end{equation}
Moreover, we define $H^{-1}_\per(\mathcal{D}) := [H^1_\per(\mathcal{D})]^*$, where the superscript $^*$ denotes the dual space. For the proofs of these results and for more details about periodic Sobolev spaces, we refer the reader to  \cite{Adams2003, Taylor2011, Temam1995}. 

\subsection{Advection-diffusion-reaction problems}
\label{sec:ADR}
Consider the sesquilinear form $a : H^1_\per(\mathcal{D}) \times H^1_\per(\mathcal{D}) \to \mathbb{C}$ defined as
\begin{equation}
a(u,v) 
:= \int_{\mathcal{D}} 
\left[\eta(\bm{x}) \nabla u(\bm{x}) \cdot \overline{\nabla v(\bm{x})}
+ \bm{\beta}(\bm{x})\cdot \nabla u(\bm{x}) \overline{v(\bm{x})} 
+ \rho(\bm{x}) u(\bm{x}) \overline{v(\bm{x})}\right] \de \bm{x},
\end{equation}
where $\dif, \rea : \mathcal{D} \to \mathbb{R}$ are the diffusion and reaction coefficients, respectively,  and $\bm{\beta} : \mathcal{D} \to \mathbb{R}^\Dim$ is the advective field.
Then, the weak formulation of the periodic ADR equation reads
\begin{equation}
\label{eq:weakADR}
\text{find } u \in H^1_\per(\mathcal{D}): \quad 
a(u,v) = (f,v), 
\quad \forall v \in H^1_\per(\mathcal{D}),
\end{equation} 
where $f : \mathcal{D} \to \mathbb{R}$ is a forcing term. Although we are interested in the case of real-valued coefficients $\eta$, $\rho$, $\bm{\beta}$, and $f$, the bilinear form $a(\cdot,\cdot)$ is defined over complex-valued Hilbert spaces to allow us the use of the Fourier basis in Section~\ref{sec:CORSING}.

We recall that the coercivity constant of $a(\cdot, \cdot)$ is the largest $\alpha > 0$ such that
\begin{equation}
|a(u,u)| \geq \alpha \|u\|_{H^1(\mathcal{D})}^2, \quad \forall u \in H^1_\per(\mathcal{D}),
\end{equation}
and the continuity constant of $a(\cdot,\cdot)$ is the smallest constant $\mathcal{A} >0 $ such that
\begin{equation}
|a(u,v)| \leq \mathcal{A} \|u\|_{H^1(\mathcal{D})} \|v\|_{H^1(\mathcal{D})}, \quad \forall u,v \in H^1_\per(\mathcal{D}).
\end{equation}
In the following proposition, we provide conditions on the coefficients $\eta,\bm{\beta},\rho$ sufficient to ensure the well-posedness of problem \eqref{eq:weakADR}. A proof of this result is provided in \cite{brugiapaglia2018wavelet_supplementary}.
\begin{prop}[Well-posedness of the periodic ADR problem]
\label{prop:ADRwellposed}
Let $\eta,\rho \in L^\infty(\mathcal{D})$ and $\bm{\beta} \in [L^\infty(\mathcal{D})]^\Dim$ be such that $\bm{\beta}$ is one-periodic with respect to each variable. Moreover, assume that there exist two constants $\eta_{\min},\zeta>0$ such that 
$$
\eta \geq \eta_{\min},
\quad 
-\frac12 \nabla \cdot \bm{\beta} + \rho \geq \zeta,  \quad  \text{a.e.\ in } \mathcal{D}.
$$
Then, for every $f \in H^{-1}_\per(\mathcal{D})$, the weak problem \eqref{eq:weakADR} admits a unique solution $u$ such that  
\begin{equation}
\label{eq:aprioriest}
\|u\|_{H^1(\mathcal{D})} \leq \frac{1}{\alpha} \|f\|_{H^{-1}(\mathcal{D})}, 
\end{equation}
with $\alpha \geq \min\{\eta_{\min},\zeta\}$ the coercivity constant of $a(\cdot,\cdot)$. Moreover, the continuity constant of $a(\cdot,\cdot)$ satisfies
$$
\mathcal{A} 
\leq 
\max\left\{\|\eta\|_{L^\infty(\mathcal{D})}, 
\sup_{\bm{x}\in \mathcal{D}}\|\bm{\beta}(\bm{x})\|_2, 
\|\rho\|_{L^\infty(\mathcal{D})}\right\}.
$$
\end{prop}

\begin{rmrk}[Diffusion equation] 
Notice that Proposition~\ref{prop:ADRwellposed} does not encompass the purely diffusive case, i.e.\ $\bm{\beta} \equiv \bm{0}$ and $\rho \equiv 0$. Indeed, in this case we only have uniqueness of the solution up to constants. This issue can be fixed by  assuming  $u,v \in H^1_\per(\mathcal{D}) / \mathbb{R}$ in the weak formulation \eqref{eq:weakADR}.  \hfill $\blacksquare$
\end{rmrk}

\section{\corsing Wavelet-Fourier}
\label{sec:CORSING}

In order to solve problem \eqref{eq:weakADR}, we describe the \corsing \WF (Wavelet-Fourier) method. %This is a reduced Petrov-Galerkin method associated with a wavelets trial basis and a Fourier test basis. The former is introduced in Section~\ref{sec:wavelets} and the latter in Section~\ref{sec:Fourier}. After defining the trial and the test bases, we describe the \corsing \WF method in Section~\ref{sec:CORSINGWF}. Finally, in Section~\ref{sec:i,ii,iii} we  reformulate issues (i), (ii), and (iii) mentioned in Section~\ref{sec:main_contr} in a rigorous fashion, laying the foundations to solve them in Section~\ref{sec:localcoherence}.

%%%%%%%%%%%%%%%%%%%%%%%%%%%%%%%%%%%%%%%%%%%%%%%%%%%%%%%%%%%
\subsection{Trial functions: wavelets}
%%%%%%%%%%%%%%%%%%%%%%%%%%%%%%%%%%%%%%%%%%%%%%%%%%%%%%%%%%%
\label{sec:wavelets}

In this section, we present the biorthogonal wavelets on the periodic interval and on the periodic hypercube that will be employed as trial functions of the Petrov-Galerkin discretization.  We  refer to \cite{brugiapaglia2018wavelet_supplementary} and to \cite{Dahmen1997,Pabel2015,Urban2008} for technical details about wavelet construction. A crucial property of the biorthogonal wavelets is that, after suitable normalization, they are a Riesz basis for $H^1_\per(\mathcal{D})$, thanks to the so-called norm equivalence property (see Theorem~\ref{thm:normequiv}).

\paragraph{Biorthogonal wavelets on the periodic interval.} 
Given $\ell_0, L \in \mathbb{N}_0$ with $\ell_0 < L$, we consider biorthogonal B-spline wavelet basis on the real line 
$$
\Psi := \Phi_{\ell_0} \cup \bigcup_{\ell = \ell_0}^{L-1} \Psi_\ell,
$$
where $\Phi_\ell = \{\varphi_{\ell,k}\}_{k \in \mathbb{Z}}$ are scaling functions and $\Psi_\ell = \{\psi_{\ell,k}\}_{k \in \mathbb{Z}}$ are wavelet functions (for more details, see Figure~\ref{Figure1}, Setting~\ref{ass:Bsplwave} and \cite{brugiapaglia2018wavelet_supplementary}). 
\begin{figure}[t]
\centering
\includegraphics[height = 5cm]{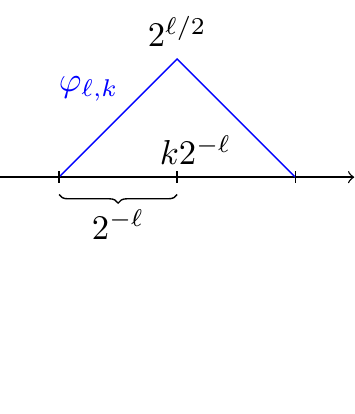}
\hspace{2cm}
\includegraphics[height = 5cm]{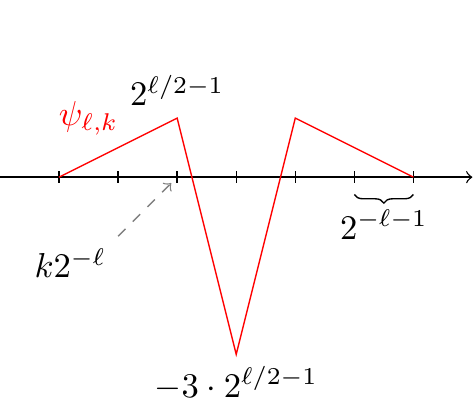}
\caption{\label{Figure1}The translated and rescaled scaling functions $\varphi_{\ell,k}$ (left) and  wavelets $\psi_{\ell,k}$ (right) corresponding to the construction of biorthogonal B-spline wavelets of order $(d, \widetilde d) = (2,2)$ on the real line (Setting~\ref{ass:Bsplwave}).}
\end{figure}
Notice that the dependence of $\Psi$ on the levels $\ell_0$ and $L$ is understood. For $L = \infty$, $\Psi$ is a basis for $L^2(\mathbb{R})$. In order to build a basis for $L^2(\mathcal{D})$, we resort to periodization (see \cite{brugiapaglia2018wavelet_supplementary}). Similarly, we denote by
$$
\Psi^{\per} := \Phi_{\ell_0}^{\per} \cup \bigcup_{\ell = \ell_0}^{L-1} \Psi_\ell^{\per},
$$
where $\Phi_\ell^{\per} = \{\varphi_{\ell,k}^{\per}\}_{k \in \mathbb{Z}/(2^\ell\mathbb{Z})}$ are periodized scaling functions and $\Psi_\ell^{\per} = \{\psi_{\ell,k}^\per\}_{k \in \mathbb{Z}/(2^\ell\mathbb{Z})}$ are periodized wavelet functions. In particular, $\ZZ/(2^\ell\ZZ)$ denotes the ring of integers modulo $2^\ell$ that coincides with the set of canonical representatives, i.e.,
\begin{equation}
\mathbb{Z}/(2^\ell\mathbb{Z}) \equiv \{0,1, \ldots,2^\ell - 1\}, \quad \forall \ell \in \mathbb{N}_{0}.
\end{equation}
Assuming the coarsest level $\ell_0 \in \mathbb{N}_0$ to be fixed, we also introduce the following notation:
\begin{equation}
\label{eq:notationl0-1}
\psi_{\ell_0-1,k} \equiv \varphi_{\ell_0,k}, \quad \forall k \in \ZZ.
\end{equation}

\begin{setting}[1D biorthogonal B-spline wavelets \cite{Cohen1992}]
\label{ass:Bsplwave}
We consider 1D biorthogonal B-spline wavelets of order $(d,\widetilde d) = (2,2)$, corresponding to primal and dual filters
\begin{align}
&\bm{a}_{[-1:1]} = [\tfrac12,1,\tfrac12], \qquad \qquad \, \,
\widetilde{\bm{a}}_{[-2:2]} = [-\tfrac14,\tfrac12,\tfrac32,\tfrac12,-\tfrac14],\\
&{\bm{b}}_{[-1:3]} = [\tfrac14,\tfrac12,-\tfrac32,\tfrac12,\tfrac14],
\quad \;\,
\widetilde{\bm{b}}_{[0:2]} = [\tfrac12,-1,\tfrac12].
\end{align} 
Moreover, we assume scaling functions $\Phi_\ell^{\text{per}}$ and wavelets $\Psi_\ell^{\text{per}}$ to be normalized with respect to the $L^2(\mathcal{D})$-norm (i.e., $\|\varphi_{\ell,k}\|_{L^2(\mathcal{D})} \sim \|\psi_{\ell,k}\|_{L^2(\mathcal{D})} \sim 1$, for every $\ell \in \mathbb{N}_0$ and $k \in \mathbb{Z}/(2^\ell \mathbb{Z})$) and $\ell_0 \geq 2$.
\end{setting}
When wavelets are built as in Setting~\ref{ass:Bsplwave}, the corresponding $\Psi^\per$ is a basis for $H^1_\per(\mathcal{D})$ when $L = \infty$. In fact, for $L \in \mathbb{N}$, $\Span(\Psi^\per)$ coincides with the space of functions in $H^1_\per(\mathcal{D})$ that are continuous and piecewise linear on the uniform grid $2^{-L}\ZZ \cap [0,1] $. Moreover, choosing $\ell_0 \geq 2$ guarantees that\footnote{In particular, the assumption $\ell_0 \geq 2$ ensures that the periodization $\psi_{\ell,k}^\per$ of $\psi_{\ell,k}$ (defined as $\psi_{\ell,k}^{\per}(x) = 2^{\ell/2} \sum_{j \in \ZZ} \psi (2^{\ell}(x + j) - k)$, for every  $x \in \mathbb{R}$ -- see also \cite{brugiapaglia2018wavelet_supplementary}) is the sum of terms with disjoint support. On the contrary, let us assume, for example, $\ell_0 = 1$. Then, since $\supp(\psi_{1,1}) \subseteq (0,3/2)$ the periodization $\psi_{1,1}^\per$ is built by summing two overlapping terms on $[0,1/2]$. In particular, $\psi_{1,1}^\per$ is constant over $[0,1/2]$, whereas $\psi_{1,1}$ is not. This shows that condition \eqref{eq:nooverlap_cond} is  not satisfied for $\ell_0 = 1$.}
\begin{equation}
\label{eq:nooverlap_cond}
\psi_{\ell,k}^\per|_{\supp(\psi_{\ell,k})} \equiv \psi_{\ell,k}|_{\supp(\psi_{\ell,k})}, \quad \forall \ell \geq \ell_0 - 1, \; \forall k \in \mathbb{Z}/(2^\ell\mathbb{Z}).
\end{equation}

For the sake of simplicity, the apex $^\per$ will be omitted in the subsequent developments. In fact, from now on we will always assume to be in the periodic setting. \\

In order to generalize this construction to arbitrary dimension $n > 1$, we consider  \emph{anisotropic} and \emph{isotropic} tensor product wavelets (see Figure~\ref{fig:iso_vs_ani}). Here we just recall the basic definitions and we refer the reader to \cite{Dahmen1997, Pabel2015,Urban2008} for a more detailed discussion. 
\begin{figure}[t]
\centering
\raisebox{1.5cm}{%
\includegraphics[width = 6cm]{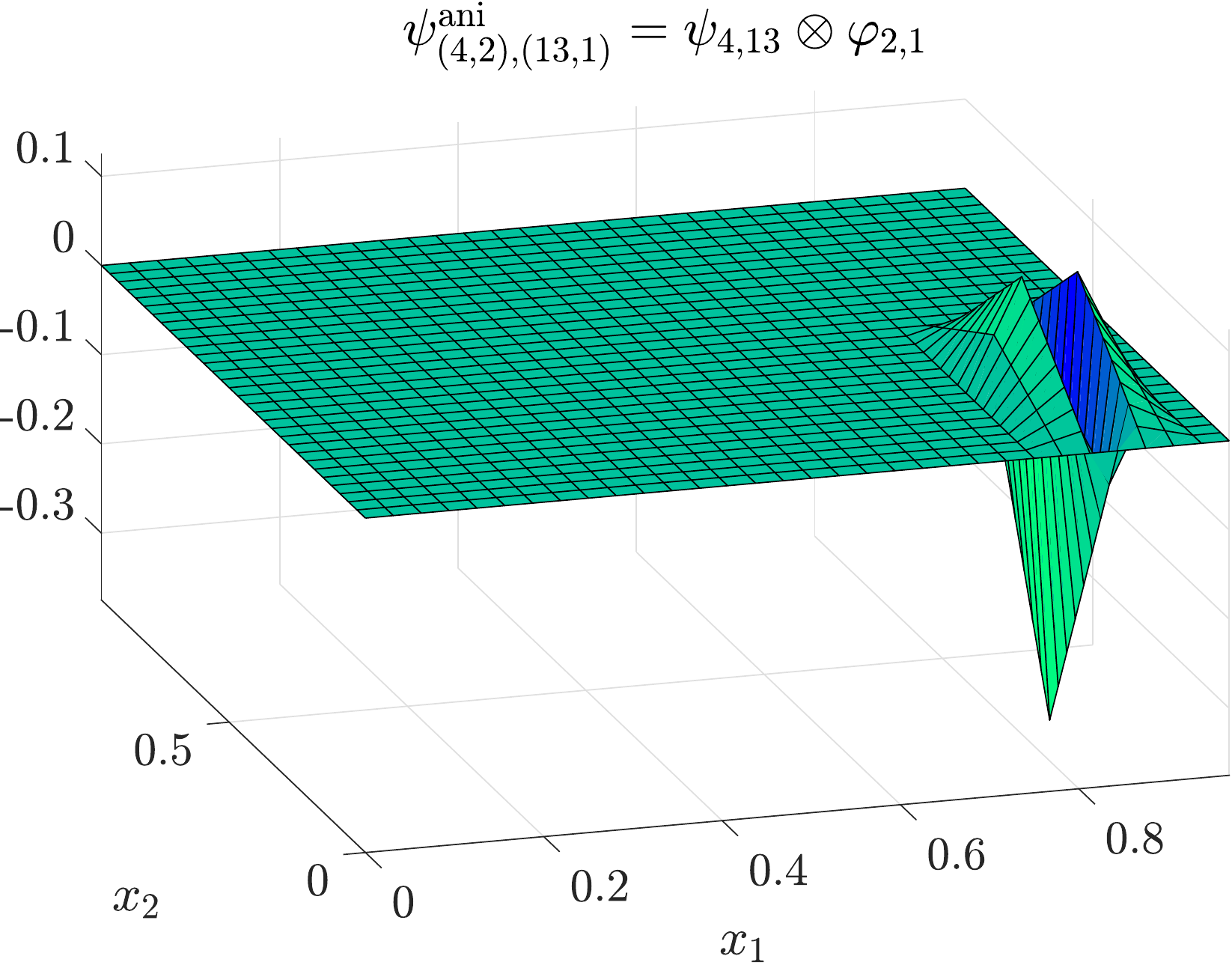}
}
\hspace{1cm}
\includegraphics[width = 6cm]{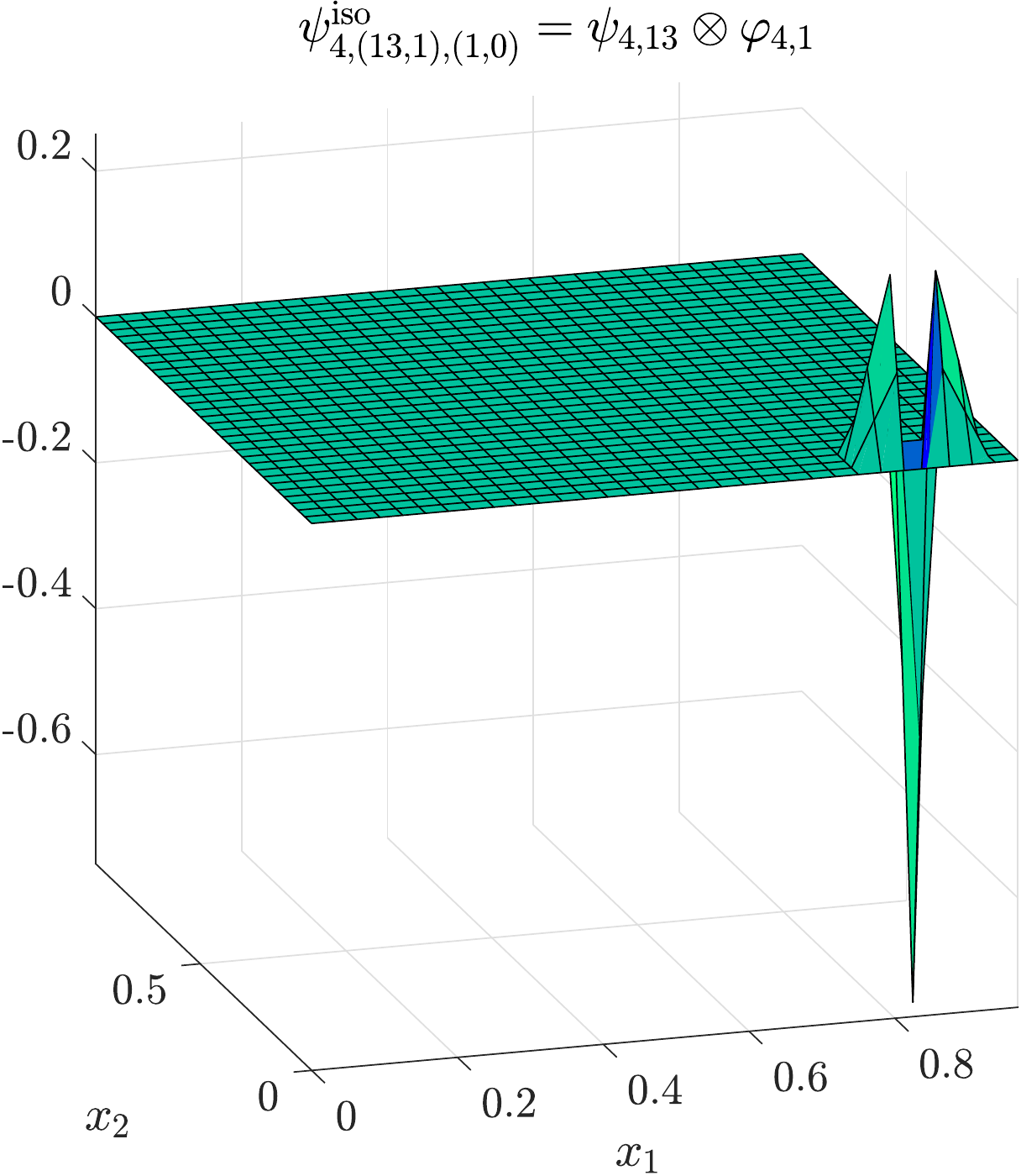}
\caption{\label{fig:iso_vs_ani}Surface plot of the 2D anisotropic (left) and isotropic (right) tensor product wavelets.}
\end{figure}

\paragraph{Anisotropic tensor product wavelets.}

Given a $\ell_0 \in \mathbb{N}$ and  multi-indices $\bm{\ell}\in \mathbb{N}^\Dim$, with $\bm{\ell} \geq \ell_0-1$ and $\bm{k} \in \mathbb{Z}^\Dim$, a first straightforward way to obtain a multi-dimensional basis is by tensorizing the 1D wavelet basis $\Psi$ with itself $\Dim$ times, namely,
\begin{equation}
\psi_{\bm{\ell},\bm{k}}^{\ani} 
:= \psi_{\ell_1,k_1} \otimes \cdots \otimes \psi_{\ell_\Dim,k_\Dim},
\end{equation}
with $\psi_{\ell,k} = \psi_{\ell,k}^\per$. For every level multi-index $\bm{\ell} \in \mathbb{N}^\Dim$, the spatial multi-index $\bm{k}$ takes values in $\mathbb{Z}/(2^\bm{\ell}\mathbb{Z})$, where
$$
\mathbb{Z}/(2^\bm{\ell}\mathbb{Z}):= \mathbb{Z}/(2^{\ell_1}\mathbb{Z}) \times \cdots \times \mathbb{Z}/(2^{\ell_\Dim}\mathbb{Z}) 
\equiv
\prod_{j = 1}^\Dim \{0,1,\ldots,2^{\ell_j-1}\}.
$$ 
Therefore, fixing $L \in \mathbb{N}$ with $L > \ell_0$ and defining the multi-index set 
\begin{equation}
\label{eq:defJPsiani}
\mathcal{J}^\ani := \{(\bm{\ell},\bm{k}) : \bm{\ell} \in \mathbb{N}^\Dim, \, \ell_0-1 \leq \bm{\ell} < L, \, \bm{k} \in \mathbb{Z}/(2^\bm{\ell}\mathbb{Z})\},
\end{equation}
we have
\begin{equation}
\label{eq:defaniwave_short}
\Psi^\ani = \{\psi_\bm{j}^\ani\}_{\bm{j} \in \mathcal{J}^\ani},
\end{equation}
where we set $\bm{j} = (\bm{\ell},\bm{k})$.

\paragraph{Isotropic tensor product wavelets.}

An alternative (and less straightforward) way to define a multi-dimensional wavelet basis is by isotropic tensorization. In this case, we need the following auxiliary notation:
\begin{equation}
\vartheta_{\ell,k,e}:= 
\begin{cases}
\varphi_{\ell,k}, & \text{if } e = 0,\\
\psi_{\ell,k}, & \text{if } e = 1.
\end{cases}
\end{equation}
Then, given $\ell \in\mathbb{N}$ with $\ell \geq \ell_0$, $\bm{k} \in \mathbb{Z}^\Dim$, and $\bm{e} \in \{0,1\}^\Dim$, we define
\begin{equation}
\psi_{\ell,\bm{k},\bm{e}}^\iso
:= \vartheta_{\ell,k_1,e_1} \otimes \cdots \otimes\vartheta_{\ell,k_\Dim,e_\Dim}.
\end{equation}
Then, fixing $L \in \mathbb{N}$ with $L > \ell_0$ and defining the multi-index set
\begin{equation}
\label{eq:defJPsiiso}
\mathcal{J}^\iso := \{(\ell,\bm{k},\bm{e}) : \ell \in \mathbb{N}, \ell_0 \leq \ell < L, \bm{k} \in (\mathbb{Z}/(2^\ell\mathbb{Z}))^\Dim, \bm{e} \in \{0,1\}^\Dim\},
\end{equation}
we have 
\begin{equation}
\label{eq:defisowave_short}
\Psi^\iso := \{\psi_\bm{j}^\iso\}_{\bm{j} \in \mathcal{J}^\iso},
\end{equation}
where we set $\bm{j} = (\ell,\bm{k},\bm{e})$.

The main difference between anisotropic and isotropic tensor product structure is visualized in Figures~\ref{fig:iso_vs_ani} and \ref{Figure2} (in the 2D case).  
\begin{figure}[t]
\centering
\begin{tabular}{cc}
Anisotropic tensorization & Isotropic tensorization\\
\includegraphics[height = 6cm]{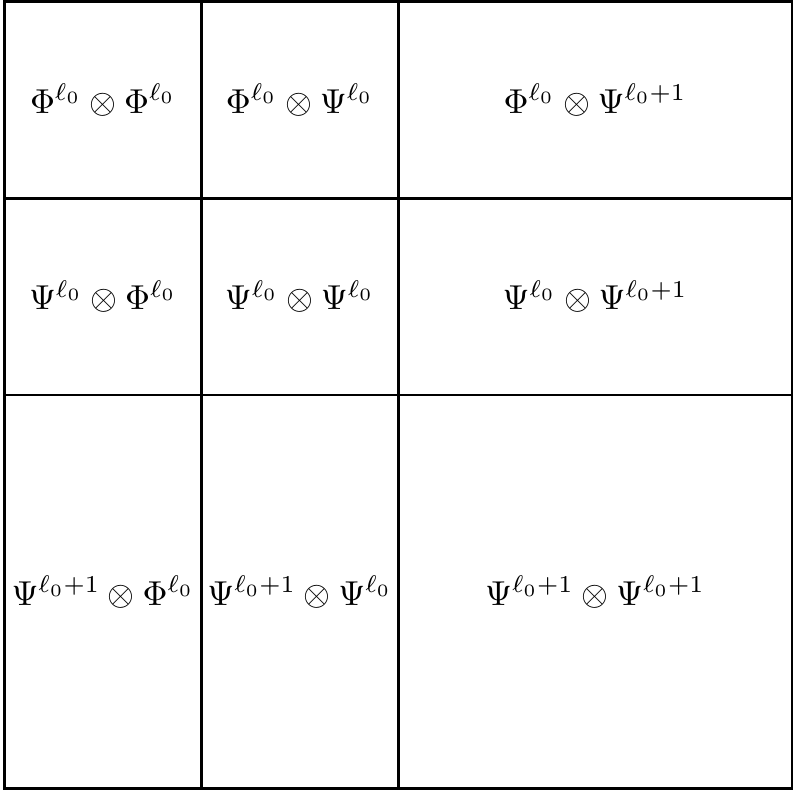} & 
\includegraphics[height = 6cm]{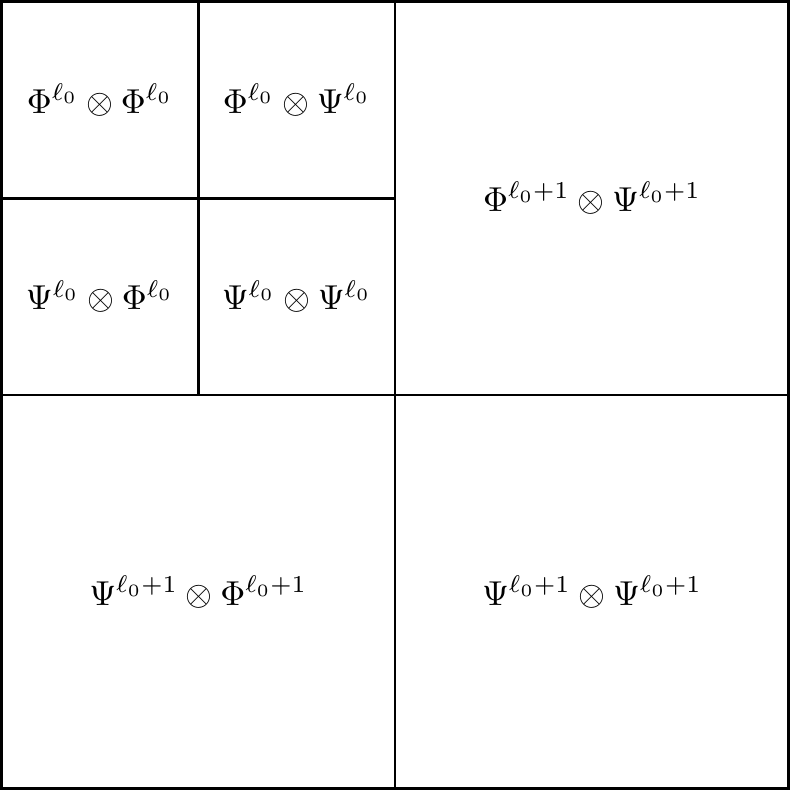}
\end{tabular}
\caption{\label{Figure2}Visualization of the 2D tensorized anisotropic (left) and isotropic (right) wavelets for $L = \ell_0 +1$.}
\end{figure}
In particular, the anisotropic tensorization blends all the dyadic levels together, whereas the isotropic tensorization only combines basis functions of the same level.  

\paragraph{Riesz basis property.}

The biorthogonal wavelets employed here satisfy the norm equivalence property, presented in the following theorem. We refer to \cite{brugiapaglia2018wavelet_supplementary} for the proof. 
\begin{thm}[Norm equivalence]
\label{thm:normequiv}
Let $\Psi \in \{\Psi^\ani,\Psi^\iso\}$ be the tensor product wavelet basis defined as in \eqref{eq:defaniwave_short} or \eqref{eq:defisowave_short} from periodized 1D biorthogonal B-spline wavelets of order $(d,\widetilde d)$. Then, the following norm equivalence holds:
\begin{equation}
\label{eq:normequiv}
\bigg\|\sum_{\bm{j} \in \mathcal{J}^\star}
c_{\bm{j}} \psi_{\bm{j}}^\star
\bigg\|_{H^s(\mathcal{D})}
\sim
\|D_{\star}^s \bm{c}\|_2,
\quad 
\forall \star \in\{\ani,\iso\},
\end{equation}
for every $s \in [0, d-1]$, where $\bm{c} = (\bm{c})_{\bm{\ell},\bm{k}}$ and $D_\star \in \{D_\ani,D_\iso\}$ is the diagonal matrix defined by 
$$
(D_\star)_{\bm{j},\bm{j}'} = 
\begin{cases}
2^{\|\bm{\ell}\|_{\infty}} \delta_{\bm{j},\bm{j}'},  & \text{if } \star = \ani,\\
2^{\ell} \delta_{\bm{j},\bm{j}'},  & \text{if } \star = \iso,
\end{cases}
\quad \forall \bm{j},\bm{j}' \in \mathcal{J}^\star,
$$
where $\delta_{\bm{j},\bm{j}'}$ is Kronecker's delta function and where $\mathcal{J}^\ani$ and $\mathcal{J}^\iso$ are defined as in \eqref{eq:defJPsiani} and \eqref{eq:defJPsiiso}, respectively. 
\end{thm}
In Setting~\ref{ass:Bsplwave}, Theorem~\ref{thm:normequiv} implies that a suitable weighted $\ell^2$-norm of the wavelet coefficients is equivalent to the $H^1(\mathcal{D})$-norm. In particular, using the fact that
\begin{align}
\label{eq:normpsij_ani_1}
\|\psi_{\bm{j}}^\ani\|_{H^1(\mathcal{D})} & \sim 2^{\|\bm{\ell}\|_\infty}, \quad \forall \bm{j} \in \mathcal{J}^\ani,\\
\label{eq:normpsij_iso_1}
\|\psi_{\bm{j}}^\iso\|_{H^1(\mathcal{D})} & \sim 2^{\ell}, \quad \forall \bm{j} \in \mathcal{J}^\iso.
\end{align}
we can rescale the basis functions of $\Psi^{\ani}$ and $\Psi^{\iso}$ and normalize them with respect to the $H^1(\mathcal{D})$-norm and obtain two Riesz bases with respect to the $H^1(\mathcal{D})$-norm. More details on the norm equivalence property are provided in \cite{brugiapaglia2018wavelet_supplementary}. %Moreover, \eqref{eq:normpsij_ani_1} and \eqref{eq:normpsij_iso_1} are a direct consequence of Theorem~\ref{thm:normequiv}.

%%%%%%%%%%%%%%%%%%%%%%%%%%%%%%%%%%%%%%%%%%%%%%%%%%%%%%%%%%%
\subsection{Test functions: Fourier basis}
%%%%%%%%%%%%%%%%%%%%%%%%%%%%%%%%%%%%%%%%%%%%%%%%%%%%%%%%%%%
\label{sec:Fourier}
Consider the one-periodic Fourier basis functions in dimension one and their tensorized version in the $\Dim$-dimensional case
\begin{equation}
\xi_q(x) := \exp( 2 \pi \iunit q x), \quad \forall q \in \ZZ, \quad 
\xi_{\bm{q}} := \xi_{q_1} \otimes \cdots \otimes \xi_{q_\Dim}, \quad \forall \bm{q} \in \ZZ^\Dim.
\end{equation}
It is easy to verify that $\{\xi_\bm{q} : \bm{q} \in \ZZ^\Dim\}$  is an orthonormal basis for  $L^2(\mathcal{D})$  and that its elements are orthogonal with respect to the $H^1(\mathcal{D})$-inner product. Moreover, we have
\begin{equation}
\label{eq:normxiqmulti-d}
\|\xi_{\bm{q}}\|_{L^2(\mathcal{D})}^2 = 1, 
\quad 
\|\xi_\bm{q}\|_{H^1(\mathcal{D})}^2 = 1 + (2 \pi \|\bm{q}\|_2)^2, \quad \forall \bm{q} \in \ZZ^\Dim.
\end{equation}
Given $R \in \mathbb{N}$, let us consider the following finite multi-index set:
\begin{equation}
\label{eq:defQcal}
\Qcal:=\{\bm{q} \in \ZZ^\Dim : -\lfloor R/2\rfloor + 1 \leq q_i \leq \lfloor R/2\rfloor, \; \forall i \in [n] \}.
\end{equation}
Then, we define the Fourier basis as 
\begin{equation}
\label{eq:defXi}
\Xi := \{\xi_{\bm{q}}\}_{\bm{q} \in \mathcal{Q}}.
\end{equation}
In particular, $\Xi$ is a Riesz basis with respect to the $H^1(\mathcal{D})$-norm.

\subsection{The \corsing Wavelet-Fourier method}
\label{sec:CORSINGWF}

We are now in a position to introduce the \corsing \WF (Wavelet-Fourier) method.

\paragraph{Normalization with respect to the $H^1(\mathcal{D})$-norm.} Let $\Psi = \{\psi_{\bm{j}}\}_{\bm{j} \in \mathcal{J}}$ be a tensor product of periodized biorthogonal B-spline wavelets (or simply a family of  periodized biorthogonal B-spline wavelets for $\Dim = 1$)  defined as in Section~\ref{sec:wavelets}. In particular, $\Psi \in \{\Psi^\ani, \Psi^\iso\}$ and $\mathcal{J} \in \{\mathcal{J}^\ani,\mathcal{J}^\iso\}$. Let $\Xi = \{\xi_{\bm{q}}\}_{\bm{q} \in \mathcal{Q}}$ be the Fourier basis in Section~\ref{sec:Fourier}. Then, we normalize both trial and test functions with respect to the  $H^1(\mathcal{D})$-norm, namely,
$$
\widehat\Psi = \{\widehat\psi_{\bm{j}}\}_{\bm{j} \in \mathcal{J}}, \quad
\widehat\Xi = \{\widehat\xi_{\bm{q}}\}_{\bm{q} \in \mathcal{Q}},
$$ 
such that 
$$
\|\widehat \psi_\bm{j} \|_{H^1(\mathcal{D})} \sim 1, \quad \|\widehat \xi_\bm{q} \|_{H^1(\mathcal{D})} = 1, \quad \forall \bm{j} \in \mathcal{J}, \; \forall \bm{q} \in \mathcal{Q}.
$$
In view of \eqref{eq:normpsij_ani_1}, \eqref{eq:normpsij_iso_1}, and \eqref{eq:normxiqmulti-d},  this normalization can be realized by defining
$$
\widehat\psi_{\bm{j}}^\ani:= 2^{-\|\bm{\ell}\|_\infty} \psi_\bm{j}^\ani,
\quad
\widehat\psi_{\bm{j}}^\iso:= 2^{-\ell} \psi_\bm{j}^\iso,
\quad 
\widehat\xi_\bm{q} := (1+(2\pi\|\bm{q}\|_2)^2)^{-\frac12}\xi_\bm{q}.
$$ 

\paragraph{Petrov-Galerkin discretization.} We consider a  Petrov-Galerkin discretization of \eqref{eq:weakADR} associated with the trial basis $\widehat\Psi$ and test basis $\widehat\Xi$ (see \cite{QuarteroniValli} for an introduction to the Petrov-Galerkin method). The stiffness matrix $B \in \mathbb{C}^{M \times N}$, with $M, N \in \mathbb{N}$ and $M \geq N$, and the load vector $\bm{g} \in \mathbb{C}^M$ are  defined as
\begin{equation}
B_{\bm{q},\bm{j}} := a(\widehat\psi_{\bm{j}}, \widehat\xi_{\bm{q}}), \quad
g_\bm{q} := (f,\widehat\xi_{\bm{q}}),
\quad \forall \bm{j} \in \mathcal{J}, \; \forall\bm{q} \in \Qcal,
\end{equation}
where 
\begin{equation}
N:= |\mathcal{J}|= 2^{\Dim L},\quad   M:= |\Qcal|=R^\Dim,
\end{equation} 
and the resulting Petrov-Galerkin discretization is given by the linear system 
\begin{equation}
\label{eq:PGsystem}
B \bm{v} = \bm{g},
\end{equation}
with $\bm{v} \in \mathbb{C}^N$ the vector of the unknowns.

\paragraph{The \corsing \WF method.} The next step is to  reduce the dimensionality of the Petrov-Galerkin discretization \eqref{eq:PGsystem} by random subsampling. Given a probability distribution $\bm{p} \in \mathbb{R}^M$ over $\mathcal{Q}$, we draw $m \ll M$ multi-indices  $\bm{\tau}_1,\ldots,\bm{\tau}_m \in \mathcal{Q}$ i.i.d.\  randomly according to 
$$
\Prob\{\bm{\tau}_i = \bm{q}\} = p_\bm{q}, \quad \forall \bm{q} \in \mathcal{Q}, \; \forall i \in [m].
$$
Then, we define the \corsing stiffness matrix $A \in \mathbb{C}^{m \times N}$ and load vector $\bm{f} \in \mathbb{C}^m$ as
\begin{equation}
A_{i,\bm{j}} := a(\widehat\psi_{\bm{j}}, \widehat\xi_{\bm{\tau}_i}), 
\quad 
f_i  := (f,\widehat\xi_{\bm{\tau}_i}),
\quad \bm{j} \in \mathcal{J}, \; i \in [m].
\end{equation} 
The \corsing reduced discretization corresponds to the underdetermined linear system
\begin{equation}
\label{eq:CORSINGsystem}
A \bm{z} =  \bm{f},
\end{equation}
with $\bm{z} \in \mathbb{C}^N$ the vector of the unknowns. Then, we consider the diagonal preconditioner $D \in \mathbb{C}^{m \times m}$, defined as\footnote{The diagonal  preconditioner $D$ is chosen in such a way that  $\Expe [(DA)^* D A] = B^* B$.  For further details, see \cite{Brugiapaglia2018}.} 
\begin{equation}
D_{i,k} = \delta_{i,k} / \sqrt{m p_{\bm{\tau}_i}}, \quad i,k \in [m],
\end{equation}
where $\delta_{i,k}$ is Kronecker's delta function. Given a target sparsity level $s \in \mathbb{N}$, with $s \ll N$, we consider the following optimization problem:
\begin{equation}
\label{eq:P0s}
\min_{\bm{z} \in \mathbb{C}^N} \|D (A \bm{z} -  \bm{f})\|_2, \quad \text{s.t. } \|\bm{z}\|_0 \leq s.
\end{equation} 
Although NP-hard, \eqref{eq:P0s} can be approximated by the orthogonal matching pursuit (OMP) algorithm. Finally, the \corsing solution is given by
\begin{equation}
\label{eq:uhat}
\widehat u := \sum_{\bm{j} \in \mathcal{J}} \widehat u_\bm{j} \widehat{\psi}_{\bm{j}},
\end{equation}
where $\widehat{\bm{u}} = (\widehat u_{\bm{j}})_{\bm{j} \in \mathcal{J}} \in \mathbb{C}^N$ is an approximate solution to \eqref{eq:P0s} computed via OMP.

%A straightforward implementation of this procedure %has a computational cost $O(smN)$ because it 
%corresponds to running $O(s)$ iterations of OMP with the matrix $DA  \in \mathbb{C}^{m \times N}$. 
Note that the sampling complexity $m$ can be reduced in order to avoid repeated indices among the $\tau_i$'s. In that case, the preconditioner $D$ has to be slightly modified (see \cite[Remark 3.9]{Brugiapaglia2018}).

\subsection{Towards a recovery error analysis}
\label{sec:i,ii,iii}
We are now able to restate issues (i), (ii), and (iii) in Section~\ref{sec:main_contr} in a rigorous way. Assuming to fix $s,N \in \mathbb{N}$ as parameters chosen by the user, with $s \ll N$, the \corsing procedure outlined above depends on the following three choices:
\begin{enumerate}
\item [(i)] Choose $R = R(s,N)$, defining the size of $\mathcal{Q}$ and, consequently, the test space dimension  $M = R^\Dim$ of the Petrov-Galerkin discretization \eqref{eq:PGsystem}.
\item [(ii)] Choose the number $m = m(s,N,M)$ of random samples, depending sublinearly on $N$ and $M$ (in order to have dimensionality reduction from \eqref{eq:PGsystem} to \eqref{eq:CORSINGsystem});
\item [(iii)] Define the probability distribution $\bm{p} \in \mathbb{R}^M$ on the test multi-index space $\mathcal{Q}$.
\end{enumerate}

We observe that the choice of $s$ and $N$ theoretically depends on the best $s$-term approximation error decay rate for the target function class (containing  $u$). This, in turn, could be studied using results from nonlinear approximation theory (see, e.g., \cite{devore1998nonlinear} or \cite[Chapter 9]{mallat1999wavelet}).

In order to solve issues (i), (ii), and (iii), we resort to the theoretical analysis carried out in \cite{Brugiapaglia2018}, based on the following 
\begin{defn}[Local $a$-coherence]
\label{def:mu}
The local $a$-coherence of $\widehat\Psi$ with respect to $\widehat \Xi$ is a sequence $\bm{\mu} \in \ell(\mathbb{Z}^n)$ defined by
\begin{equation}
\label{eq:defmu}
\mu_\bm{q} := \sup_{\bm{j} \in \mathcal{J}}|a(\widehat\psi_{\bm{j}},\widehat\xi_{\bm{q}})|^2, \quad \forall \bm{q} \in \mathbb{Z}^n.
\end{equation}
\end{defn}
Since the \corsing solution $\widehat u$ is $s$-sparse with respect to $\Psi$ (or, equivalently, to $\widehat\Psi$), the corresponding best possible accuracy is the best $s$-term approximation error
\begin{equation}
\sigma_{s}(u)_{H^1(\mathcal{D})} 
:= \inf \bigg\{\|u - w\|_{H^1(\mathcal{D})} : w = \sum_{\bm{j} \in \mathcal{J}} c_\bm{j} \psi_\bm{j}, \; \|\bm{c}\|_0 \leq s\bigg\}.
\end{equation}
We specialize \cite[Theorem 3.15]{Brugiapaglia2018} to the \corsing \WF setting, providing a recovery error estimate in expectation. The fact that $\widehat\Psi$ and $\widehat\Xi$ are Riesz bases with respect to the $H^1(\mathcal{D})$-norm (guaranteed by Theorem~\ref{thm:normequiv}) is required to apply this result, and this justifies the use of biorthogonal wavelets when $n > 1$. Indeed, tensorizing the hierachical basis of hat functions as in \cite{PhDThesis,Brugiapaglia2015, Brugiapaglia2018} does not give the Riesz basis assumption in the multi-dimensional case. It is also possible to state an analogous result in probability instead of expectation (see \cite[Theorems 3.16 and 3.18]{Brugiapaglia2018}). 
\begin{thm}[\corsing \WF recovery in expectation]
\label{thm:CORSINGrecovery}
Let $\Dim, s,L \in \mathbb{N}$, with $s \ll N = 2^{\Dim L}$,  $K > 0$ be such that $\|u\|_{H^1(\mathcal{D})} \leq K$, where $u$ is the unique solution to \eqref{eq:weakADR}, and assume to have an upper bound to the local $a$-coherence $\bm{\mu} \in \ell(\mathbb{Z}^{n})$, i.e., there exists a sequence $\bm{\nu} \in \ell(\mathbb{Z}^{n})$ such that
\begin{equation}
\label{eq:mu<=nu}
\bm{\mu} \lesssim \bm{\nu}.
\end{equation} 
Choose $R \in \mathbb{N}$ (or, equivalently, $\mathcal{Q}$) such that the truncation condition\footnote{Let us clarify a small difference between \cite[Theorem 3.15]{Brugiapaglia2018} and Theorem~\ref{thm:CORSINGrecovery}. In \cite[Theorem 3.15]{Brugiapaglia2018}, $\bm{\nu}$ is an upper bound to $\bm{\mu}|_{\mathcal{Q}}$ and the truncation condition \eqref{eq:trunc_cond} involves $\bm{\mu}$ instead of $\bm{\nu}$. Therefore, the truncation condition of Theorem~\ref{thm:CORSINGrecovery} implies the truncation condition of \cite[Theorem 3.15]{Brugiapaglia2018}. However, this does not make any difference since in practice the truncation condition of \cite[Theorem 3.15]{Brugiapaglia2018} is verified using an upper bound to $\bm{\mu}$, and not $\bm{\mu}$ itself.}
\begin{equation}
\label{eq:trunc_cond}
 \|  \bm{\nu}|_{\mathcal{Q}^c} \|_1\lesssim \frac{1}{s},
\end{equation}
holds and where $\mathcal{Q}^c := \mathbb{Z}^n \setminus \mathcal{Q}$. Then, for every $0 < \varepsilon < 1$, the \corsing solution $\widehat u \in H^1_\per(\mathcal{D})$ exactly solving \eqref{eq:P0s} satisfies
\begin{equation}
\label{eq:corsingrecovery}
\Expe \left\|\min\left\{1, \frac{K}{\|\widehat u\|_{H^1(\mathcal{D})}}\right\} \widehat u - u\right\|_{H^1(\mathcal{D})}
\lesssim \frac{\mathcal{A}}{\alpha} \sigma_{s}(u)_{H^1(\mathcal{D})} + K \varepsilon,
\end{equation}
where $\alpha$ and $\mathcal{A}$ are the inf-sup and the continuity constants respectively associated with $a(\cdot, \cdot)$, provided that
\begin{equation}
\label{eq:m_rate}
m \gtrsim s \|\bm{\nu}|_{\mathcal{Q}}\|_1  (s\ln(eN/s) + \ln(2s/\varepsilon)),
\end{equation}
and that the drawings $\tau_1, \ldots, \tau_m \in \mathcal{Q}$ are i.i.d.\ according to the probability distribution 
\begin{equation}
\label{eq:defp}
\bm{p} = \frac{\bm{\nu}|_{\mathcal{Q}}}{\|\bm{\nu}|_{\mathcal{Q}}\|_1}.
\end{equation}
\end{thm}
Some considerations are in order:

\begin{itemize}
\item Relation \eqref{eq:m_rate} corresponds to a quadratic scaling of $m$ with respect to $s$ (up to logarithmic factors and up to the quantity $\|\bm{\nu}|_{\mathcal{Q}}\|_1$). In practice, a linear dependence of $m$ on $s$ (up to logarithmic factors) seems to be sufficient (see the  \cite[Section 5.4]{Brugiapaglia2018}). A different theoretical analysis carried out in \cite[Section 3.2.5]{PhDThesis} based on the concept of restricted isometry property seems to confirm this conjecture up to rescaling $DA$ by a suitable factor depending on $\mathcal{A}$ and on the true Riesz constant of $\Psi$, hidden in the norm equivalence \eqref{eq:normequiv} with $s = 1$. We have preferred to employ this slightly suboptimal result to avoid this technical rescaling issue. 
\item A necessary condition for \eqref{eq:trunc_cond} is $\|\bm{\nu}\|_1 < \infty$. This will always be the case in the applications discussed in this paper.
\item In order to have an actual compression of the Petrov-Galerkin discretization, $\|\bm{\nu}|_{\mathcal{Q}}\|_1$ has to depend sublinearly on $N$ and $M$ in \eqref{eq:m_rate}. 
\item Notice that knowing an upper bound $K$ to $\|u\|_{H^1(\mathcal{D})}$ is not a restrictive hypothesis in view of the \emph{a priori} estimate \eqref{eq:aprioriest}. 
\item The hypothesis that $\widehat{\bm{u}}$ solves \eqref{eq:P0s} exactly does not take into account the approximation error due to the OMP algorithm, which can be included by resorting to the restricted isometry property analysis.
\item Condition \eqref{eq:trunc_cond} is sufficient to guarantee the so-called \emph{$s$-sparse restricted inf-sup property} for the Petrov-Galerkin discretization \eqref{eq:PGsystem}. This is a variant of the classical inf-sup (or Ladyzhenskaya-Babu\v{s}ka-Brezzi) condition, adapted to the sparse case (see \cite[Section 3.3]{Brugiapaglia2018} for more details).
\end{itemize}
Starting from Theorem~\ref{thm:CORSINGrecovery}, we can tackle issues (i), (ii), and (iii) in Section~\ref{sec:main_contr}:
\begin{enumerate}
\item [(i)] Choose $R = R(s,N)$ large enough to satisfy \eqref{eq:trunc_cond}.
\item [(ii)] Choose $m = m(s,N)$ according to \eqref{eq:m_rate}.
\item [(iii)] Choose $\bm{p} \in \mathbb{R}^M$ according to \eqref{eq:defp}.
\end{enumerate}
Our next goal is to find an upper bound $\bm{\nu}$ to the local $a$-coherence $\bm{\mu}$ as in \eqref{eq:mu<=nu} such that:
\begin{itemize}
\item[(a)] It is possible to find an explicit formula for $R = R(s,N)$ such that \eqref{eq:trunc_cond} is satisfied. 
\item [(b)] The quantity $\|\bm{\nu}|_{\mathcal{Q}}\|_1$ depends sublinearly on $M$ and $N$. 
\end{itemize}

%These problems have already been solved in \cite{Brugiapaglia2018} for 1D ADR problems with constant coefficients. In this paper, we extend this analysis to the 1D case with nonconstant coefficients and to the multi-dimensional case with constant coefficients.

%%%%%%%%%%%%%%%%%%%%%%%%%%%%%%%%%%%%%%%%%%%%%%%%%%%%%%%%%%%%%%%%%%%%%%%%%%%%%%%%
\section{Local $a$-coherence estimates}
%%%%%%%%%%%%%%%%%%%%%%%%%%%%%%%%%%%%%%%%%%%%%%%%%%%%%%%%%%%%%%%%%%%%%%%%%%%%%%%%
\label{sec:localcoherence}

This section is the technical core of the article. We extend the results in \cite{Brugiapaglia2018} by deriving upper bounds to the local $a$-coherence defined in \eqref{eq:defmu} for ADR equations with nonconstant coefficients in dimension one (Section~\ref{sec:est_1d_noncost}) and with constant coefficients in arbitrary dimension (Section~\ref{sec:est_dD_ani}).

\subsection{The 1D ADR problem with nonconstant coefficients}
\label{sec:est_1d_noncost}

We start by showing some auxiliary inequalities involving inner products between biorthogonal B-spline wavelets and the Fourier basis functions, and their respective derivatives. Note that in the following statement, the basis functions are normalized with respect to the $L^2(\mathcal{D})$-norm.

\begin{lem}[Auxiliary inequalities, $\Dim = 1$]
\label{lem:1DestADR}
In Setting~\ref{ass:Bsplwave}, the following inequalities hold for every $\ell \geq \ell_0$, $k \in \mathbb{Z}/(2^\ell\mathbb{Z})$,  $q\in\mathbb{Z}\setminus\{0\}$, $(\alpha_1,\alpha_2) \in \{0,1\}^2$, and $\gamma \in [0,2]$:
\begin{align}
  |(D^{\alpha_1}\varphi_{\ell,k},D^{\alpha_2}\xi_q)| 
& \leq 2^{\alpha_1 + \alpha_2}\; 2^{(\frac32-\gamma)\ell} |\pi q|^{\gamma - 2 + \alpha_1 + \alpha_2}, \label{eq:1DestADRtermsscal}\\
 |(D^{\alpha_1}\psi_{\ell,k},D^{\alpha_2}\xi_q)| 
 & \leq 2^{\alpha_1+\alpha_2+1-\gamma} \|\bm{b}\|_2 \|\bm{b}\|_0^{\frac 12}  2^{(\frac32-\gamma)\ell} |\pi q|^{\gamma - 2 + \alpha_1 + \alpha_2  }. \label{eq:1DestADRtermswave}
\end{align}
Moreover, for $q=0$, we have 
\begin{align}
|(\varphi_{\ell,k}',\xi_0')| &=   |(\psi_{\ell,k}',\xi_0')|  = 0, \label{eq:1Destdifq0}\\
|(\varphi_{\ell,k}',\xi_0)| &=   |(\psi_{\ell,k}',\xi_0)|  = 0, \label{eq:1Destadvq0}\\
|(\varphi_{\ell,k},\xi_0)| &= 2^{-\ell/2}, \quad |(\psi_{\ell,k},\xi_0)|  = 0. \label{eq:1Destreaq0}
\end{align}

\end{lem}
\begin{proof}
If $q=0$, equations \eqref{eq:1Destdifq0}-\eqref{eq:1Destreaq0} can be verified via direct computation, being $\xi_0 \equiv 1$. Now, we analyze the case $q \neq 0$. Considering the case $(\alpha_1, \alpha_2) = (1,1)$, thanks to hypothesis \eqref{eq:nooverlap_cond}, we can directly compute 
\begin{align}
(\varphi_{\ell,k}',\xi_q') 
& = 2^{3\ell/2} \bigg(\int_{(k-1)2^{-\ell}}^{k 2^{-\ell}} \overline{\xi_q'(x)} \de x 
- \int_{k2^{-\ell}}^{(k+1) 2^{-\ell}} \overline{\xi_q'(x)} \de x\bigg)\\
& = 2^{3\ell/2} (2\enum^{-2\iunit\pi q k 2^{-\ell}} - \enum^{-2\iunit\pi q (k-1) 2^{-\ell}}  - \enum^{-2\iunit\pi q (k+1) 2^{-\ell}})\\
& = 2^{3\ell/2} \enum^{-2\iunit\pi q k 2^{-\ell}} (2-\enum^{2\iunit\pi q 2^{-\ell}} - \enum^{-2\iunit\pi q 2^{-\ell}})\\
& = 2^{3\ell/2} \enum^{-2\iunit\pi q k 2^{-\ell}} 2(1-\cos(2\pi q 2^{-\ell}))
 = 4 \cdot 2^{3\ell/2} \enum^{-2\iunit\pi q k 2^{-\ell}} \sin^2(\pi q 2^{-\ell}).
\end{align}
The inequality $\sin^2(x) \leq |x|^{\gamma}$, which holds for every $x \in \RR$ and $\gamma \in [0,2]$ (see Figure~\ref{fig:sharpness_UB_Dphi_Dxi}), yields 
\begin{figure}
\centering
\includegraphics[width = 12cm]{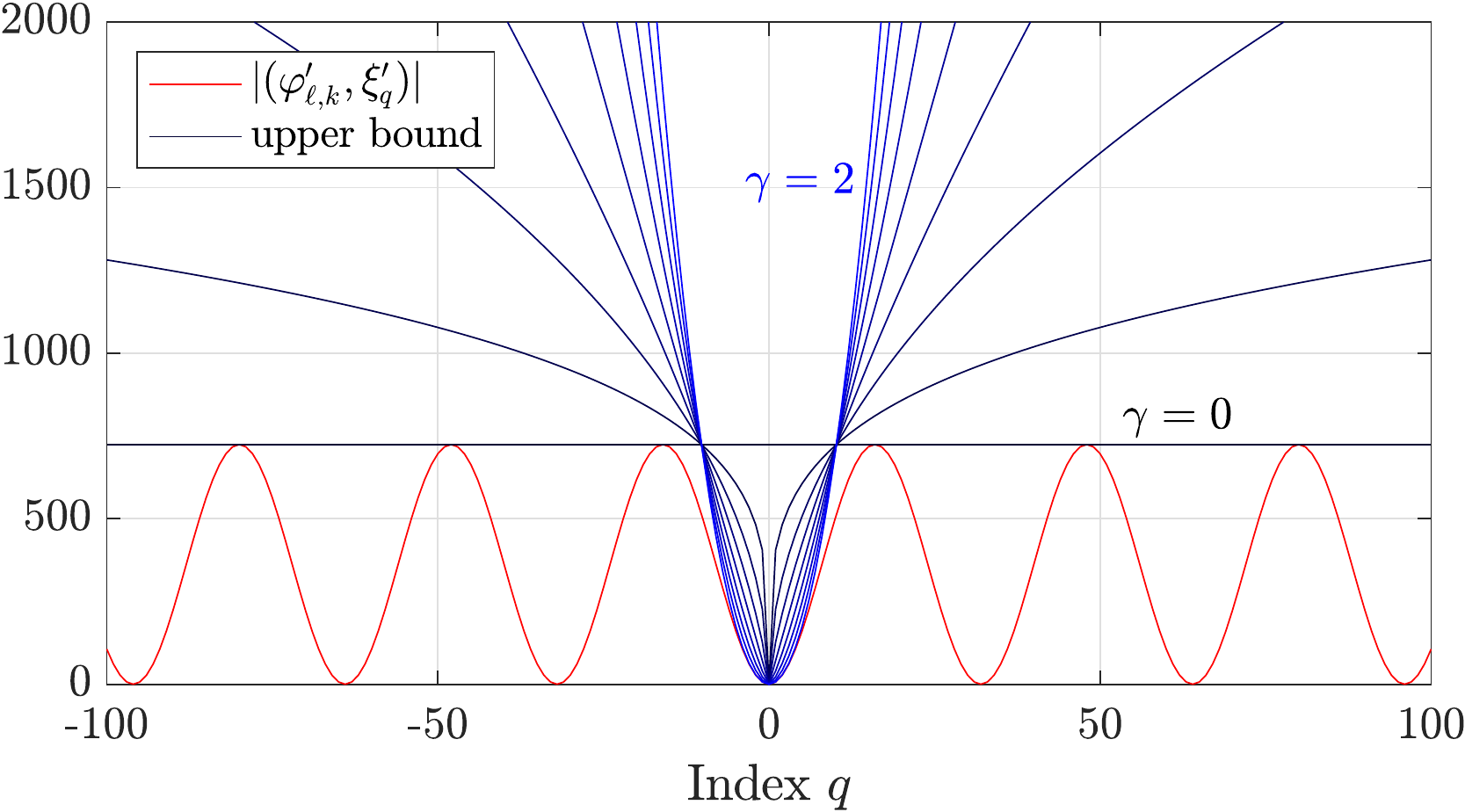}
\caption{\label{fig:sharpness_UB_Dphi_Dxi}Sharpness of the upper bound (\ref{eq:diffterm1Dub}), with $\ell = 5$, $k = 4$, and different values of $\gamma \in [0,2]$.}
\end{figure}
\begin{align}
\label{eq:diffterm1Dub}
|(\varphi_{\ell,k}',\xi_q')| 
&\leq 4 \cdot 2^{(\frac32-\gamma)\ell} |\pi q|^\gamma
\end{align}
Moreover, thanks again to hypothesis \eqref{eq:nooverlap_cond}, employing \eqref{eq:diffterm1Dub}, the discrete Cauchy-Schwarz inequality, and the definition of the mother wavelet (see also \cite{brugiapaglia2018wavelet_supplementary})
$$
\psi(x)=\sum_{k\in\mathbb{Z}} b_k \varphi(2x-k), \quad \forall x\in\mathbb{R},
$$
we obtain
\begin{align}
|(\psi_{\ell,k}',\xi_q')| 
& = \frac{1}{\sqrt{2}} \bigg|\sum_{j \in \mathbb{Z}/(2^{\ell+1}\mathbb{Z})}b_{j-2k} (\varphi_{\ell+1,j}',\xi_q')\bigg| 
 \leq \frac{\|\bm{b}\|_2}{\sqrt{2}}  \Bigg[\sum_{j \in \mathbb{Z}/(2^{\ell+1}\mathbb{Z}), \, b_{j-2k} \neq 0 } |(\varphi_{\ell+1,j}',\xi_q')|^2\Bigg]^{\frac12}\\
& \leq 4\cdot \frac{\|\bm{b}\|_2 \|\bm{b}\|_0^{\frac12}}{\sqrt{2}} 2^{(\frac32-\gamma)(\ell+1)} |\pi q|^\gamma
 = 2^{3-\gamma}\|\bm{b}\|_2 \|\bm{b}\|_0^{\frac 12}  2^{(\frac32-\gamma)\ell} |\pi q|^\gamma.
\end{align}
This concludes the case $(\alpha_1,\alpha_2) =(1,1)$. The case $(\alpha_1,\alpha_2) \neq (1,1)$ can be addressed using integration by parts since $\xi_q' = (2\pi \iunit q) \xi_q$. 

\end{proof}

%As an immediate consequence, we derive some pretty sharp local $a$-coherence upper bound for the one-dimnesional case with constant coefficients. It is worth noticing that the upper bound given in \eqref{eq:1DestADR} below exhibits a two-phase behavior. This is an intrinsic characteristic of the upper bound that will be observed also in the multi-dimensional scenario.
%
%\begin{prop}[Local $a$-coherence: 1D ADR with constant coefficients]
%In the case of an ADR with constant coefficients, it holds
%\begin{equation}
%\label{eq:1DestADR}
%\mu^L_q \lesssim 
%\bigg(|\eta| + \frac{|b|}{|q|} + \frac{|\rho|}{q^2}\bigg)^2 \min \bigg\{\frac{2^{L}}{q^2}, \frac{1}{|q|}\bigg\}, 
%\quad \forall q\in\ZZ\setminus\{0\}.
%\end{equation}
%\end{prop}
%\begin{proof}
%In order to prove \eqref{eq:1DestADR}, we notice that
%\begin{align}
%a(\psi_{\ell,k},\xi_q) 
% & = \eta  (\psi_{\ell,k}',\xi_q') 
% + b  (\psi_{\ell,k}',\xi_q) 
% + \rho  (\psi_{\ell,k},\xi_q).  
%\end{align} 
%Now, we employ estimates \eqref{eq:1DestADRtermsscal}, \eqref{eq:1DestADRtermswave} with $\gamma = 1/2$ and $\gamma = 0$, and the fact that $\|\psi_{\ell,k}\|_{H^1(\mathcal{D})}^2 \sim 2^{2\ell}$, and $\|\xi_q\|_{H^1(\mathcal{D})}^2 \sim |q|^2$. Hence, we obtain
%\begin{equation}
%\frac{|a(\psi_{\ell,k}, \xi_q)|^2}{\|\psi_{\ell,k}\|_{H^1(\mathcal{D})}^2\|\xi_q\|_{H^1(\mathcal{D})}^2}
%\lesssim
%\bigg(|\eta| + \frac{|b|}{|q|} + \frac{|\rho|}{q^2}\bigg)^2 \min \bigg\{\frac{2^{\ell}}{q^2}, \frac{1}{|q|}\bigg\}.
%\end{equation}
%Considering the supremum over $\ell \in [L]$ gives the desired estimate for the local $a$-coherence.
%\end{proof}

We are now in a position to estimate the local $a$-coherence of the wavelet basis with respect to the Fourier basis for 1D ADR equations with nonconstant coefficients.

\begin{thm}[Local $a$-coherence upper bound, $\Dim = 1$]
\label{thm:mu_AReq1D}
In Setting~\ref{ass:Bsplwave}, for $\eta,\beta\in H^1_\per(\mathcal{D})$ and $\rho \in L^2(\mathcal{D})$, the local $a$-coherence in \eqref{eq:defmu} can be bounded from above as 
\begin{align}
\mu_0 & \lesssim 2^{-2\ell_0}(|\beta|_{H^1(\mathcal{D})}^2 + \|\rho\|_{L^2(\mathcal{D})}^2)\\
\mu_q &  \lesssim \left(\|\dif\|_{H^1(\mathcal{D})}^2 + \frac{\|\beta\|_{H^1(\mathcal{D})}^2}{q^2} + \|\rea\|_{L^2(\mathcal{D})}^2\right)  \min\bigg\{\frac{2^L}{q^2},\frac{1}{|q|}\bigg\}, \quad \forall q \in \mathbb{Z} \setminus \{0\}.
\end{align}
\end{thm}

\begin{proof}
First, let us consider the case $q = 0$. We have $|(\eta \psi_{\ell,k}',\xi_0')| = |(\eta \psi_{\ell,k}',0)| = 0$,
$$
|(\beta \psi_{\ell,k}',\xi_0)| 
= \left|\int_{\mathcal{D}}\beta(x) \psi_{\ell,k}'(x) \de x\right|
= \left|\int_{\mathcal{D}}\beta'(x) \psi_{\ell,k}(x) \de x\right| 
\leq |\beta|_{H^1(\mathcal{D})} \|\psi_{\ell,k}\|_{L^2(\mathcal{D})},
$$
and 
$$
|(\rho \psi_{\ell,k},\xi_0)| 
%= |(\rho \psi_{\ell,k},1)|
\leq \|\rho\|_{L^2(\mathcal{D})} \|\psi_{\ell,k}\|_{L^2(\mathcal{D})}.
$$
Therefore, 
\begin{align*}
|a(\widehat{\psi}_{\ell,k},\widehat{\xi_0})|^2
& \leq  \left(\frac{|\beta|_{H^1(\mathcal{D})} \|\psi_{\ell,k}\|_{L^2(\mathcal{D})} + \|\rho\|_{L^2(\mathcal{D})} \|\psi_{\ell,k}\|_{L^2(\mathcal{D})}}{\|\psi_{\ell,k}\|_{H^1(\mathcal{D})}\|\xi_0\|_{H^1(\mathcal{D})}} \right)^2 \\
& = \left(\frac{\|\psi_{\ell,k}\|_{L^2(\mathcal{D})}}{\|\psi_{\ell,k}\|_{H^1(\mathcal{D})}}\right)^2 (|\beta|_{H^1(\mathcal{D})}  + \|\rho\|_{L^2(\mathcal{D})})^2 
 \lesssim 2^{-2\ell}(|\beta|_{H^1(\mathcal{D})}^2  + \|\rho\|_{L^2(\mathcal{D})}^2),
\end{align*}
which, by maximization over $\ell$ and $k$, implies the estimate for $\mu_0$.

Now, let $q\neq 0$.  The idea is to expand the diffusion, advection, and reaction terms with respect to the Fourier basis 
\begin{equation}
\eta = \sum_{r \in \mathbb{Z}} \eta_r \xi_r, \quad
\beta = \sum_{r \in \mathbb{Z}} \beta_r \xi_r, \quad
\rho = \sum_{r \in \mathbb{Z}} \rho_r \xi_r,
\end{equation}
with $\eta_r := (\eta,\xi_r)$, $\beta_r := (\beta,\xi_r)$, and $\rho_r := (\rho,\xi_r)$, for every $r \in \mathbb{Z}$. The decay of the coefficients $(\eta_r)_{r \in \mathbb{Z}}$, $(\beta_r)_{r \in \mathbb{Z}}$, and $(\rho_r)_{r \in \mathbb{Z}}$ is strictly linked with the Sobolev regularity of  $\eta$, $\beta$, and $\rho$, respectively, thanks to the norm equivalence \eqref{eq:equivHmFourier}. The estimate of $\mu_q$ is divided into four parts. In the first three parts, we assess the impact of the terms $\mu$, $\beta$, and $\rho$ on the final estimate separately. Then, we combine these estimates in the fourth part.

\paragraph{Part I: diffusion term $\eta$ ($q \neq 0$).} We start by considering the diffusion term
\begin{equation}
\label{eq:diffterm1Ddeco}
|(\dif\psi_{\ell,k}',\xi_q')|^2 
 = \bigg|\sum_{r\in\ZZ} \bigg(\frac{1+r^2}{1+r^2}\bigg)^{\frac12} \dif_r (\xi_r \psi_{\ell,k}',\xi_q')\bigg|^2
\leq \bigg(\sum_{r\in\ZZ}(1+r^2)|\dif_r|^2 \bigg) 
\sum_{r\in\ZZ} \frac{|(\xi_r \psi_{\ell,k}',\xi_q')|^2}{1+r^2}.
\end{equation}
Recalling \eqref{eq:equivHmFourier}, the first factor is bounded from above by  $\|\dif\|_{H^1(\mathcal{D})}^2$. Now, exploiting standard algebraic properties of the Fourier basis, performing integration by parts, and using \eqref{eq:1DestADRtermswave} and \eqref{eq:1Destadvq0}, we obtain the following estimate for every $\gamma \in [0,2]$:
\begin{equation}
\label{eq:diffterm1DdecoUB}
\sum_{r\in\ZZ} \frac{|(\xi_r \psi_{\ell,k}',\xi_q')|^2}{1+r^2} 
 \sim |q|^2 \sum_{r\in\ZZ} \frac{|(\psi_{\ell,k}',\xi_{q-r})|^2}{1+r^2} 
\lesssim
q^2 2^{(3-2\gamma)\ell} S_{2(1-\gamma)}(q),
\end{equation}
where
\begin{equation}
S_y(q):=\sum_{r\in\ZZ\setminus\{ q\}}\frac{ 1}{|q-r|^{y}(1+r^2)}, \quad \forall  y \in \mathbb{R}.
\end{equation}
Now, we study the asymptotic behavior of $S_{y}(q)$ for $y = 1$, and $y = 2$ (corresponding to $\gamma = 1/2$, and $\gamma = 0$, respectively).\footnote{Choosing $\gamma = 1/2$ and $\gamma = 0$ leads to the estimates $S_1(q) \lesssim 1/|q|$ and $S_2(q) \lesssim 1/|q|^2$, respectively. These, in turn, imply two upper bounds to $\mu_q$: the first one independent of $N$ and decaying linearly with respect to $q$, the second one linear in $N$ but decaying quadratically with respect to $q$. These two properties will be crucial to answer issues (i) and (ii) in Theorem~\ref{thm:CORSING_rec_1D}.}

We start by considering $S_{1}(q)$, when $q>0$. We have the splitting
\begin{equation}
S_{1}(q) = 
\underbrace{\sum_{r=-\infty}^{-1}\frac{1}{(q-r)(1+r^2)}}_{S_{1,1}(q)} 
+ \underbrace{\sum_{r=0}^{q-1}\frac{1}{(q-r)(1+r^2)}}_{S_{1,2}(q)}
+ \underbrace{\sum_{r=q+1}^{+\infty}\frac{1}{(r-q)(1+r^2)}}_{S_{1,3}(q)}.
\label{eq:splittingS-1} 
\end{equation}
We study the three sums separately, to show that each term can be bounded from above by $1/q$, up to a constant. Indeed,
\begin{align}
S_{1,1}(q) & =\sum_{r=1}^{+\infty}\frac{1}{(q+r)(1+r^2)}
\leq \frac1q \sum_{r=1}^{+\infty}\frac{1}{1+r^2}
\lesssim \frac1q,\\
S_{1,2}(q) & \leq \frac1q + \frac{1}{1+(q-1)^2}+\int_0^{q-1}\frac{1}{(q-r)(1+r^2)}\de r\\
 & = \frac1q + \frac{1}{1+(q-1)^2}+
 \bigg[\frac{2q\arctan(r)-2\log(q-r) + \log(1+r^2)}{2(1+q^2)}\bigg]\bigg|_{r=0}^{r=q-1}\\
 & = \frac1q + \frac{1}{1+(q-1)^2}+\frac{2q\arctan(q-1)+2\log(q) + \log(1+(q-1)^2)}{2(1+q^2)}
 \lesssim \frac1q,\\
S_{1,3}(q) & 
\lesssim \frac{1}{q} \sum_{r=q+1}^{+\infty} \frac{1}{(r-q)(r+1)}
\lesssim \frac{1}{q} \sum_{r=1}^{+\infty} \frac{1}{r^2}
\lesssim \frac{1}{q}.
\end{align}
The first inequality employed to bound $S_{1,2}(q)$ relies on a property of the function $g(r) := 1/[(q-r)(1+r^2)]$ such that for every $q \geq 2$, $g$ admits only one stationary point $r^*=\frac13(q-\sqrt{q^2-3})$ in the open interval $(0,q-1)$ such that $g(r^*) \leq g(q-1) \leq g(0)$, and  $g$ decreases monotonically in $[0,r^*]$ and increases monotonically in $[r^*,q-1]$. In particular, this implies
\begin{align*}
\sum_{r=0}^{q-1} g(r)
& = \sum_{r=0}^{\lfloor r^* \rfloor} g(r) + \sum_{r=\lfloor r^* \rfloor+1}^{q-1} g(r)
\leq g(0) + \int_0^{\lfloor r^*\rfloor} g(r) \de r + g(q-1) + \int_{\lfloor r^*\rfloor + 1}^{q-1} g(r) \de r\\
& \leq g(0) +  g(q-1) + \int_{0}^{q-1} g(r) \de r.
\end{align*}
Also, notice that when $q = 1$ the term $S_{1,2}(q)$ is equal to 1. 

Observing that $S_y(q)$ is even with respect to $q$, we conclude that $S_{1}(q) \lesssim 1/|q|$, for every $q\neq 0$.

We carry out a similar analysis for $S_{2}(q)$. For $q>0$, we have
\begin{equation}
\label{eq:splittingS-2}
S_{2}(q) = 
\underbrace{\sum_{r=-\infty}^{-1}\frac{1}{(q-r)^2(1+r^2)}}_{S_{2,1}(q)} 
+ \underbrace{\sum_{r=0}^{q-1}\frac{1}{(q-r)^2(1+r^2)}}_{S_{2,2}(q)}
+ \underbrace{\sum_{r=q+1}^{+\infty}\frac{1}{(r-q)^2(1+r^2)}}_{S_{2,3}(q)}. 
\end{equation}
Using arguments similar to those employed for $S_{1}(q)$, we obtain
\begin{align}
S_{2,1}(q)& 
\leq \frac{1}{(q+1)^2}\sum_{r=1}^{+\infty} \frac{1}{1+r^2}\lesssim \frac{1}{q^2} \\
S_{2,2}(q)& 
\leq \frac{1}{q^2} + \frac{1}{1+(q-1)^2} 
+ \int_0^{q-1}\frac{1}{(q-r)^2(1+r^2)}  \de r\\
& \lesssim \frac{1}{q^2} 
+ \frac{(q-1)(1+q^2) + (q^3-q)\arctan(q-1) + q^2 (2\log(q)+\log(1+(q-1)^2))}{q(1+q^2)^2}
 \lesssim \frac{1}{q^2},\\
S_{2,3}(q) & 
\leq \frac{1}{1+(q+1)^2}\sum_{r=q+1}^{+\infty}\frac{1}{(r-q)^2}
\lesssim\frac{1}{q^2}.
\end{align}
Using again that $S_y(q)$ is even with respect to $q$, we conclude that $S_{2}(q) \lesssim 1/q^2$, for every $q \neq 0$.\footnote{In view of Lemma~\ref{lem:1DestADR}, we conjecture that $S_{y}(q) \lesssim 1/|q|^{y}$ for every $y \in [-4,2]$ (corresponding to $\gamma\in[0,2]$ and $y = 2(1-\gamma)$). Nevertheless, proving this rigorously is not straightforward, since the terms corresponding to $S_{1,2}(q)$ and $S_{2,2}(q)$ become  very difficult to analyze.}

Recalling relations \eqref{eq:diffterm1Ddeco} and \eqref{eq:diffterm1DdecoUB}, we have  
\begin{equation}
\label{eq:mu_dif_1D}
|(\dif \psi_{\ell,k}',\xi_q')|^2 
\lesssim \|\dif\|_{H^1(\mathcal{D})}^2 q^2 \min\{ 2^{3\ell} S_{2}(q), 2^{2\ell} S_{1}(q)\}
\lesssim \|\dif\|_{H^1(\mathcal{D})}^2 2^{2\ell} \min\{ 2^{\ell} , |q| \}.
\end{equation}

\paragraph{Part II: advection term $\beta$ ($q \neq 0$).}
Analogously to the diffusion case, we have
\begin{align}
|(\beta \psi_{\ell,k}', \xi_q)|^2 
& \leq \bigg(\sum_{r \in \mathbb{Z}}|\beta_r|^2(1+r^2)\bigg)
\sum_{r \in \mathbb{Z}} \frac{|(\xi_r \psi'_{\ell,k}, \xi_q)|^2}{1+r^2}
\sim \|\beta\|_{H^1(\mathcal{D})}^2 \sum_{r \in \mathbb{Z}} \frac{|(\psi'_{\ell,k}, \xi_{q-r})|^2}{1+r^2}.
\end{align}
Estimating the sum in the right-hand side as before, we obtain
\begin{equation}
\label{eq:mu_adv_1D}
|(\beta \psi_{\ell,k}', \xi_q)|^2 
\lesssim
\frac{\|\beta\|^2_{H^1(\mathcal{D})}}{q^2} 2^{2\ell} \min\{2^\ell,|q|\}.
\end{equation}

\paragraph{Part III: reaction term $\rho$ ($q \neq 0$).} We deal with the nonconstant reaction term in an analogous way. Recalling Lemma~\ref{lem:1DestADR} and the norm equivalence \eqref{eq:equivHmFourier} with $k=0$, we have
\begin{equation}
|(\rea \psi_{\ell,k},\xi_q)|^2 
\leq \bigg(\sum_{r \in \ZZ} |\rea_r|^2\bigg)
\sum_{r\in\ZZ} |(\psi_{\ell,k},\xi_{q-r})|^2
\lesssim \|\rea\|_{L^2(\mathcal{D})}^2 \bigg(2^{(3-2\gamma)\ell}\bigg(\sum_{r \in \ZZ \setminus \{q\}} |q-r|^{2(\gamma-2)}\bigg)  + 2^{-\ell/2}\bigg),
\end{equation}
for every $\gamma \in [0,2]$, where we have employed \eqref{eq:1DestADRtermsscal}-\eqref{eq:1DestADRtermswave} with $(\alpha_1,\alpha_2) = (0,0)$ and $\xi_{q-r}$ for $r\neq q$ and \eqref{eq:1Destreaq0} with $\xi_{q-r}=\xi_0$ for $r=q$. Choosing $\gamma = 1/2$, we see that  
\begin{equation}
\label{eq:mu_rea_1D}
|(\rea \psi_{\ell,k},\xi_q)|^2 
\lesssim \|\rea\|_{L^2(\mathcal{D})}^2 (2^{2\ell} + 1) 
\lesssim \|\rea\|_{L^2(\mathcal{D})}^2 2^{2\ell}, 
\end{equation}
since $\sum_{r \in \ZZ \setminus \{q\}} |q-r|^{-3} \lesssim 1$, for every $q \neq 0$.  
\paragraph{Part IV: conclusion ($q \neq 0$).} Combining \eqref{eq:mu_dif_1D}, \eqref{eq:mu_adv_1D}, and \eqref{eq:mu_rea_1D} finally yields
\begin{equation}
|a(\psi_{\ell,k},\xi_q)|^2 
\lesssim \left(\|\dif\|_{H^1(\mathcal{D})}^2 + \frac{\|\beta\|_{H^1(\mathcal{D})}^2}{q^2} + \|\rea\|_{L^2(\mathcal{D})}^2\right) 2^{2\ell} \min\{2^\ell,|q|\}.
\end{equation}
As a consequence, normalizing the trial and test functions with respect to the $H^1(\mathcal{D})$-norm and using that $\|\psi_{\ell,k}\|_{H^1(\mathcal{D})} \sim 2^\ell$ and that $\|\xi_q\|_{H^1(\mathcal{D})} \sim |q|$ (recall \eqref{eq:normpsij_ani_1} and \eqref{eq:normxiqmulti-d}), we obtain
\begin{equation}
\mu_q \lesssim \left(\|\dif\|_{H^1(\mathcal{D})}^2 + \frac{\|\beta\|_{H^1(\mathcal{D})}^2}{q^2} + \|\rea\|_{L^2(\mathcal{D})}^2\right)  \min\bigg\{\frac{2^L}{q^2},\frac{1}{|q|}\bigg\}.
\end{equation}
This completes the proof.
\end{proof}

\begin{rmrk}[Sharper upper bound for $\rho \in H^1_\per(\mathcal{D})$]
\label{rmrk:mu_bound_rhoH1}
It is not difficult to show that an upper bound for $\mu_q$ as in Theorem~\ref{thm:mu_AReq1D} holds when $\rea \in H^1_\per(\mathcal{D})$, with the following estimate for the local $a$-coherence when $q \neq 0$:
\begin{equation}
\label{eq:mu_est_1D_2}
\mu_q 
\lesssim \left(\|\dif\|_{H^1(\mathcal{D})}^2 + \frac{\|\beta\|_{H^1(\mathcal{D})}^2}{q^2} + \frac{\|\rea\|_{H^1(\mathcal{D})}^2}{q^4}\right)  
\min\bigg\{\frac{2^L}{q^2},\frac{1}{|q|}\bigg\}.
\end{equation}
Notice that \eqref{eq:mu_est_1D_2} generalizes \cite[Proposition 4.4]{Brugiapaglia2018} (in particular, we refer to \cite[equation (145)]{Brugiapaglia2018}) for 1D ADR equations with constant coefficients and nonperiodic boundary conditions. \hfill $\blacksquare$
\end{rmrk}

Finally, we have the \corsing \WF recovery theorem that provides an answer to the three items (i), (ii), and (iii) in Section~\ref{sec:i,ii,iii} for the 1D case.

\begin{thm}[{\corsing \WF recovery}] 
\label{thm:CORSING_rec_1D}
In Setting~\ref{ass:Bsplwave} and under the same hypotheses as in Theorem~\ref{thm:CORSINGrecovery}, let $\Dim = 1$, $\eta, \beta \in H^1_\per(\mathcal{D})$ and $\rho \in L^2(\mathcal{D})$. Then, provided 
\begin{itemize}
\item[\textup{(i)}] $R  \sim C s N$,
\item[\textup{(ii)}] $m  \gtrsim C  s  (s \ln(eN/(2s)) + \ln(2 s/\varepsilon)) (\ln N + \ln s + \ln C)$,
\item[\textup{(iii)}] $p_q  \propto \begin{cases}
1, & q = 0, \\
\min\left\{\frac{N}{q^2}, \frac{1}{|q|}\right\}, & q \neq 0,
\end{cases}$
\end{itemize}
where 
\begin{equation}
\label{eq:defC_1D}
C := \|\eta\|_{H^1(\mathcal{D})}^2 + \|\beta\|_{H^1(\mathcal{D})}^2 + \|\rho\|_{L^2(\mathcal{D})}^2,
\end{equation}
the \corsing \WF method recovers the best $s$-term approximation to $u$ in expectation, in the sense of estimate \eqref{eq:corsingrecovery}.
\end{thm}

\begin{proof}
Employing Theorem~\ref{thm:mu_AReq1D}, we choose
\begin{equation}
\nu_q = 
\begin{cases}
C, & q = 0, \\
C \min\left\{\frac{N}{q^2}, \frac{1}{|q|}\right\}, & q \neq 0.
\end{cases}
\end{equation}
As a consequence, we have $\|\bm{\nu}\|_1 < +\infty$. Then, to ensure condition \eqref{eq:trunc_cond}, using that $\nu_q \leq C N/q^2$ for $q\neq0$ we estimate
\begin{equation}
s \|\bm{\nu}|_{\mathcal{Q}^c}\|_1
\leq C s N \sum_{|q| \geq \lfloor R/2 \rfloor -1} \frac{1}{q^2} 
\lesssim \frac{C s N}{R}.
\end{equation}
Therefore, to ensure \eqref{eq:trunc_cond} we let $R = R(s,N) \sim C s N$. Moreover, using that $\nu_q\leq C/|q|$ for $q\neq 0$, we see that
$$
\|\bm{\nu}|_{\mathcal{Q}}\|_1 
\leq C\bigg(1 + \sum_{0< |q| \leq \lfloor R/2 \rfloor} \frac{1}{|q|}\bigg)
\lesssim C \ln R = C \ln M \sim C(\ln N  + \ln s +\ln{C}),
$$
which depends sublinearly on $M$ and $N$, as desired. 
\end{proof}

%%%%%%%%%%%%%%%%%%%%%%%%%%%%%%%%%%%%%%%%%%%%%%%%%%%
%%%%%%%%%%%%%%%%%%%%%%%%%%%%%%%%%%%%%%%%%%%%%%%%%%%
%%%%%%%%%%%%%%%%%%%%%%%%%%%%%%%%%%%%%%%%%%%%%%%%%%%
\subsection{The multi-dimensional case}
\label{sec:est_dD_ani}

We consider issues (i), (ii), and (iii) in Section~\ref{sec:i,ii,iii} for the \corsing \WF  for multi-dimensional ADR equations with constant coefficients. In Section~\ref{sec:aniso} we analyze the case of anisotropic tensor product wavelet, while in Section~\ref{sec:iso} we deal with the isotropic case.

\subsubsection{Anisotropic tensor product wavelets}
\label{sec:aniso}

We provide local $a$-coherence upper bounds (Theorem~\ref{thm:mu_bound_multi_ani}) and a recovery result (Theorem~\ref{thm:CORSING_rec_multi_ani}) for the \corsing \WF method with anisotropic tensor product wavelets.

We start by proving a technical result analogous to Lemma~\ref{lem:1DestADR}. To shorten notations, we introduce 
$$
|\bm{x}|^\bm{y} := \prod_{j = 1}^k |x_j| ^{y_j}, \quad \forall \bm{x}, \bm{y} \in \mathbb{R}^k, \; \forall k \in \NN.
$$
Moreover, we denote $\bm{1} = (1,1,\ldots,1)$, $\bm{2} = (2,2,\ldots,2)$, and so on. 
\begin{lem}[Auxiliary inequalities, anisotropic wavelets]
\label{lem:aux_ineq_multid_ani}
In Setting~\ref{ass:Bsplwave}, let $n >1$, $\bm{\ell} \in \NN^\Dim$, with $\bm{\ell} \geq \ell_0-1$, $\bm{k} \in \mathbb{Z}/(2^{\bm{\ell}}\mathbb{Z})$, and $\bm{q} \in \ZZ^\Dim$. Moreover, define\footnote{Note that, according to \eqref{eq:notationl0-1}, $\scal(\bm{\ell})$ is the set of indices $j$ such that $\psi_{\ell_j,k_j}$ is a scaling function.} 
\begin{equation}
\label{eq:def_zero_scal}
\zero(\bm{q})  := [\Dim] \setminus \supp(\bm{q}), \quad
\scal(\bm{\ell})  := \{j\in[\Dim] : \ell_j = \ell_0 -1\}.
\end{equation}
Then, it follows
\begin{itemize}
\item if $\bm{q} = \bm{0}$, we have
\begin{align}
\label{eq:aux_multi_1}
|(\nabla \psi^{\ani}_{\bm{\ell},\bm{k}}, \nabla \xi_{\bm{0}})| & = 0,\\
\label{eq:aux_multi_2}
|(\bm{\beta} \cdot \nabla \psi^{\ani}_{\bm{\ell},\bm{k}},  \xi_{\bm{0}})| & = 0, \quad \forall \bm{\beta} \in \mathbb{R}^\Dim\\
\label{eq:aux_multi_3}
|(\psi^{\ani}_{\bm{\ell},\bm{k}}, \xi_{\bm{0}})| & = 
\begin{cases}
2^{- \Dim\ell_0 / 2}, & \mathrm{if }\; \scal(\bm{\ell}) = [\Dim],\\
0, & \mathrm{otherwise;}
\end{cases}
\end{align}
\item if $\bm{q} \neq \bm{0}$ and $\zero(\bm{q}) \subseteq \scal(\bm{\ell})$, then, for every $\gammavec \in [0,2]^{\|\bm{q}\|_0}$, it holds
\begin{align}
|(\nabla \psi^{\ani}_{\bm{\ell},\bm{k}},\nabla \xi_{\bm{q}})| 
& \lesssim
2^{-\frac12(\Dim-\|\bm{q}\|_0)\ell_0 + (\frac32 - \widehat\gammavec) \cdot \widehat{\bm{\ell}}} 
|\widehat{\bm{q}}|^{\widehat\gammavec - \bm{2}}
\|\bm{q}\|_2^2, \label{eq:aux_multi_4}\\
|(\bm{\beta} \cdot \nabla \psi^{\ani}_{\bm{\ell},\bm{k}},\xi_{\bm{q}})| 
& \lesssim 
2^{-\frac12(\Dim-\|\bm{q}\|_0)\ell_0 + (\frac32 - \widehat\gammavec) \cdot \widehat{\bm{\ell}}} 
|\widehat{\bm{q}}|^{\widehat\gammavec - \bm{2}} \|\bm{\beta}\|_2 \|\bm{q}\|_2, 
\quad \forall \bm{\beta} \in \mathbb{R}^\Dim,\label{eq:aux_multi_5} \\
|(\psi^{\ani}_{\bm{\ell},\bm{k}},\xi_{\bm{q}})| 
& \lesssim 
2^{-\frac12(\Dim-\|\bm{q}\|_0)\ell_0 + (\frac32 - \widehat\gammavec) \cdot \widehat{\bm{\ell}}} 
|\widehat{\bm{q}}|^{\widehat\gammavec - \bm{2}},
\label{eq:aux_multi_6}
\end{align}
where $\widehat{\bm{x}} := \bm{x}|_{\supp(\bm{q})} \in \mathbb{R}^{\|\bm{q}\|_0}$ for every $\bm{x} \in \mathbb{R}^\Dim$ and the inequalities hide constants that depend exponentially on $\Dim$;

\item if $\bm{q} \neq \bm{0}$ and $\zero(\bm{q}) \not\subseteq \scal(\bm{\ell})$, it holds 
\begin{align}
\label{eq:aux_multi_7}
|(\nabla \psi^{\ani}_{\bm{\ell},\bm{k}}, \nabla \xi_{\bm{q}})| & = 0,\\
\label{eq:aux_multi_8}
|(\bm{\beta} \cdot \nabla \psi^{\ani}_{\bm{\ell},\bm{k}}, \xi_{\bm{q}})| & = 0, \quad \forall \bm{\beta} \in \mathbb{R}^\Dim\\
\label{eq:aux_multi_9}
|(\psi^{\ani}_{\bm{\ell},\bm{k}}, \xi_{\bm{q}})| & = 0.
\end{align}
\end{itemize}
\end{lem}
\begin{proof}
In the case $\bm{q} = \bm{0}$, the equalities \eqref{eq:aux_multi_1}, \eqref{eq:aux_multi_2}, and \eqref{eq:aux_multi_3} are a direct consequence of \eqref{eq:1Destdifq0}, \eqref{eq:1Destadvq0}, and \eqref{eq:1Destreaq0}. Let us consider $\bm{q} \neq \bm{0}$. Then, we organize the proof discussing two cases: $\|\bm{q}\|_0 =\Dim$ and $\|\bm{q}\|_0 < \Dim$. 

\paragraph{Case $\|\bm{q}\|_0 =\Dim$.} Due to the tensorized form of the trial and the test basis functions, the following relations hold:
\begin{align}
\label{eq:diftermtensor}
(\nabla \psi^{\ani}_{\bm{\ell},\bm{k}},\nabla \xi_{\bm{q}})
& = \sum_{j = 1}^\Dim (\psi_{\ell_j,k_j}',\xi_{q_j}')
\prod_{i \neq j} (\psi_{\ell_i,k_i},\xi_{q_i}),\\
(\bm{\beta} \cdot \nabla\psi^{\ani}_{\bm{\ell},\bm{k}},\xi_{\bm{q}})
& = \sum_{j = 1}^\Dim \beta_j (\psi_{\ell_j,k_j}',\xi_{q_j})
\prod_{i \neq j} (\psi_{\ell_i,k_i},\xi_{q_i}), \quad \forall \bm{\beta}\in\mathbb{R}^\Dim, \label{eq:advtermtensor}\\
(\psi^{\ani}_{\bm{\ell},\bm{k}},\xi_{\bm{q}})
& = \prod_{j = 1}^\Dim (\psi_{\ell_j,k_j},\xi_{q_j}).\label{eq:reatermtensor}
\end{align}
Plugging relations \eqref{eq:1DestADRtermsscal} and  \eqref{eq:1DestADRtermswave} into \eqref{eq:diftermtensor} we see that, for every $\gammavec \in [0,2]^\Dim$, we have
\begin{align}
|(\nabla \psi^{\ani}_{\bm{\ell},\bm{k}},\nabla \xi_{\bm{q}})| 
%& \leq \sum_{j = 1}^\Dim |(\psi_{\ell_j,k_j}',\xi_{q_j}')| \prod_{i \neq j} |(\psi_{\ell_i,k_i},\xi_{q_i})|\\
& \lesssim
\sum_{j = 1}^\Dim 2^{(\frac32 - \gamma_j) \ell_j}  |q_j|^{\gamma_j}
\prod_{i \neq j} 2^{(\frac32 - \gamma_i) \ell_i}  |q_i|^{\gamma_i-2}
 =
2^{\frac32 \|\bm{\ell}\|_1 - \gammavec \cdot \bm{\ell}} 
|\bm{q}|^{\gammavec -\bm{2}}
\|\bm{q}\|_2^2.
\end{align}
Similarly, plugging \eqref{eq:1DestADRtermsscal} and  \eqref{eq:1DestADRtermswave} into \eqref{eq:advtermtensor}, we obtain
\begin{align}
|(\bm{\beta} \cdot \nabla\psi^{\ani}_{\bm{\ell},\bm{k}}, \xi_{\bm{q}})| 
%& \leq \sum_{j = 1}^\Dim |\beta_j||(\psi_{\ell_j,k_j}',\xi_{q_j})| \prod_{i \neq j} |(\psi_{\ell_i,k_i},\xi_{q_i})|\\
& \lesssim
\sum_{j = 1}^\Dim |\beta_j|2^{(\frac32 - \gamma_j) \ell_j}  |q_j|^{\gamma_j-1}
\prod_{i \neq j} 2^{(\frac32 - \gamma_i) \ell_i}  |q_i|^{\gamma_i-2}
 \leq
2^{\frac32 \|\bm{\ell}\|_1 - \gammavec \cdot \bm{\ell}} 
|\bm{q}|^{\gammavec -\bm{2}}
\|\bm{\beta}\|_2 \|\bm{q}\|_2.
\end{align}
Finally, plugging \eqref{eq:1DestADRtermsscal} and  \eqref{eq:1DestADRtermswave} into \eqref{eq:reatermtensor}, it follows
\begin{equation}
|(\psi^{\ani}_{\bm{\ell},\bm{k}}, \xi_{\bm{q}})| 
%=\prod_{j = 1}^\Dim |(\psi_{\ell_j,k_j}, \xi_{q_j})| 
\lesssim \prod_{j = 1}^\Dim 2^{(\frac32-\gamma_j)\ell_j} |q_j|^{\gamma_j - 2} 
= 2^{\frac32 \|\bm{\ell}\|_1 - \gammavec \cdot \bm{\ell}} 
|\bm{q}|^{\gammavec -\bm{2}}.
\end{equation}
The relations above prove \eqref{eq:aux_multi_4}, \eqref{eq:aux_multi_5}, and \eqref{eq:aux_multi_6}.

\paragraph{Case $\|\bm{q}\|_0 < \Dim$.} Let us consider the diffusion term $(\nabla \psi^{\ani}_{\bm{\ell},\bm{k}}, \nabla \xi_{\bm{q}})$. 

First, assume that $\zero(\bm{q}) \nsubseteq \scal(\bm{\ell})$ (notice also that $\zero(\bm{q})$ is nonempty since  $\|\bm{q}\|_0 < \Dim$). In this case, we can pick an index $j_0 \in \zero(\bm{q}) \setminus \scal(\bm{\ell})$, i.e., such that $q_{j_0} = 0$ and that  $\ell_{j_0} \geq \ell_0$ (that is  such that $\psi_{\ell_{j_0},k_{j_0}}$ is a wavelet function and not a scaling function).  Combining  \eqref{eq:diftermtensor} with relations \eqref{eq:1Destdifq0} and \eqref{eq:1Destreaq0} yields 
\begin{equation}
(\nabla \psi_{\bm{\ell},\bm{k}}^{\ani},\nabla \xi_{\bm{q}}) 
= 
\underbrace{(\psi_{\ell_{j_0},k_{j_0}}',\xi_{0}')}_{=0}
\prod_{i \neq j_0} (\psi_{\ell_i,k_i},\xi_{q_i})
+
\sum_{j \neq j_0} (\psi_{\ell_j,k_j}',\xi_{q_j}') 
\underbrace{(\psi_{\ell_{j_0},k_{j_0}},\xi_{0})}_{=0} 
\prod_{i \notin\{ j,j_0\}} (\psi_{\ell_i,k_i},\xi_{q_i})
= 0.
\end{equation}
This proves \eqref{eq:aux_multi_7} (\eqref{eq:aux_multi_8} and \eqref{eq:aux_multi_9} are shown analogously). As a consequence, the only possibility for $(\nabla \psi^{\ani}_{\bm{\ell},\bm{k}}, \nabla \xi_{\bm{q}})$ to be nonzero is to have $\zero(\bm{q}) \subseteq \scal(\bm{\ell})$. In this case, by splitting the  sum above and employing \eqref{eq:1Destdifq0}, we obtain
\begin{align}
|(\nabla \psi^{\ani}_{\bm{\ell},\bm{k}}, \nabla\xi_{\bm{q}})|
& \leq 
\sum_{j \in \zero(\bm{q})} \underbrace{|(\varphi_{\ell_0,k_j}',\xi_0')|}_{=0} \prod_{i \neq j} |(\psi_{\ell_i,k_i},\xi_{q_i})|
+
\sum_{j \in \supp(\bm{q})} |(\psi_{\ell_j,k_j}',\xi_{q_j}')| 
\prod_{i \neq j} |(\psi_{\ell_i,k_i},\xi_{q_i})|.
\end{align}
Now, splitting the product and using \eqref{eq:1Destreaq0} yields
\begin{align}
|(\nabla \psi^{\ani}_{\bm{\ell},\bm{k}}, \nabla\xi_{\bm{q}})|
& \leq
\sum_{j \in \supp(\bm{q})} |(\psi_{\ell_j,k_j}',\xi_{q_j}')| 
\prod_{i \in \zero(\bm{q}) } \underbrace{|(\varphi_{\ell_0,k_i},\xi_0)|}_{ = 2^{-\ell_0/2} }
\prod_{i \in \supp(\bm{q}) \setminus \{j\}} |(\psi_{\ell_i,k_i},\xi_{q_i})|\\
& =
2^{-(\Dim-\|\bm{q}\|_0)\ell_0/2}
\sum_{j \in \supp(\bm{q})} |(\psi_{\ell_j,k_j}',\xi_{q_j}')| 
\prod_{i \in \supp(\bm{q}) \setminus \{j\}} |(\psi_{\ell_i,k_i},\xi_{q_i})|.
\end{align}
In order to prove \eqref{eq:aux_multi_4}, it is sufficient to apply an argument analogous to the case $\|\bm{q}\|_0 =\Dim$, where the set $[\Dim]$ is replaced with $\supp(\bm{q})$,  $\bm{q}$ with $\widehat{\bm{q}}$ and $\bm{\ell}$ with $\widehat{\bm{\ell}}$.

Similar arguments lead to \eqref{eq:aux_multi_5}, \eqref{eq:aux_multi_6}.
\end{proof}

\begin{rmrk}[Curse of dimensionality]
It is worth stressing that estimates \eqref{eq:aux_multi_1}--\eqref{eq:aux_multi_9} are affected by the curse of dimensionality, since they hide constants that blow up exponentially with $\Dim$. This makes the \corsing \WF approach applicable only for moderate values of $\Dim$. This is not a major problem for fluid-dynamics applications, where the physical domain has always dimension $\Dim \leq 3$. \hfill $\blacksquare$
\end{rmrk}

Equipped with the auxiliary inequalities of Lemma~\ref{lem:aux_ineq_multid_ani}, we are now in a position to provide upper bounds to the local $a$-coherence in the multi-dimensional case for anisotropic tensor product wavelets.

\begin{thm}[Local $a$-coherence upper bound, anisotropic wavelets]
\label{thm:mu_bound_multi_ani}
In Setting~\ref{ass:Bsplwave}, let $n >1$. Then,   for every $\eta, \rho \in \mathbb{R}$ and $\bm{\beta} \in \mathbb{R}^\Dim$, the following upper bounds hold: 
\begin{align}
\label{eq:mu_UB_multid_q0}
\mu_{\bm{0}} & \lesssim |\rho|^2 2^{-(2+ \Dim)\ell_0}.  \\
\label{eq:mu_UB_multid}
\mu_{\bm{q}} & 
\lesssim 
\bigg(|\eta|^2 +  \frac{\|\bm{\beta}\|_2^2}{\|\bm{q}\|_2^2} + \frac{|\rho|^2}{\|\bm{q}\|_2^4}
\bigg)  
2^{-(\Dim-\|\bm{q}\|_0)\ell_0} 
\min\bigg\{\frac{2^{(3\|\bm{q}\|_0-2)L} \|\bm{q}\|_2^2}{ |\widehat{\bm{q}}|^{\bm{4}}}, 
\frac{ \|\bm{q}\|_2^2}{\|\bm{q}\|_\infty^2 |\widehat{\bm{q}}|^{\bm{1}}}\bigg\}, 
\quad \forall \bm{q} \neq \bm{0},
\end{align}
where the inequalities hide constants depending exponentially on $\Dim$.
\end{thm}
\begin{proof}
Let us assume $\bm{q} = \bm{0}$. Recalling \eqref{eq:aux_multi_1} and \eqref{eq:aux_multi_2}, we have 
$$
|a(\psi_{\bm{\ell},\bm{k}}^\ani,\xi_{\bm{0}})|
\leq 
|\eta|\underbrace{|(\nabla\psi_{\bm{\ell},\bm{k}}^\ani,\nabla\xi_{\bm{0}})|}_{=0}
+
\underbrace{|(\bm{\beta}\cdot\nabla\psi_{\bm{\ell},\bm{k}}^\ani,\xi_{\bm{0}})|}_{=0}
+
|\rho||(\psi_{\bm{\ell},\bm{k}}^\ani,\xi_{\bm{0}})|.
$$
Employing \eqref{eq:aux_multi_3} and recalling \eqref{eq:normpsij_ani_1} and  \eqref{eq:normxiqmulti-d}, we obtain
$$
|a(\widehat\psi_{\bm{\ell},\bm{k}}^\ani,\widehat\xi_{\bm{0}})|^2
\leq 
\begin{cases}
|\rho|^2 2^{-\Dim\ell_0 } / 2^{2\ell_0} = 2^{-(2 + \Dim)\ell_0}|\rho|^2,  & \text{if } \scal(\bm{\ell}) = [\Dim],\\
0, & \text{otherwise,}
\end{cases}
$$
which, in turn, implies \eqref{eq:mu_UB_multid_q0}.

When $\bm{q} \neq \bm{0}$, we consider the cases $\|\bm{q}\|_0 = \Dim$ and $\|\bm{q}\|_0 < \Dim$.

\paragraph{Case  $\|\bm{q}\|_0 =\Dim$.} Employing the auxiliary inequalities \eqref{eq:aux_multi_4}--\eqref{eq:aux_multi_6} (notice that in this case $\zero(\bm{q}) = \emptyset \subseteq \scal(\bm{\ell})$) and recalling relations \eqref{eq:normpsij_ani_1} and  \eqref{eq:normxiqmulti-d} on the $H^1(\mathcal{D})$-norm of $\psi^{\ani}_{\bm{\ell},\bm{k}}$ and of $\xi_{\bm{q}}$, for every $\gammavec \in [0,2]^\Dim$ and $\bm{q} \neq \bm{0}$, we have 
\begin{equation}
\label{eq:rationormaliz}
|a(\widehat\psi^{\ani}_{\bm{\ell},\bm{k}},\widehat\xi_{\bm{q}})|
= \frac{|a(\psi^{\ani}_{\bm{\ell},\bm{k}},\xi_{\bm{q}})|}{\|\psi^{\ani}_{\bm{\ell},\bm{k}}\|_{H^1(\mathcal{D})}\|\xi_{\bm{q}}\|_{H^1(\mathcal{D})}}
\lesssim
2^{(\frac32- \gammavec )\cdot \bm{\ell} - \|\bm{\ell}\|_\infty}
|\bm{q}|^{\gammavec -\bm{2}}
\|\bm{q}\|_2
\bigg(
|\eta| +  \frac{\|\bm{\beta}\|_2}{\|\bm{q}\|_2} + \frac{|\rho|}{\|\bm{q}\|_2^2}
\bigg).
\end{equation}
For any fixed  $\bm{q}\in \ZZ^\Dim\setminus\{\bm{0}\}$, we  make two different choices of $\gammavec = \gammavec(\bm{q})$ that, in turn, generate two upper bounds to the local $a$-coherence. 

We start by choosing $\gammavec = \bm{0}$, in order to let the upper bound decay as fast as possible in $\bm{q}$. Squaring \eqref{eq:rationormaliz} and considering the maximum over $\bm{j} \in \mathcal{J}$, we obtain
\begin{equation}
\mu_\bm{q} 
\lesssim 
\bigg(|\eta|^2 +  \frac{\|\bm{\beta} \|_2^2}{\|\bm{q}\|_2^2} + \frac{|\rho|^2}{\|\bm{q}\|_2^4}
\bigg)
\frac{\|\bm{q}\|_2^2}{|\bm{q}|^{\mathbf{4}}}
\max_{\ell_0 \leq \bm{\ell} < L} 
2^{3 \|\bm{\ell}\|_1 - 2\|\bm{\ell}\|_\infty}.
\end{equation}
Finally, by noticing that
\begin{equation}
\max_{\ell_0 \leq \bm{\ell} < L} 
2^{3 \|\bm{\ell}\|_1 - 2\|\bm{\ell}\|_\infty}
\leq
\max_{\ell_0 \leq \bm{\ell} < L} 
2^{(3 \Dim-2)\|\bm{\ell}\|_\infty}
\leq
2^{(3 \Dim-2)L},
\end{equation}
we obtain the upper bound corresponding to the first argument of the minimum in \eqref{eq:mu_UB_multid}, for $\|\bm{q}\|_0 = \Dim$.

Now, we find a second upper bound to $\mu_\bm{q}$ that is independent of $L$, but decays more slowly with respect to $\bm{q}$. Consider an index $j_{\infty} \in [\Dim]$ such that $|q_{j_\infty}| = \|\bm{q}\|_\infty$. We define $\widetilde\gammavec = \widetilde\gammavec(\bm{q})$ componentwise as 
\begin{equation}
\label{eq:defgammatilde}
\widetilde\gamma_j := 
\begin{cases}
1/2 & \text{if } j = j_\infty, \\
3/2 & \text{otherwise}, 
\end{cases}
\end{equation}
and choose $\bm{\gamma} = \widetilde {\bm{\gamma}}$. Considering  the set $\triangle:=\{\bm{\ell} \in \mathbb{N}^\Dim : \ell_0 \leq \bm{\ell} < L, \, \ell_1 \geq \ell_2 \geq \cdots \geq \ell_\Dim\}$  and defining $S_\Dim$ as the permutation group of the set $[\Dim]$, we have
\begin{align}
\max_{\ell_0 \leq \bm{\ell} < L}
2^{\frac32 \|\bm{\ell}\|_1 - \|\bm{\ell}\|_\infty - \widetilde\gammavec \cdot \bm{\ell}} 
& = \max_{\bm{\ell} \in \triangle} \max_{\sigma \in S_\Dim}
2^{\frac32 \|\sigma(\bm{\ell})\|_1 - \|\sigma(\bm{\ell})\|_\infty - \widetilde\gammavec \cdot \sigma(\bm{\ell})} \\
& = \max_{\bm{\ell} \in \triangle} 
2^{\frac32 \|\bm{\ell}\|_1 - \|\bm{\ell}\|_\infty }
\max_{\sigma \in S_\Dim}
2^{ - \widetilde\gammavec \cdot \sigma(\bm{\ell})}\\
& = \max_{\bm{\ell} \in \triangle} 
2^{\frac32 \|\bm{\ell}\|_1 - \ell_1} 
2^{ - \frac12 \ell_1 - \frac32 \sum_{j >1} \ell_j} = 1, \label{eq:truccopermutaz}
\end{align}
where we have exploited the identity $\{\bm{\ell} \in \mathbb{N}^\Dim : \ell_0 \leq \bm{\ell} < L\} = \bigcup_{\sigma \in S_\Dim} \sigma (\triangle)$ and the fact that $\|\bm{\ell}\|_1$ and $\|\bm{\ell}\|_\infty$ are invariant with respect to permutations of the components of $\bm{\ell}$. Combining \eqref{eq:rationormaliz} with \eqref{eq:truccopermutaz} (with $\bm{\gamma} = \widetilde{\bm{\gamma}}$), and observing that
$$
|\bm{q}|^{2\widetilde\gammavec(\bm{q}) - \bm{4}}
= |q_{j_\infty}|^{-3} \prod_{j\neq j_\infty}|q_j|^{-1}
= \frac{1}{\|\bm{q}\|_{\infty}^2|\bm{q}|^{\bm{1}}},
$$
we obtain 
\begin{equation}
\mu_\bm{q} 
\lesssim 
|\bm{q}|^{2\widetilde\gammavec(\bm{q}) - \bm{4}}
\|\bm{q}\|_2^2 \bigg(|\eta|^2 +  \frac{\|\bm{\beta}\|_2^2}{\|\bm{q}\|_2^2} + \frac{|\rho|^2}{\|\bm{q}\|_2^4}
\bigg)
= 
\frac{ \|\bm{q}\|_2^2}{\|\bm{q}\|_\infty^2|\bm{q}|^{\onevec} }
\bigg(|\eta|^2 +  \frac{\|\bm{\beta}\|_2^2}{\|\bm{q}\|_2^2} + \frac{|\rho|^2}{\|\bm{q}\|_2^4}
\bigg),
\end{equation}
where the right hand side does not depend on $\bm{\ell}$. This relation corresponds to the second argument of the minimum in \eqref{eq:mu_UB_multid} when $\|\bm{q}\|_0 = \Dim$. The case $\|\bm{q}\|_0 =\Dim$ is hence concluded.

\paragraph{Case $0 < \|\bm{q}\|_0 < \Dim$.}  The argument is analogous to the case $\|\bm{q}\|_0 = \Dim$. We just need to replace $[\Dim]$ with $\supp(\bm{q})$, $\Dim$ with $\|\bm{q}\|_0$, $\bm{q}$ with $\widehat{\bm{q}}$, $\bm{\ell}$ with $\widehat{\bm{\ell}}$, and $\bm{\gamma}$ with $\widehat{\bm{\gamma}}$. Moreover, notice that $\|\widehat{\bm{q}}\|_2 = \|\bm{q}\|_2$, $\|\widehat{\bm{q}}\|_\infty = \|\bm{q}\|_\infty$, and $\|\widehat{\bm{\ell}}\|_{\infty} \leq \|\bm{\ell}\|_{\infty}$ (note that the last relation is not an equality since $\widehat{\bm{\ell}} = \bm{\ell}|_{\supp(\bm{q})}$). This concludes the proof.
\end{proof}

\begin{rmrk}[Consistency with the 1D case.]
The upper bounds of Theorem~\ref{thm:mu_AReq1D} are compatible with those in Theorem~\ref{thm:mu_bound_multi_ani}. Indeed, when the coefficients are constant, they all belong to $H^1_\per(\mathcal{D})$ and it is immediate to verify that \eqref{eq:mu_est_1D_2} coincides with \eqref{eq:mu_UB_multid} when $\Dim = 1$. Moreover, \eqref{eq:mu_UB_multid_q0} with $n=1$ yields $\mu_0 \lesssim |\rho|^2 2^{-3\ell_0}$, which is sharper than the upper bound $\mu_0 \lesssim |\rho|^2 2^{-2\ell_0}$ implied by Theorem~\ref{thm:mu_AReq1D} when the coefficients are constant. This small discrepancy is due to the upper bound $|(\rho\psi_{\ell,k},\xi_0)|\leq \|\rho\|_{L^2(\mathcal{D})}\|\psi_{\ell,k}\|_{L^2(\mathcal{D})}$ employed in the proof of Theorem~\ref{thm:mu_AReq1D}, which is not sharp when $\rho$ is constant. \hfill $\blacksquare$
\end{rmrk}

The local $a$-coherence estimates of Theorem~\ref{thm:mu_bound_multi_ani} answer issues (i), (ii), and (iii) in Section~\ref{sec:i,ii,iii} and lead to the following

\begin{thm}[{\corsing \WF recovery, anisotropic wavelets}] 
\label{thm:CORSING_rec_multi_ani}
In Setting~\ref{ass:Bsplwave} and under the same assumptions as in Theorem~\ref{thm:CORSINGrecovery}, let $\Dim > 1$ and $\eta, \rho \in \mathbb{R}$ and $\bm{\beta} \in \mathbb{R}^\Dim$. Then, provided 
\begin{itemize}
\item[\textup{(i)}] $R  \sim C s N^{3-\frac{2}{\Dim}}$,
\item[\textup{(ii)}] $m  \gtrsim C 2^{\Dim\ell_0} s  (s \ln(eN/(2s)) + \ln(2 s/\varepsilon)) (\ln N + \ln s +\ln C)^\Dim$,
\item[\textup{(iii)}] $p_{\bm{q}}  \propto \begin{cases}
2^{-(2 + \Dim)\ell_0}, & \bm{q} = \bm{0},\\
2^{-(\Dim-\|\bm{q}\|_0)\ell_0} \min\Big\{\frac{2^{(3\|\bm{q}\|_0 - 2)L}\|\bm{q}\|_2^2}{|\widehat{\bm{q}}|^{\bm{4}}}, \frac{\|\bm{q}\|_2^2}{\|\bm{q}\|_\infty^2 |\widehat{\bm{q}}|^{\bm{1}}}\Big\}, & \bm{q} \neq \bm{0},
\end{cases}$
\end{itemize}
where $N = 2^{nL}$ and $C = |\eta|^2 + \|\bm{\beta}\|_2^2 + |\rho|^2$, the \corsing \WF method with $\Psi = \Psi^\ani$ recovers the best $s$-term approximation to $u$ in expectation in the sense of estimate \eqref{eq:corsingrecovery}.
\end{thm}

\begin{proof}
Let us consider the upper bound $\bm{\nu}$ defined according to   \eqref{eq:mu_UB_multid_q0} and \eqref{eq:mu_UB_multid}. The definition of $p_{\bm{q}}$ in (iii) directly follows from the definition of $\nu_{\bm{q}}$. Now, we derive condition (i) by estimating the tail $\|\bm{\nu}|_{\mathcal{Q}^c}\|_1$ and using the upper bound corresponding to the first argument of the minimum in \eqref{eq:mu_UB_multid}. Letting $Q := \lfloor R/2 \rfloor$ and splitting the sum involved in the 1-norm with respect to the  sparsity levels of $\bm{q}$, we obtain 
\begin{align}
\|\bm{\nu}|_{\mathcal{Q}^c}\|_1 
& \leq \sum_{\bm{q} \in \mathbb{Z}^\Dim :\|\bm{q}\|_\infty > Q} \nu_\bm{q}
 = 
\sum_{s = 1}^\Dim  
\;
\sum_{\bm{q} \in \mathbb{Z}^\Dim : \|\bm{q}\|_\infty > Q, \; \|\bm{q}\|_0 = s} 
\nu_\bm{q} \\
& \leq 
C \sum_{s = 1}^\Dim  
\;
\sum_{\bm{q} \in \mathbb{Z}^\Dim : \|\bm{q}\|_\infty > Q, \; \|\bm{q}\|_0 = s} 
2^{-(\Dim-s)\ell_0} \frac{2^{(3s-2)L}\|\bm{q}\|_2^2}{|\widehat{\bm{q}}|^{\bm{4}}} \\
& \label{eq:trunc_cond_multid_proof}
= 
C \sum_{s = 1}^\Dim  {{\Dim}\choose{s}} 2^{-(\Dim-s)\ell_0 + (3s-2)L}
\underbrace{\sum_{\bm{r} \in \mathbb{Z}^s :\|\bm{r}\|_\infty > Q, \,  \|\bm{r}\|_0 = s} \frac{\|\bm{r}\|_2^2}{|\bm{r}|^\bm{4}}}_{=:T(s)}.
\end{align}
Now, we analyze $T(s)$. It is easy to verify that, for every $s \in [\Dim]$, it holds
\begin{equation}
\label{eq:defXkt}
\{\bm{r} \in \ZZ^s: \|\bm{r}\|_\infty > Q, \; \|\bm{r}\|_0 = s\} 
= 
\bigcup_{k = 1}^s \; \bigcup_{\bm{t}\in\{ -1,1\}^s} 
\underbrace{\{\bm{r} \in \ZZ^s: |r_k| > Q, \; \sign(\bm{r}) = \bm{t} \}}_{=:X_k^{\bm{t}} },
\end{equation}
where $X_k^{\bm{t}}\subseteq \mathbb{Z}^s$ is the set of multi-indices having the $k^{\text{th}}$ component larger than $Q$ and with sign pattern $\bm{t}$ (see Figure~\ref{fig:setXprooflocalacohe} for $s = 2$ and $Q=2$).
\begin{figure}
\centering
\includegraphics[width = 8cm]{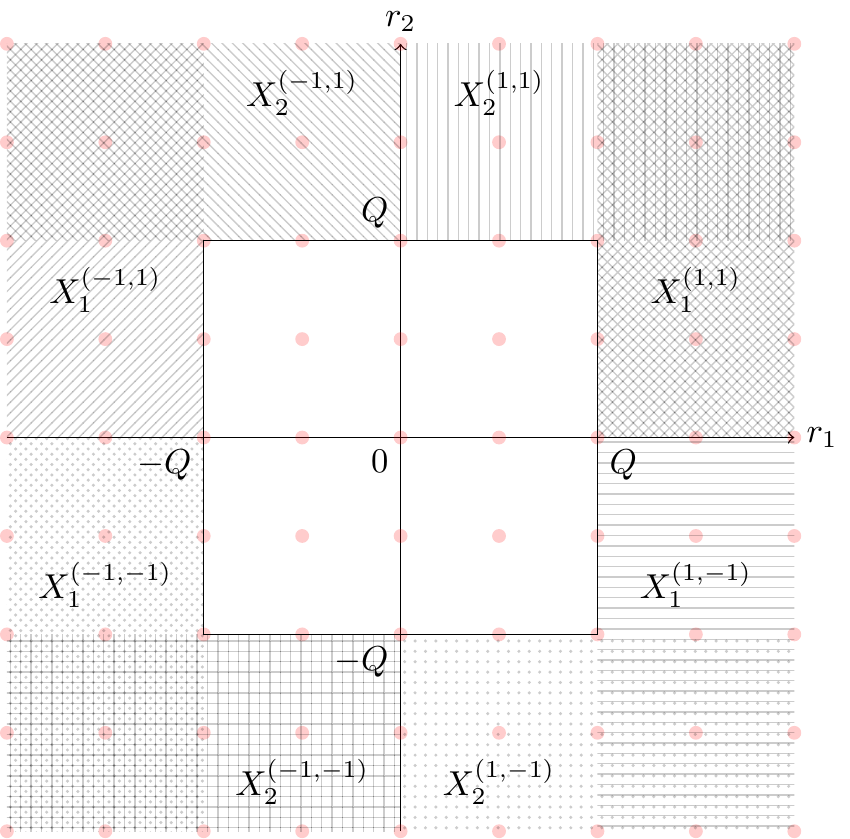}
\caption{\label{fig:setXprooflocalacohe}The sets $X_k^{\bm{t}}\subseteq\mathbb{Z}^s$  in \eqref{eq:defXkt} for $s = 2$ and $Q=2$, restricted to $[-4,4]^2$. Different textures correspond to different sets.}
\end{figure}
Since the function $\bm{r} \mapsto \|\bm{r}\|_2^2/|\bm{r}|^\bm{4}$ is invariant with respect to $\sign(\bm{r})$ and to permutations of the components of $\bm{r}$, we can just consider the set $X_1^{\onevec}$. Moreover, the number of possible sets $X_k^\bm{t}$ in $\ZZ^s$ is $2^s s$. Therefore, we estimate
\begin{align}
T(s) & \leq 2^s s \sum_{\bm{r} \in X_1^{\onevec}} \frac{\|\bm{r}\|_2^2}{|\bm{r}|^4} 
%& = 2^s s \sum_{q_1 >R} 
%\sum_{q_2 > 0} \cdots \sum_{q_s > 0}
%\frac{q_1^2 + \cdots + q_s^2}{q_1^4 \cdots q_s^4}\\
 = 2^s s \sum_{r_1 > Q} 
\sum_{r_2 > 0} \cdots \sum_{r_s > 0}
\sum_{k=1}^s \frac{1}{|\bm{r}|^2}\prod_{i \neq k}\frac{1}{r_i^2}\\
& = 2^s s
\bigg(
\bigg(\sum_{r_1 >Q} \frac{1}{r_1^2} \bigg)
\prod_{i \neq 1}\bigg(\sum_{r_i>0}\frac{1}{r_i^4}\bigg)
+
\sum_{k = 2}^s 
\bigg(\sum_{r_1>Q}\frac{1}{r_1^4}\bigg)
\bigg(\sum_{r_k>0}\frac{1}{r_k^2}\bigg)
\prod_{i \notin\{1,k\}}\bigg(\sum_{r_i>0}\frac{1}{r_i^4}\bigg)
\bigg)\\
& \lesssim 2^s s 
\bigg(\frac{c^{s-1}}{Q} + s \frac{c^{s-2}}{Q^3}
\bigg)
\lesssim \frac{1}{Q},
\end{align}
where $c := \sum_{k \in \mathbb{N}}1/k^4 < \infty$  and where the last inequality involves a constant depending exponentially on $s$ and hence on $\Dim$ (since $s \leq n$). Plugging the  estimate for $T(s)$ above into \eqref{eq:trunc_cond_multid_proof} yields
\begin{align}
\|\bm{\nu}|_{\mathcal{Q}^c}\|_1
& \lesssim 
\frac{C}{Q}
\sum_{s = 1}^\Dim  {{\Dim}\choose{s}} 2^{-(\Dim-s)\ell_0 + (3s-2)L} 
 = C\frac{2^{-\Dim\ell_0-2L}}{\lfloor R/2 \rfloor}
\sum_{s = 1}^\Dim  {{\Dim}\choose{s}} 2^{(\ell_0 + 3L)s} \\
& \lesssim C\frac{2^{-\Dim\ell_0-2L}}{R} 
2^{\Dim(\ell_0 + 3L)} 
=
C\frac{2^{\Dim L (3 -\frac2\Dim)}}{R}.
\end{align}
Condition (i) is obtained from $s\|\bm{\nu}|_{\mathcal{Q}^c}\|_1 \lesssim 1$. Of course, the above inequality also  shows implicitly that $\|\bm{\nu}\|_1 < +\infty$.

Finally, in order to prove (ii), we estimate $\|\bm{\nu}|_{\mathcal{Q}}\|_1$. Employing the upper bound corresponding to \eqref{eq:mu_UB_multid_q0} and to the second argument of the minimum in \eqref{eq:mu_UB_multid}, and recalling that  $\|\bm{q}\|_2^2 \leq \|\bm{q}\|_0 \|\bm{q}\|_\infty^2$ for every $\bm{q} \in \mathbb{Z}^\Dim$, we obtain
\begin{align}
\|\bm{\nu}|_{\mathcal{Q}}\|_1 
& \leq 
C \bigg( 2^{-(2+\Dim)\ell_0} + \sum_{s = 1}^\Dim 2^{-(\Dim-s)\ell_0}
\sum_{\bm{q} \in \mathbb{Z}^\Dim : \|\bm{q}\|_{\infty}\leq Q, \; \|\bm{q}\|_0 = s} \frac{\|\bm{q}\|_2^2}{\|\bm{q}\|_\infty^2 |\widehat{\bm{q}}|^{\onevec}}\bigg)\\
& \lesssim
C 
\sum_{s = 1}^\Dim 2^{-(\Dim-s)\ell_0} 
\sum_{\bm{q} \in \mathbb{Z}^\Dim :\|\bm{q}\|_{\infty}\leq Q, \; \|\bm{q}\|_0 = s} \frac{s}{|\widehat{\bm{q}}|^{\onevec}}
 =
C 
\sum_{s = 1}^\Dim {{\Dim}\choose{s}} 2^{-(\Dim-s)\ell_0} s
\sum_{\bm{r}\in\ZZ^s: \|\bm{r}\|_\infty\leq Q, \; \|\bm{r}\|_0 = s  } \frac{1}{|\bm{r}|^{\onevec}}\\
&  = 
C
\sum_{s = 1}^\Dim{{\Dim}\choose{s}} 2^{-(\Dim-s)\ell_0} s
\bigg(\sum_{|q| \leq Q, \; q \neq 0} \frac{1}{|q|}\bigg)^s 
\lesssim
2^{-n\ell_0} C \sum_{s = 1}^\Dim{{\Dim}\choose{s}} 2^{s\ell_0} s
(\log R)^s 
\lesssim C (\log R)^\Dim,
\end{align}
which proves the theorem.
\end{proof}

\subsubsection{Isotropic tensor product wavelets}
\label{sec:iso}

As in the previous sections, we provide local $a$-coherence upper bounds  (Theorem~\ref{thm:mu_bound_multi_iso}) and a recovery result (Theorem~\ref{thm:CORSING_rec_multi_iso}) for the \corsing \WF method.

We start by proving auxiliary inequalities analogous to those in Lemma~\ref{lem:aux_ineq_multid_ani}. We skip the proof, which follows the same arguments as in Lemma~\ref{lem:aux_ineq_multid_ani}.
\begin{lem}[Auxiliary inequalities, isotropic wavelets]
\label{lem:aux_ineq_multid_iso}
In Setting~\ref{ass:Bsplwave}, let $n > 1$, $\ell \in \mathbb{N}$, with $\ell \geq \ell_0$, $\bm{k} \in (\mathbb{Z}/(2^\ell\mathbb{Z}))^\Dim$, $\bm{e} \in \{0,1\}^\Dim$, and $\bm{q} \in \ZZ^\Dim$. Moreover, define $\zero(\bm{q})$ as in \eqref{eq:def_zero_scal}. Then, the following inequalities hold:
\begin{itemize}
\item If $\bm{q} = \bm{0}$, we have
\begin{align}
\label{eq:aux_multi_iso_1}
|(\nabla \psi^{\iso}_{\ell,\bm{k},\bm{e}}, \nabla \xi_{\bm{0}})| & = 0,\\
\label{eq:aux_multi_iso_2}
|(\bm{\beta} \cdot \nabla \psi^{\iso}_{\ell,\bm{k},\bm{e}}, \xi_{\bm{0}})| & = 0, \quad \forall \bm{\beta} \in \mathbb{R}^\Dim\\
\label{eq:aux_multi_iso_3}
|(\psi^{\iso}_{\ell,\bm{k},\bm{e}}, \xi_{\bm{0}})| & = 
\begin{cases}
2^{- \Dim\ell / 2}, & \text{if } \bm{e} = \bm{0},\\
0, & \text{if } \bm{e} \neq \bm{0};
\end{cases}
\end{align}
\item if $\bm{q} \neq \bm{0}$ and $\zero(\bm{q}) \subseteq \zero(\bm{e})$, then, for every $\gammavec \in [0,2]^{\|\bm{q}\|_0}$, it holds
\begin{align}
|(\nabla \psi^{\iso}_{\ell,\bm{k},\bm{e}},\nabla \xi_{\bm{q}})| 
& \lesssim
2^{(-\frac \Dim 2 + 2 \|\bm{q}\|_0 - \|\widehat{\bm{\gamma}}\|_1)\ell} 
|\widehat{\bm{q}}|^{\widehat\gammavec -\bm{2}}
\|\bm{q}\|_2^2, \label{eq:aux_multi_iso_4}\\
|(\bm{\beta} \cdot \nabla \psi^{\iso}_{\ell,\bm{k},\bm{e}},\xi_{\bm{q}})| 
& \lesssim 
2^{(-\frac \Dim 2 + 2 \|\bm{q}\|_0 - \|\widehat{\bm{\gamma}}\|_1)\ell}
|\widehat{\bm{q}}|^{\widehat\gammavec -\bm{2}} \|\bm{\beta}\|_2 \|\bm{q}\|_2, 
\quad \forall \bm{\beta} \in \mathbb{R}^\Dim,\label{eq:aux_multi_iso_5} \\
|(\psi^{\iso}_{\ell,\bm{k},\bm{e}},\xi_{\bm{q}})| 
& \lesssim 
2^{(-\frac \Dim 2 + 2 \|\bm{q}\|_0 - \|\widehat{\bm{\gamma}}\|_1)\ell}
|\widehat{\bm{q}}|^{\widehat\gammavec -\bm{2}},\label{eq:aux_multi_iso_6}
\end{align}
where $\widehat{\bm{x}} := \bm{x}|_{\supp(\bm{q})} \in \mathbb{R}^{\|\bm{q}\|_0}$ for every $\bm{x} \in \mathbb{R}^\Dim$ and the above inequalities hide constants depending exponentially on $\Dim$;

\item if $\bm{q} \neq \bm{0}$ and $\zero(\bm{q}) \not\subseteq \zero(\bm{e})$, it holds 
\begin{align}
\label{eq:aux_multi_iso_7}
|(\nabla \psi^{\iso}_{\bm{\ell},\bm{k}}, \nabla \xi_{\bm{q}})| & = 0,\\
\label{eq:aux_multi_iso_8}
|(\bm{\beta} \cdot \nabla \psi^{\iso}_{\bm{\ell},\bm{k}}, \xi_{\bm{q}})| & = 0, \quad \forall \bm{\beta} \in \mathbb{R}^\Dim\\
\label{eq:aux_multi_iso_9}
|(\psi^{\iso}_{\bm{\ell},\bm{k}}, \xi_{\bm{q}})| & = 0.
\end{align}
\end{itemize}

\end{lem}

In the following theorem, we provide upper bounds to the local $a$-coherence for isotropic tensor product wavelets. We only outline a sketch of its proof, which is analogous to that of Theorem~\ref{thm:mu_bound_multi_ani}. 

\begin{thm}[Local $a$-coherence upper bound, $\Dim > 1$, isotropic wavelets]
\label{thm:mu_bound_multi_iso}
In Setting~\ref{ass:Bsplwave}, let $\Dim > 1$. Then,   for every $\eta, \rho \in \mathbb{R}$ and $\bm{\beta} \in \mathbb{R}^\Dim$, the following upper bounds hold:
\begin{align}
\label{eq:mu_UB_multid_q0_iso}
\mu_{\bm{0}} & \lesssim |\rho|^2 2^{-(2+ \Dim)\ell_0},\\
\label{eq:mu_UB_multid_iso}
\mu_{\bm{q}} & 
\lesssim 
\bigg(|\eta|^2 +  \frac{\|\bm{\beta}\|_2^2}{\|\bm{q}\|_2^2} + \frac{|\rho|^2}{\|\bm{q}\|_2^4}
\bigg)   
\min\bigg\{\frac{(1 + 2^{2(-\frac \Dim 2 + 2\|\bm{q}\|_0 - 1)L}) \|\bm{q}\|_2^2}{ |\widehat{\bm{q}}|^{\bm{4}}}, 
\frac{2^{-(n-\|\bm{q}\|_0)\ell_0} \|\bm{q}\|_2^2}{\|\bm{q}\|_\infty^2 |\widehat{\bm{q}}|^{\bm{1}}}\bigg\}, 
\quad \forall \bm{q} \neq \bm{0},
\end{align}
where the inequalities hide constants depending exponentially on $\Dim$.
\end{thm}
\begin{proof}
The upper bound \eqref{eq:mu_UB_multid_q0_iso} to $\mu_\bm{0}$ is easy to verify. Let us consider $\bm{q} \neq \bm{0}$. Employing the auxiliary inequalities of Lemma~\ref{lem:aux_ineq_multid_iso} and using arguments analogous to those in Theorem~\ref{thm:mu_bound_multi_ani}, we obtain, for every $\bm{\gamma} \in [0,2]^\Dim$,
\begin{equation}
|a(\widehat{\psi}^\iso_\bm{j}, \widehat{\xi}_\bm{q})|^2
\lesssim
\begin{cases}
(|\eta|^2 +  \frac{\|\bm{\beta}\|_2^2}{\|\bm{q}\|_2^2} + \frac{|\rho|^2}{\|\bm{q}\|_2^4}
)2^{2 (-\frac \Dim 2 + 2\|\bm{q}\|_0 - \|\widehat{\bm{\gamma}}\|_1 -1)\ell}
|\widehat{\bm{q}}|^{2\widehat{\bm{\gamma}}-\bm{4}} \|\bm{q}\|_2^2, & \text{if } \zero(\bm{q})  \subseteq \zero(\bm{e}),\\
0 & \text{otherwise}.
\end{cases}
\end{equation}
Analogously to Theorem~\ref{thm:mu_bound_multi_ani}, we obtain the upper bound corresponding to the first argument of the minimum in \eqref{eq:mu_UB_multid_iso} by choosing $\widehat{\bm{\gamma}} = \bm{0}$ in the inequality above. Notice in particular that 
\begin{equation}
\max_{ \ell_0 \leq \ell < L} 2^{2 (-\frac \Dim 2 + 2\|\bm{q}\|_0  -1)\ell}\leq \max\{1, 2^{2 (-\frac \Dim 2 + 2\|\bm{q}\|_0  -1)L}\}
\leq 1 + 2^{2 (-\frac \Dim 2 + 2\|\bm{q}\|_0  -1)L},
\end{equation}
where, in the second expression, we have used that $2^{2(-\frac{n}{2}+\|\bm{q}\|_0-1)\ell}\leq 1$ when ${-\frac{n}{2}+\|\bm{q}\|_0-1}$ is negative.
The upper bound corresponding to the second argument of the minimum in \eqref{eq:mu_UB_multid_iso} is proved by letting $\bm{\gamma} = \widetilde{\bm{\gamma}}$, as in \eqref{eq:defgammatilde}. Note that, in this case, we have  $\|\widehat{\bm{\gamma}}\|_1 = \frac32 \|\bm{q}\|_0 -1$ and, consequently,
$$
\max_{\ell_0 \leq \ell < L} 2^{2(-\frac{n}{2}+2\|\bm{q}\|_0 -\|\widehat{\bm{\gamma}}\|_1-1)\ell}
= \max_{\ell_0 \leq \ell < L} 2^{-(n-\|\bm{q}\|_0)\ell}
\leq 2^{-(n-\|\bm{q}\|_0)\ell_0}.
$$ 
This concludes the proof.
\end{proof}

Finally, we obtain the \corsing \WF recovery theorem for isotropic tensor product wavelets solving issues (i), (ii), and (iii) in Section~\ref{sec:i,ii,iii}. The proof is analogous to that of Theorem~\ref{thm:CORSING_rec_multi_ani} and therefore will  be omitted.

\begin{thm}[{\corsing \WF recovery,  isotropic wavelets}] 
\label{thm:CORSING_rec_multi_iso}
In Setting~\ref{ass:Bsplwave} and under the same assumptions as in Theorem~\ref{thm:CORSINGrecovery}, let $\Dim > 1$ and let $\eta, \rho \in \mathbb{R}$ and $\bm{\beta} \in \mathbb{R}^\Dim$. Then, provided 
\begin{itemize}
\item[\textup{(i)}] $R  \sim C s N^{3-\frac{2}{\Dim}}$,
\item[\textup{(ii)}] $m  \gtrsim C 2^{\Dim\ell_0} s  (s \ln(eN/(2s)) + \ln(2s/\varepsilon)) (\ln N + \ln s +\ln C)^\Dim$,
\item[\textup{(iii)}] $p_{\bm{q}}  \propto \begin{cases}
2^{-(2+n)\ell_0}, & \bm{q} = \bm{0},\\
\min\Big\{\frac{(1+2^{2(-\frac \Dim 2 + 2\|\bm{q}\|_0 - 1)L})\|\bm{q}\|_2^2}{|\widehat{\bm{q}}|^{\bm{4}}}, \frac{2^{-(n-\|\bm{q}\|_0)\ell_0}\|\bm{q}\|_2^2}{\|\bm{q}\|_\infty^2 |\widehat{\bm{q}}|^{\bm{1}}}\Big\}, & \bm{q} \neq \bm{0},
\end{cases}$
\end{itemize}
where $C = |\eta|^2 + \|\bm{\beta}\|_2^2 + |\rho|^2$, the \corsing \WF method with $\Psi = \Psi^{\iso}$ recovers the best $s$-term approximation to $u$ in expectation in the sense of estimate \eqref{eq:corsingrecovery}.
\end{thm}

%%%%%%%%%%%%%%%%%%%%%%%%%%%%%%%%%%%%%%%%%%%%%%%%%%%%%%%%%%%
%%%%%%%%%%%%%%%%%%%%%%%%%%%%%%%%%%%%%%%%%%%%%%%%%%%%%%%%%%%
\section{Numerical assessment}
\label{sec:numerics}
%%%%%%%%%%%%%%%%%%%%%%%%%%%%%%%%%%%%%%%%%%%%%%%%%%%%%%%%%%%
%%%%%%%%%%%%%%%%%%%%%%%%%%%%%%%%%%%%%%%%%%%%%%%%%%%%%%%%%%%

In this section, numerically investigate the reliability and robustness of the \corsing \WF approach in different dimensions. In Section~\ref{sec:1Dnum}, we consider a 1D ADR equation with constant and nonconstant coefficients. As predicted by the theory, we show that \corsing \WF is robust and reliable also in the case of nonconstant coefficients. In Section~\ref{sec:2Dnum}, we consider the 2D case and compare the performance of isotropic and anisotropic wavelets in different case studies. Moreover, it turns out that the nonuniform subsampling strategy based on the local $a$-coherence significantly outperforms uniform random subsampling. Finally, in Section~\ref{sec:3Dnum}, we consider a 3D case.

All the numerical experiments have been performed in \textsc{Matlab$^\text{\textregistered}$} with the aid of OMP-Box for OMP \cite{ompbox,Rubinstein2008}. We have used \textsc{Matlab$^\text{\textregistered}$} R2017b version 9.3 64-bit on a MacBook Pro equipped with a 3 GHz Intel Core i7 processor and with 8 GB DDR3 RAM.
%In the spirit of reproducible research, the reader can download the \textsc{Matlab$^\text{\textregistered}$} code used to produce the figures from the GitHub repository \red{[WEB LINK]}. 

\subsection{1D case with nonconstant diffusion}
\label{sec:1Dnum}

We consider a 1D equation with $\beta = 0$, $\rho = 1$ and let the diffusion coefficient vary. In particular, we consider $\eta(x_1) \equiv 1$ and $\eta(x_1) = 1 +0.5 \sin(6 \pi x_1)$. 

\paragraph{The Petrov-Galerkin stiffness matrix $B$.} We compute the stiffness matrix $B$ associated with the Petrov-Galerkin discretization  and show the absolute value of the entries in Figure~\ref{fig:plot|B|} (top row).
\begin{figure}
\centering
\includegraphics[width = 7.5cm]{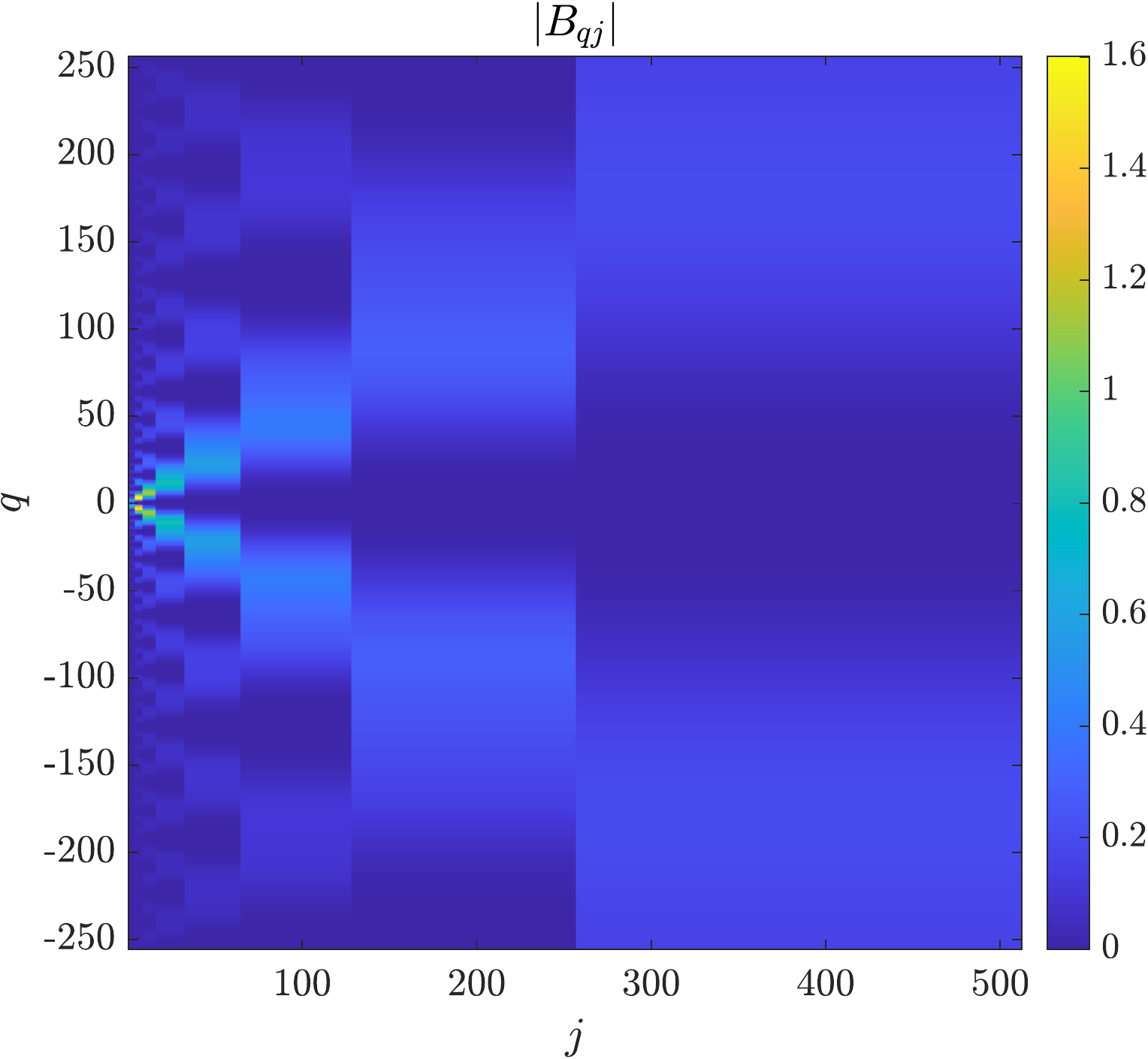}
\includegraphics[width = 7.5cm]{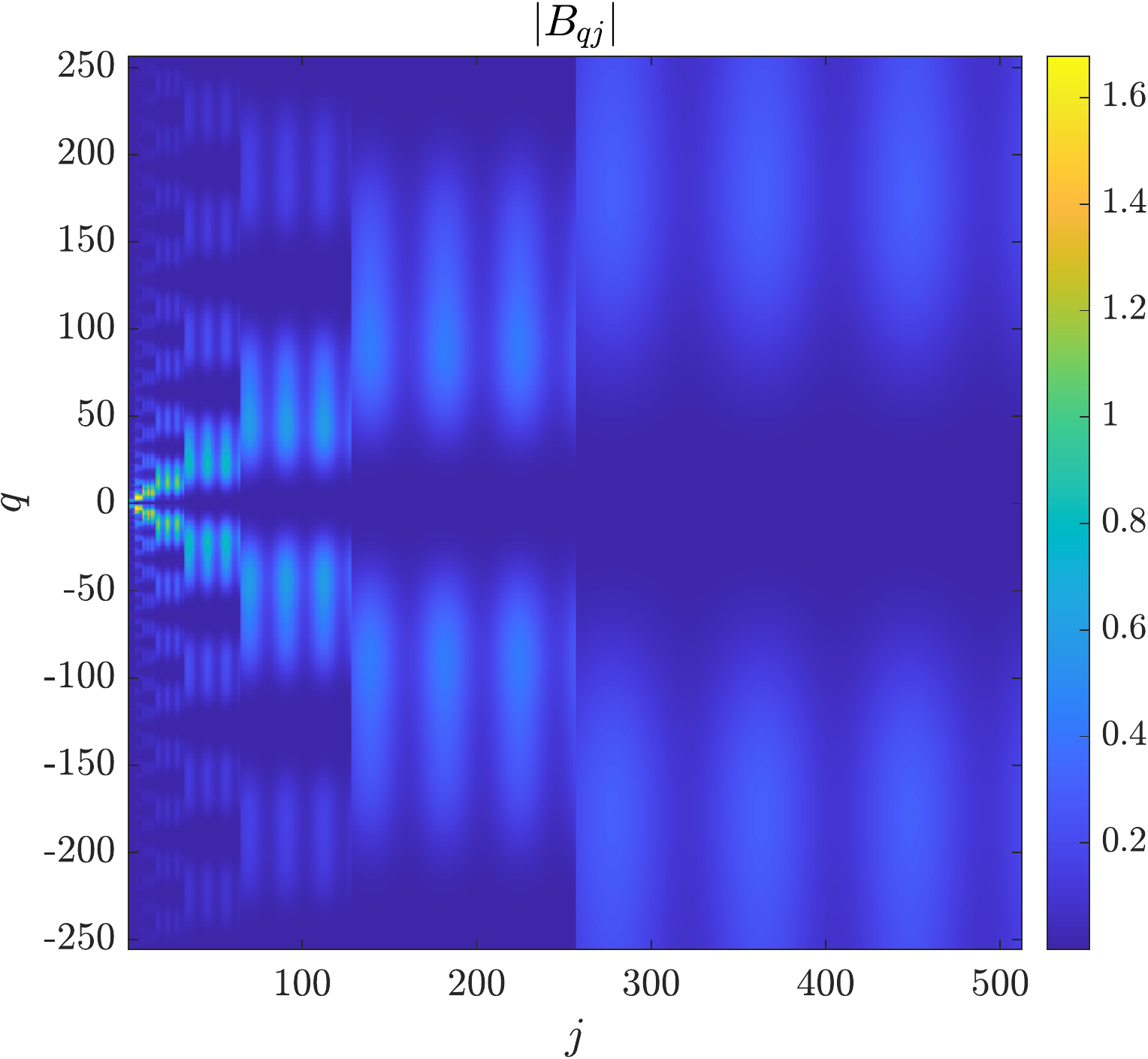}

\vspace{0.25cm}

\includegraphics[width = 7.5cm]{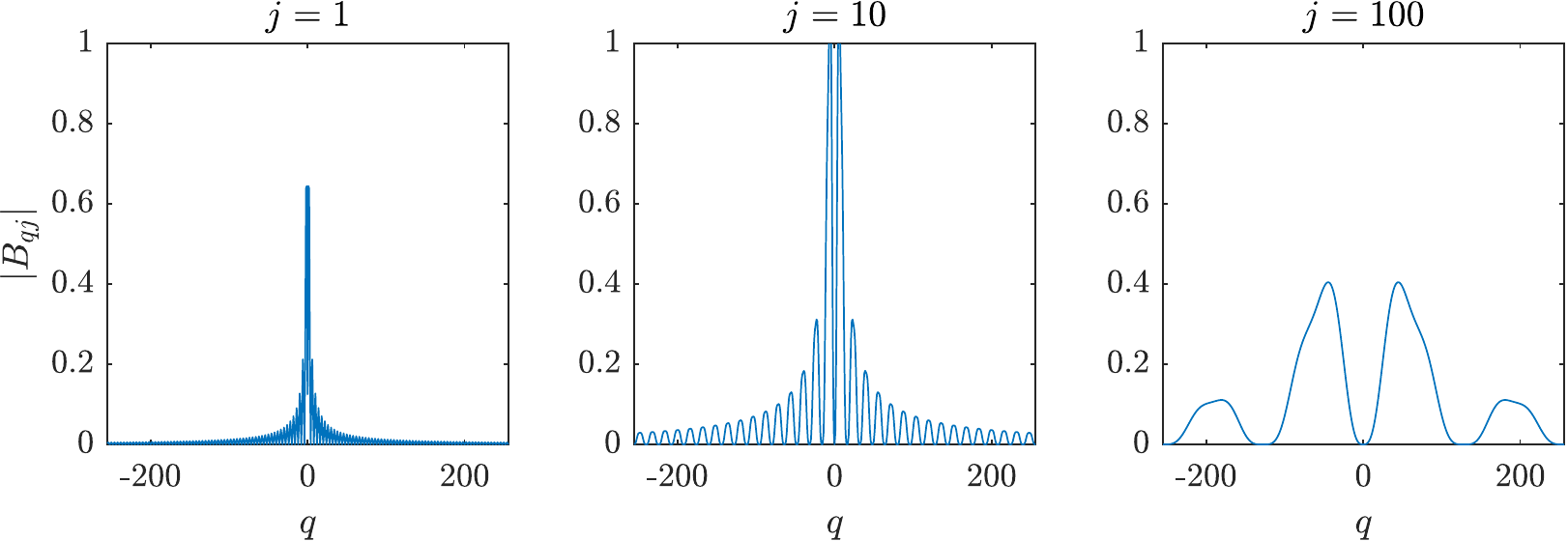}
\includegraphics[width = 7.5cm]{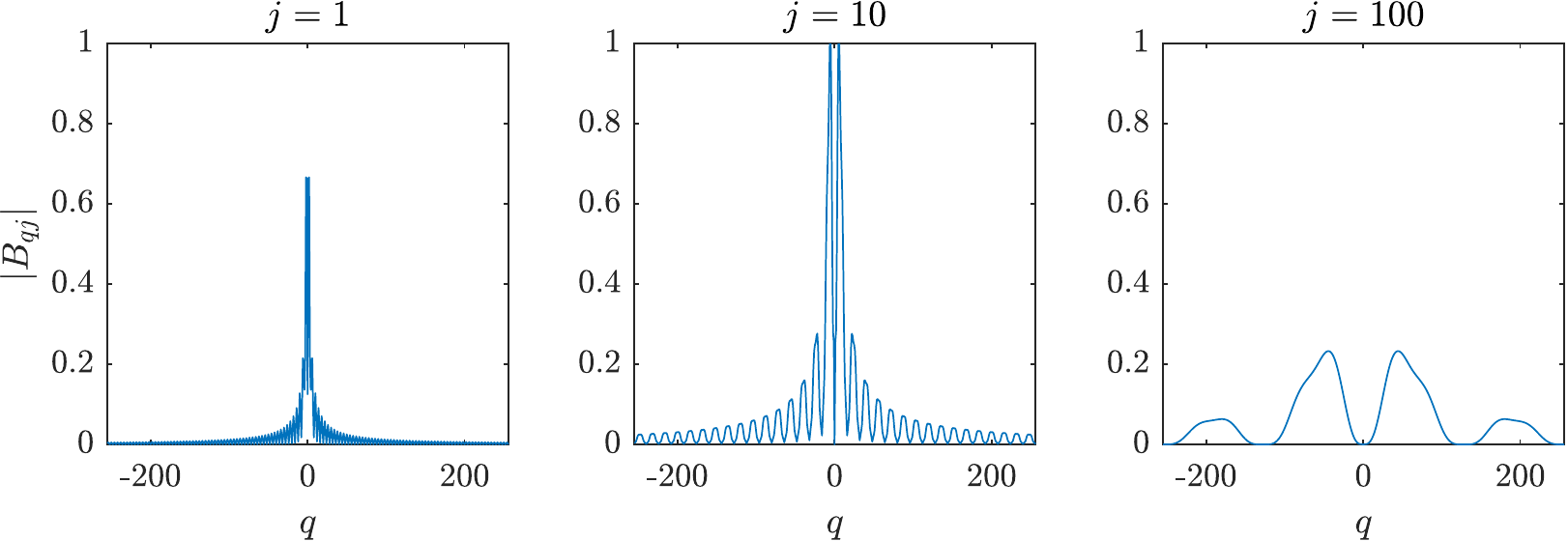}
\caption{\label{fig:plot|B|}(1D DR problem) First row: Absolute value of the entries $|B_{q,j}|$ of the stiffness matrix associated with the Petrov-Galerkin discretization for  $\eta \equiv 1$ (left) and $\eta = 1 + 0.5 \sin(6 \pi x)$ (right). Second row: Three vertical slices of the plots in the first row, i.e., plot of $|B_{q,j}|$ as a function of $q$, for $j = 1, 10, 100$.}
\end{figure} 
We set $L = 9$, resulting in $N = 512$, and choose $R = N$. We observe that the oscillations of the diffusion coefficient only impact the stiffness matrix ``horizontally''. In particular, this does not impact the decay properties of the local $a$-coherence $\bm{\mu}$ (see Figure~\ref{fig:plot|B|}, bottom row). The Wavelet-Fourier Petrov-Galerkin discretization of the ADR equation gives rise to a matrix with a comparable structure with respect to the matrices in \cite[Figure 4]{BreakingBarrier}. This qualitatively confirms that the proposed discretization is suitable for compressed sensing. We also point out that the condition number of $B\in\mathbb{C}^{512\times 512}$ is very small, being $20.4$ for the constant diffusion case and $28.2$ for the nonconstant diffusion case, to be compared  with $10^6$ and $1.6 \cdot 10^6$, respectively, when the trial and test functions are not normalized with respect to the $H^1(\mathcal{D})$-norm.

\paragraph{Best $s$-term approximation.} 
We consider the synthetic solution
\begin{equation}
\label{eq:1D_exa_sol}
u_1(x_1) = 1+\exp\left(-\frac{(x_1-0.3)^2}{0.0005}\right) + \frac12 \cos(2 \pi x_1), \quad 0 \leq x_1 <1.
\end{equation}
This solution is smooth, with a global support in $[0,1)$, and exhibits a bump close to the point $x=0.3$. Moreover, it is periodic up to machine precision (see Figure~\ref{fig:1D_fun} (left)). In Figure~\ref{fig:1D_fun} (right), we show the wavelet coefficients and highlight in red the largest 50 ones in absolute value.
\begin{figure}[t]
\centering
\includegraphics[width = 7.5cm]{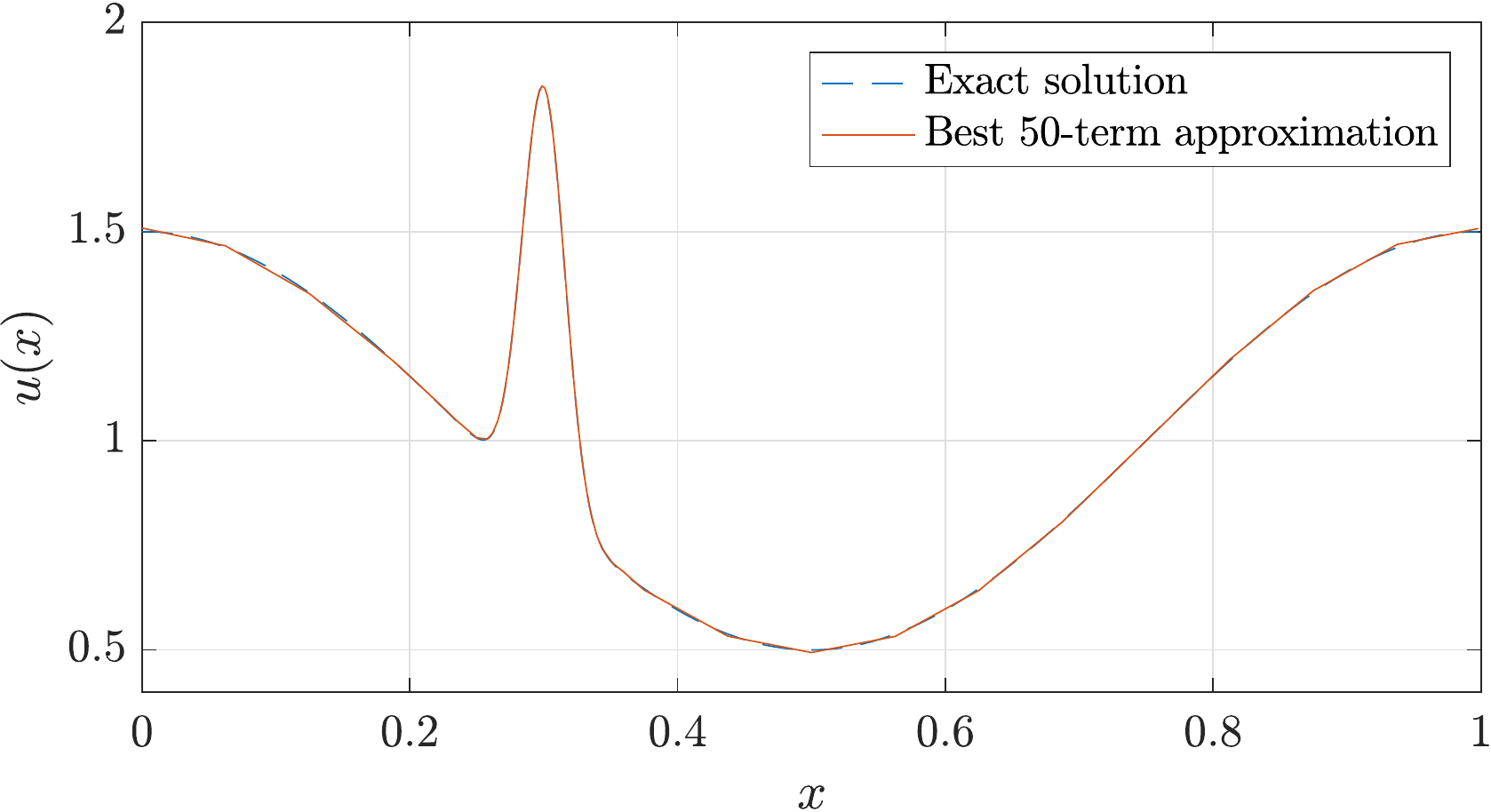}
\includegraphics[width = 7.5cm]{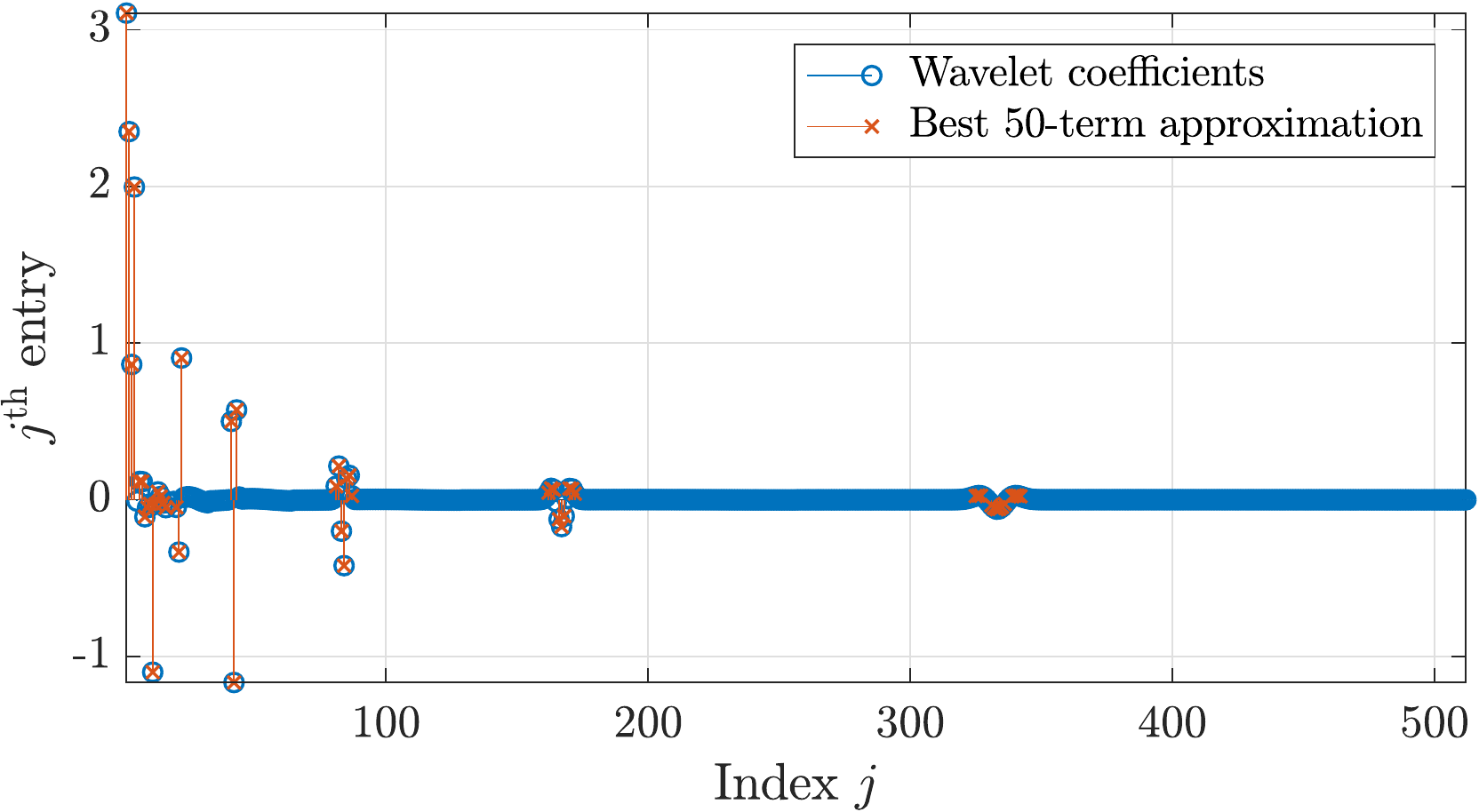}
\caption{\label{fig:1D_fun}(1D DR problem) Plot of $u_1$ and of its best $50$-term approximation $\widetilde{u}_1$ (left). Wavelet coefficients of $u$ with the 50 largest in magnitude highlighted (right).}
\end{figure} 
The resulting relative best 50-term approximation error with respect to the $H^1(\mathcal{D})$-norm is 
\begin{equation}
\label{eq:H1errorapprox}
\frac{\|u_1-\widetilde{u}_1\|_{H^1(\mathcal{D})}}{\|u_1\|_{H^1(\mathcal{D})}} \sim \frac{\|\bm{u}_1-\widetilde{\bm{u}}_{1}\|_2}{\|\bm{u}_1\|_2} = 1.68\cdot 10^{-2},
\end{equation}
where $\bm{u}_1$ is the vector of coefficients of $u_1$ with respect to the biorthogonal wavelet basis $\widehat\Psi$ (normalized with respect to the $H^1(\mathcal{D})$-norm), $\widetilde{\bm{u}}_{1}$ is the best 50-term approximation of $\bm{u}_1$, and $\widetilde{u}_{1}$ is the function corresponding to the wavelet coefficients in $\widetilde{\bm{u}}_{1}$.

\paragraph{Sensitivity of the recovery error to the number of test functions.}
In Figure~\ref{fig:1D_m_vs_err}, we show the box plot of the relative error between $u_1$ and the \corsing approximation, $\widehat{u}_1$, with respect to the $H^1(\mathcal{D})$-norm as a function of the number $m$ of test functions. 
\begin{figure}[t]
\centering
\includegraphics[height = 4.2cm]{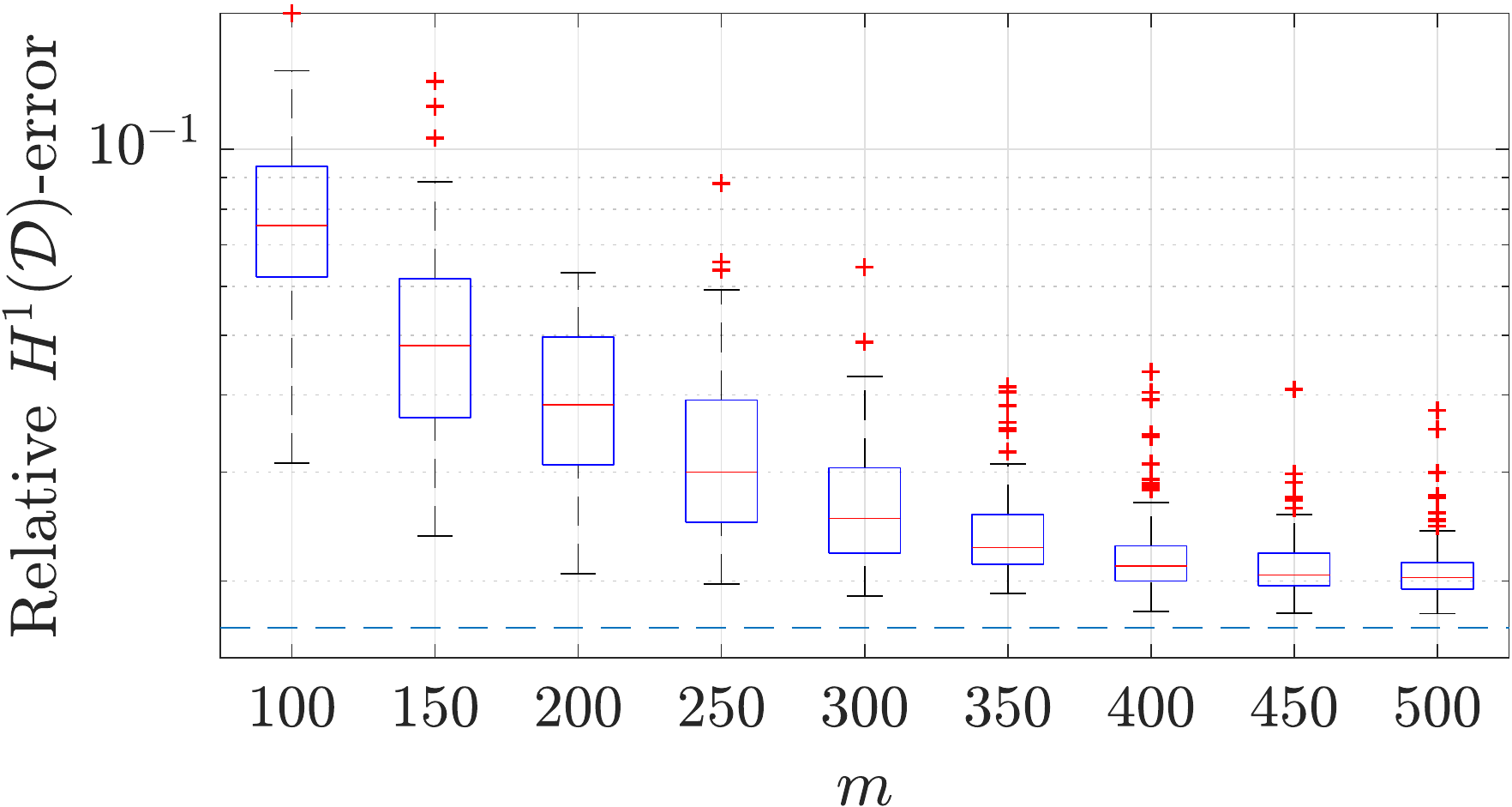}
\includegraphics[height = 4.2cm]{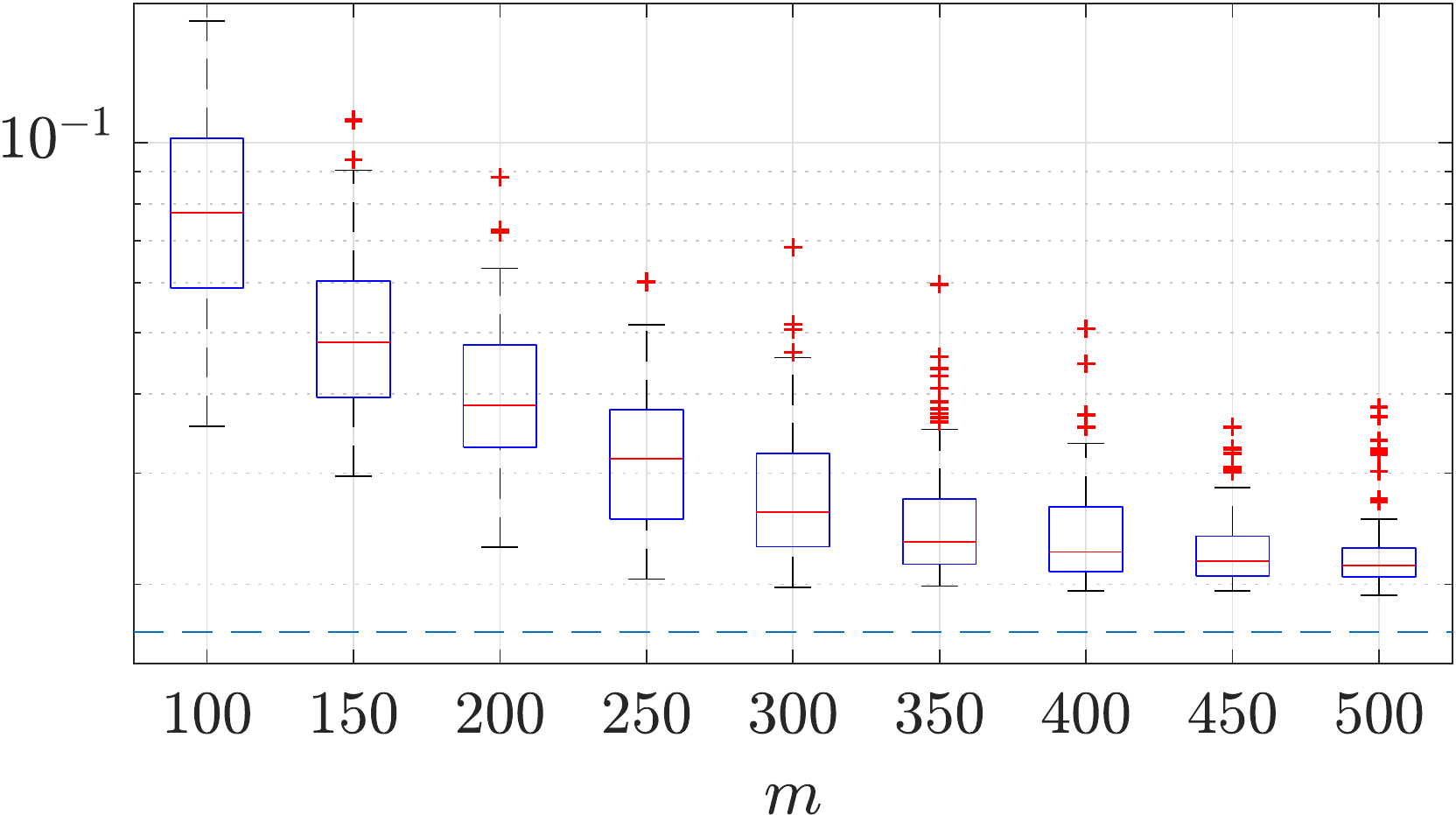}
\caption{\label{fig:1D_m_vs_err}(1D DR problem) Relative recovery error as a function of the number of random tests, $m$, for a constant (left) and nonconstant (right) diffusion term. In dashed line, the relative best 50-term approximation error \eqref{eq:H1errorapprox}.}
\end{figure}
We fix $s = 50$ and let $m$ vary from $100$ to $500$ ($N = 512$). The data are relative to 100 random runs of the \corsing procedure. We can appreciate that for both choices of $\eta$ \corsing is able to reach a good accuracy (less than twice the best $50$-term approximation error) for $m \geq 250$. The presence of a nonconstant diffusion term does not impact the performance of the method to a substantial extent. We only observe more outliers for the nonconstant diffusion.

\paragraph{Sensitivity of the recovery error to the sparsity.}
In Figure~\ref{fig:1D_s_vs_err}, we show the relative \textsf{CORSING} error with respect to the $H^1(\mathcal{D})$-norm as a function of $s$ and compare it with the best $s$-term approximation error. 
\begin{figure}
\centering
\includegraphics[height = 4.2cm]{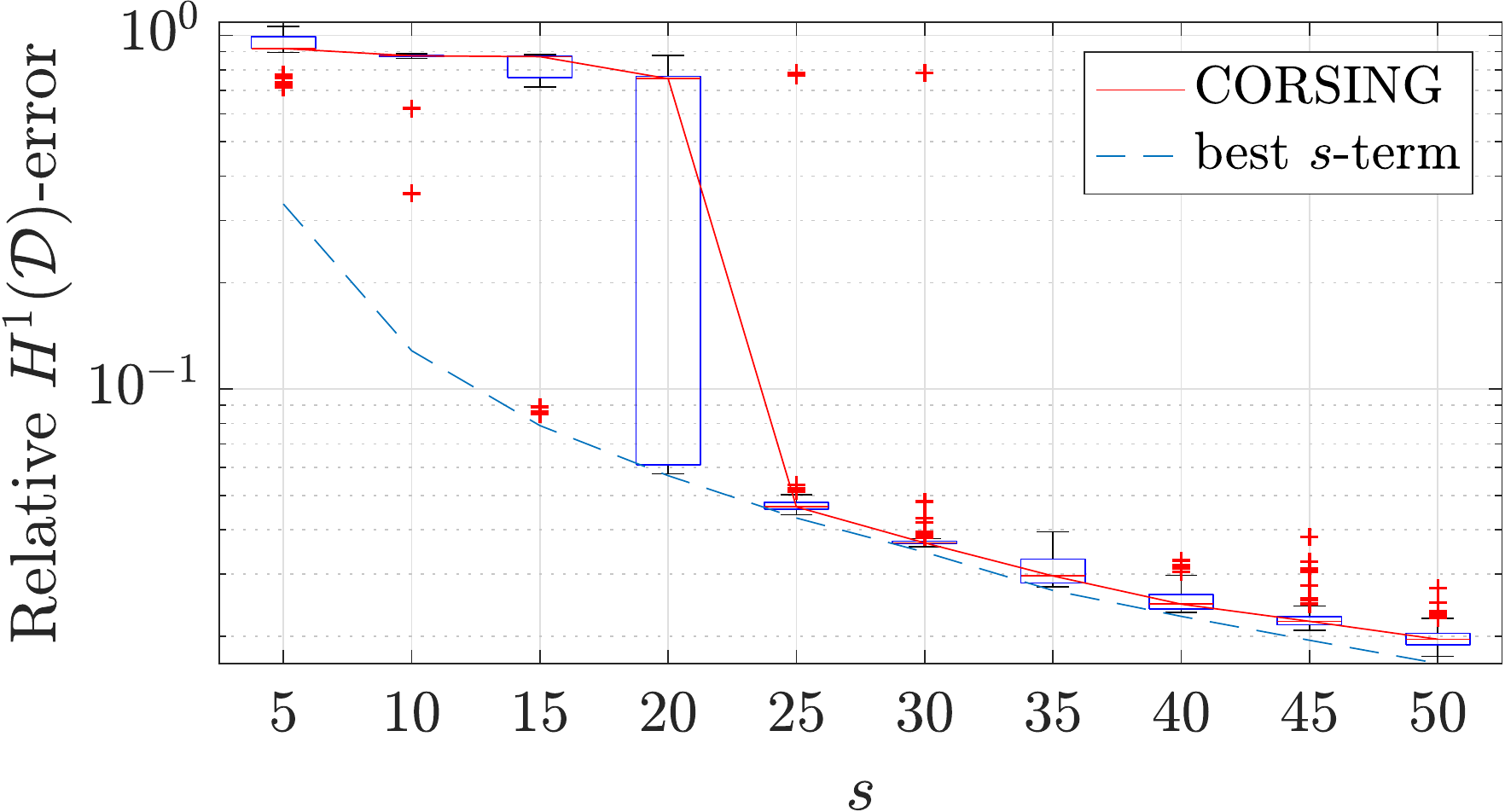}
\includegraphics[height = 4.2cm]{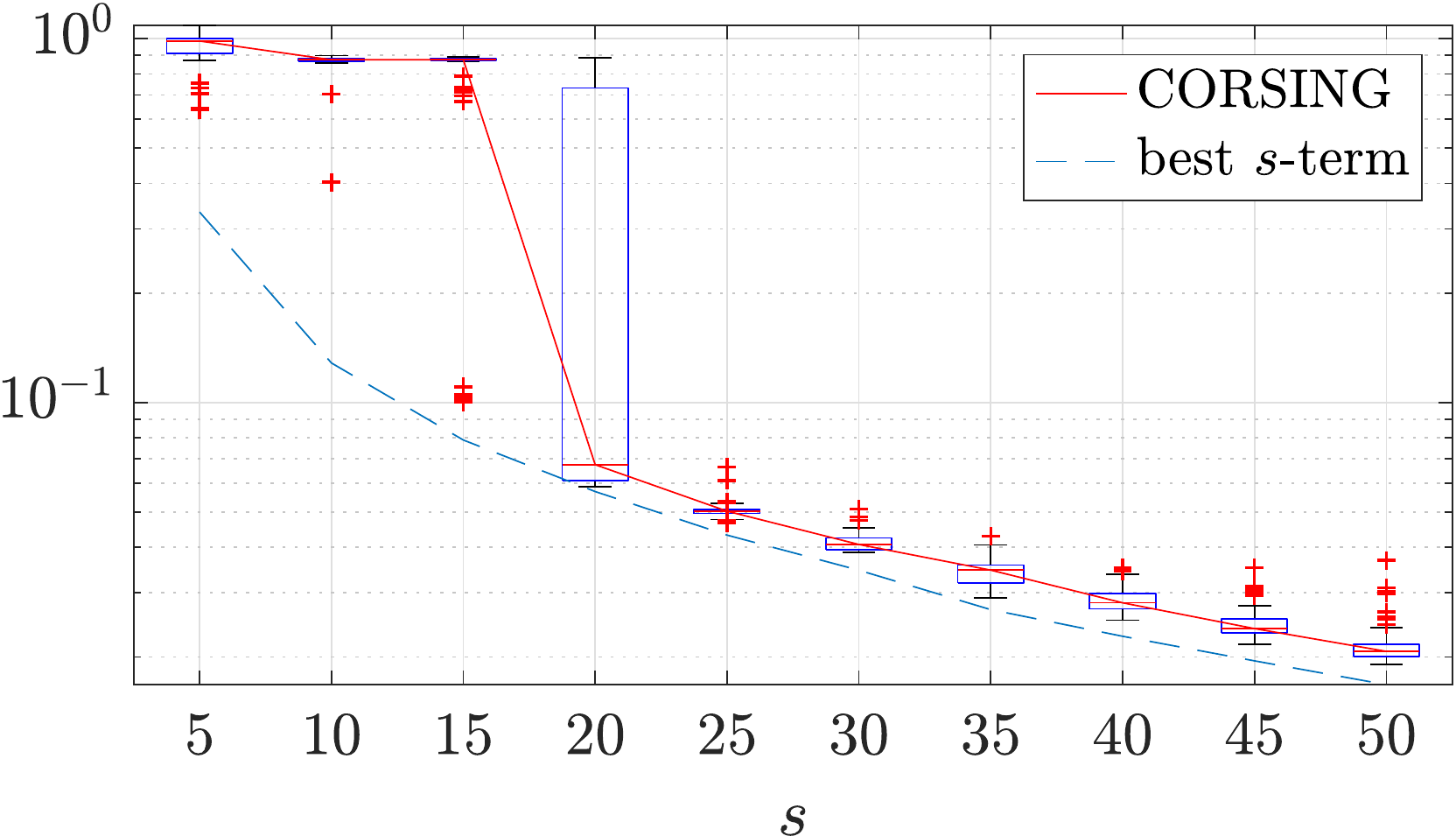}
\caption{\label{fig:1D_s_vs_err}(1D DR problem) Relative recovery and best $s$-term approximation errors as functions of the sparsity $s$ with constant (left) and nonconstant (right) diffusion term.}
\end{figure}
The box plot is relative to 100 runs of \corsing. For each value of $s$ varying between 5 and 50, we set $m = \lceil 2 s \log(N)\rceil$. We remark that, for $s$ large enough, the recovery error exhibits the same decay rate as the best $s$-term approximation error. No striking difference can be detected varying the diffusion.

\subsection{2D case}
\label{sec:2Dnum}
We consider a 2D ADR problem over $\mathcal{D}=(0,1)^2$ with constant coefficients $\mu = \rho = 1$, and $\bm{b} = [1,1]^T$. 

On the one hand, we compare the performance of anisotropic and isotropic wavelets on solutions that exhibit different spatial features. On the other hand, we show that nonuniform sampling strategy based on local $a$-coherence outperforms the uniform random subsampling.

\paragraph{Wavelet coefficients and best $s$-term approximation error.} 
We consider the following solutions:
\begin{align}
\label{eq:def_u2}
u_2(x_1,x_2) & = \exp\left(-\frac{(x_1-0.3)^2}{0.0005}\right)\exp\left(-\frac{(x_2-0.4)^2}{0.0005}\right) + 2\exp\left(-\frac{(x_1-0.6).^2}{0.001}\right)\exp\left(-\frac{(x_2-0.5)^2}{0.005}\right),\\
\label{eq:def_u3}
u_3(x_1,x_2) & =\exp\left(-\frac{(x_1-0.45)^2}{0.005}\right),
\end{align}
both periodic up to machine precision. The function $u_2$ exhibits two local Gaussian-shaped features, one isotropic around the point $(0.3,0.4)$ and the other anisotropic, around $(0.6,0.5)$. The function $u_3$ is purely anisotropic, having a Gaussian behavior along the $x_1$-direction and being constant along the $x_2$-direction. The functions $u_2$ and $u_3$ are shown in Figure~\ref{Figure10} along with the corresponding anisotropic and isotropic $64 \times 64$ wavelet coefficients (with respect to $H^1(\mathcal{D})$-normalized wavelets).
\begin{figure}[t]
\centering
\begin{tabular}{cccc}
\raisebox{1.5cm}{$u_2$} &
\includegraphics[height = 3.7cm]{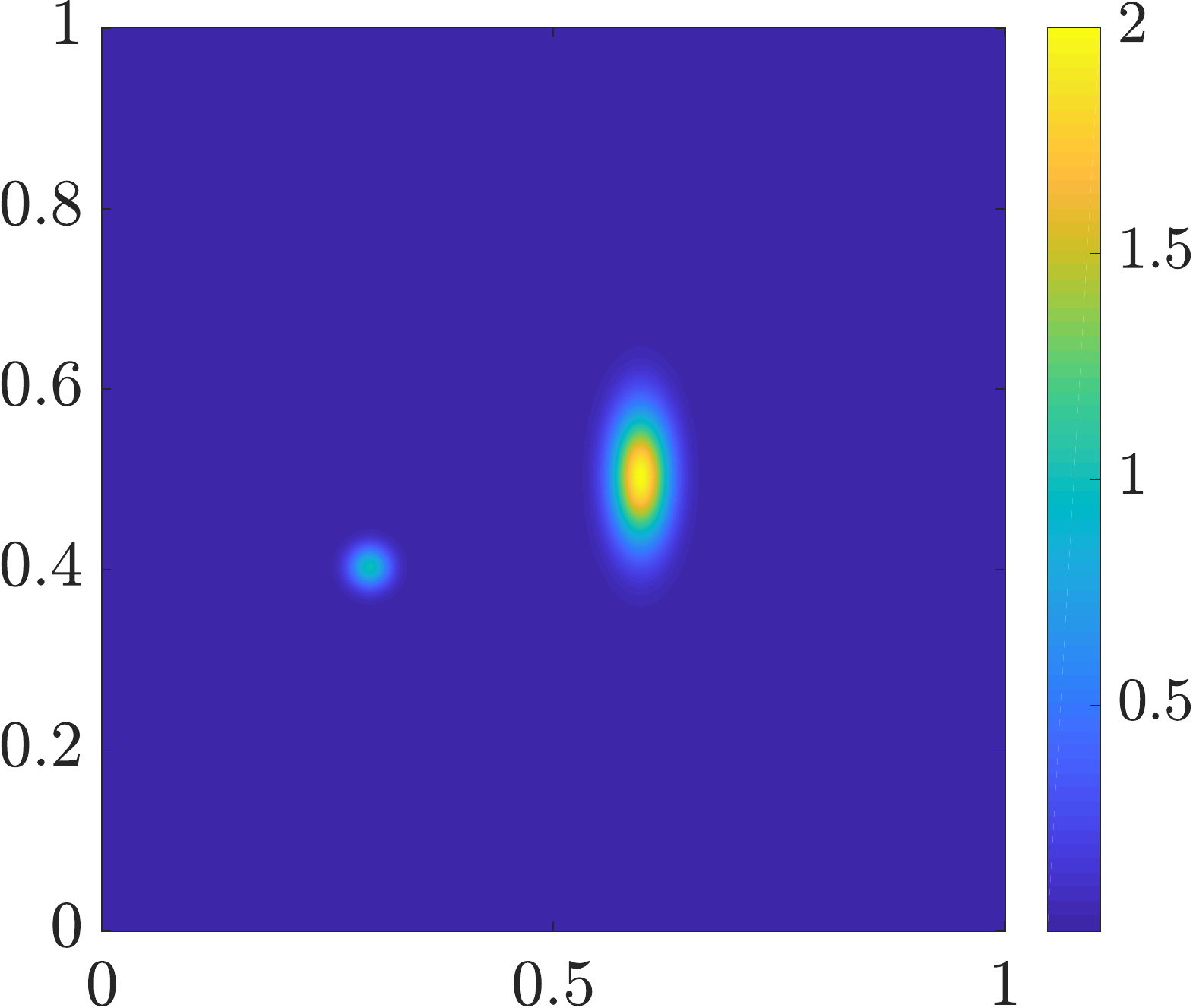} & 
\includegraphics[height = 3.7cm]{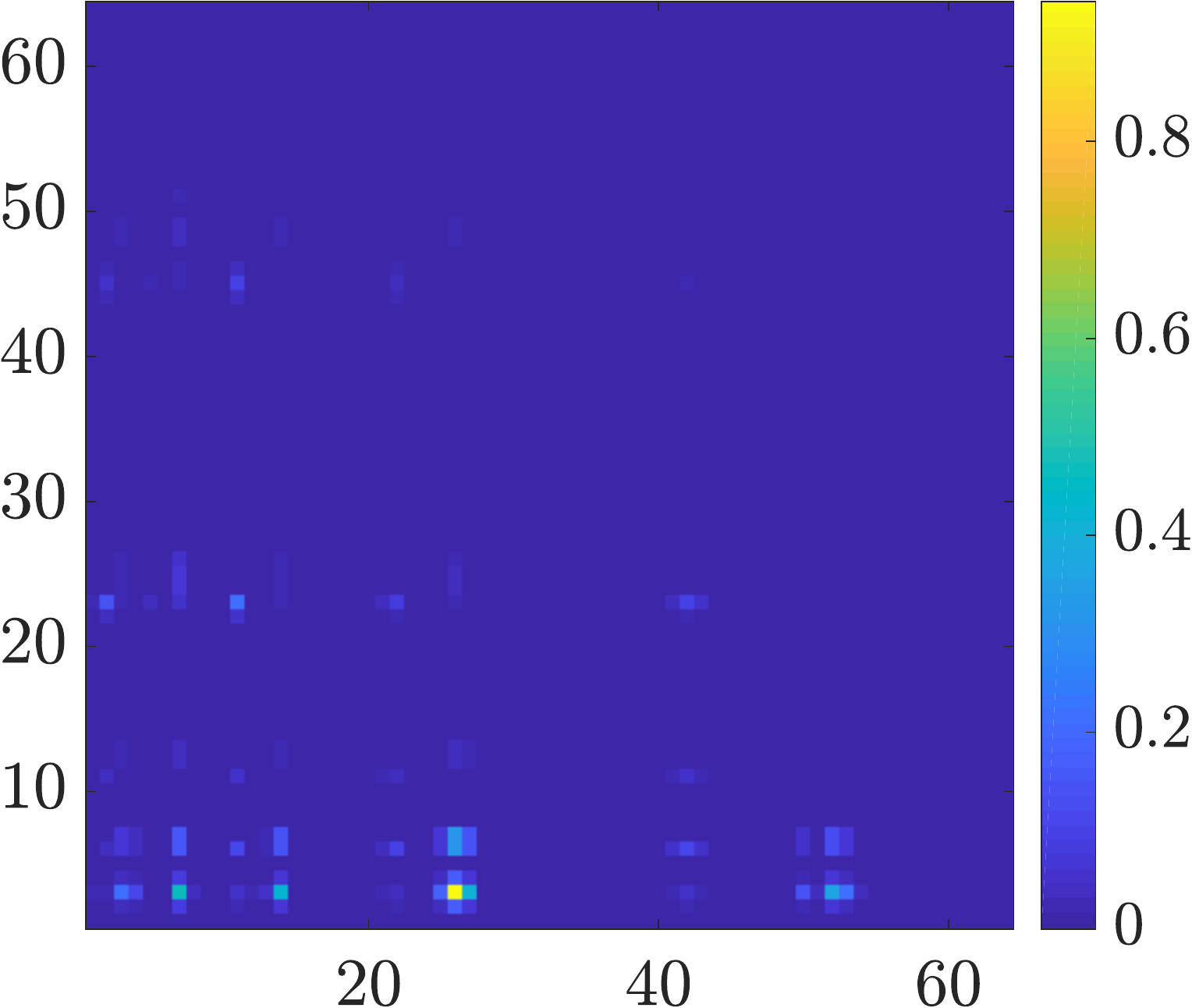} & 
\includegraphics[height = 3.7cm]{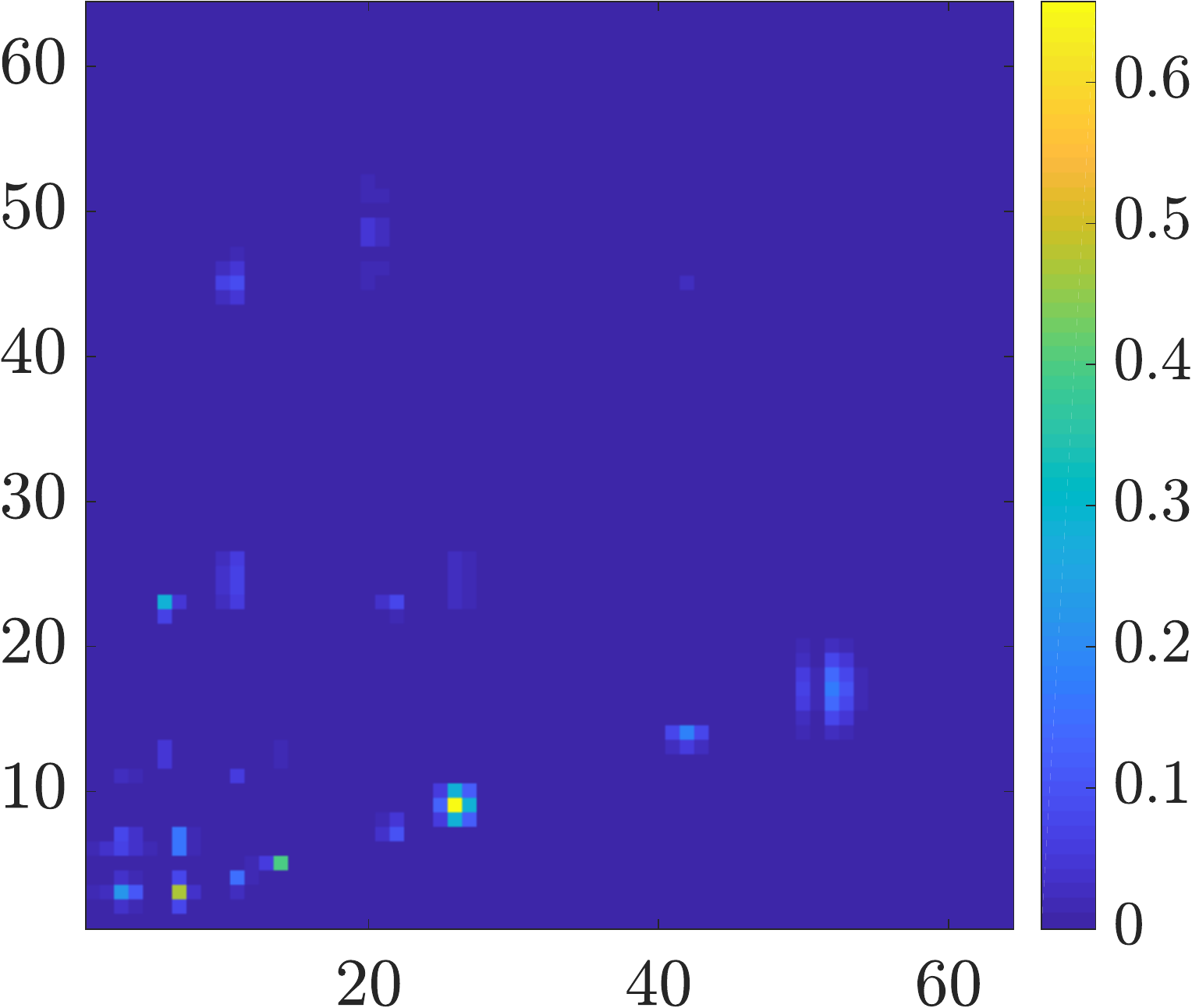} \\
\raisebox{1.5cm}{$u_3$} & 
\includegraphics[height = 3.7cm]{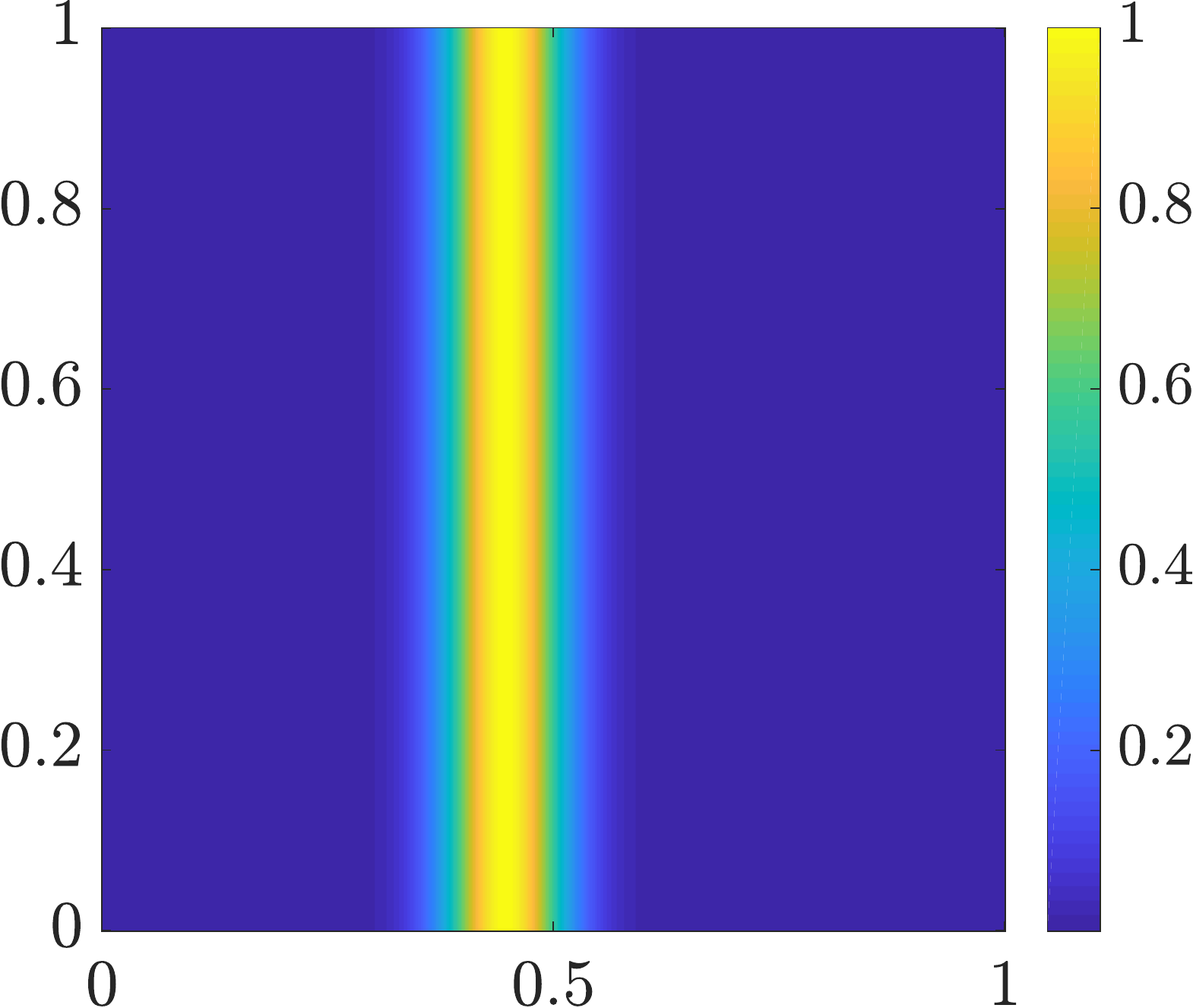} & 
\includegraphics[height = 3.7cm]{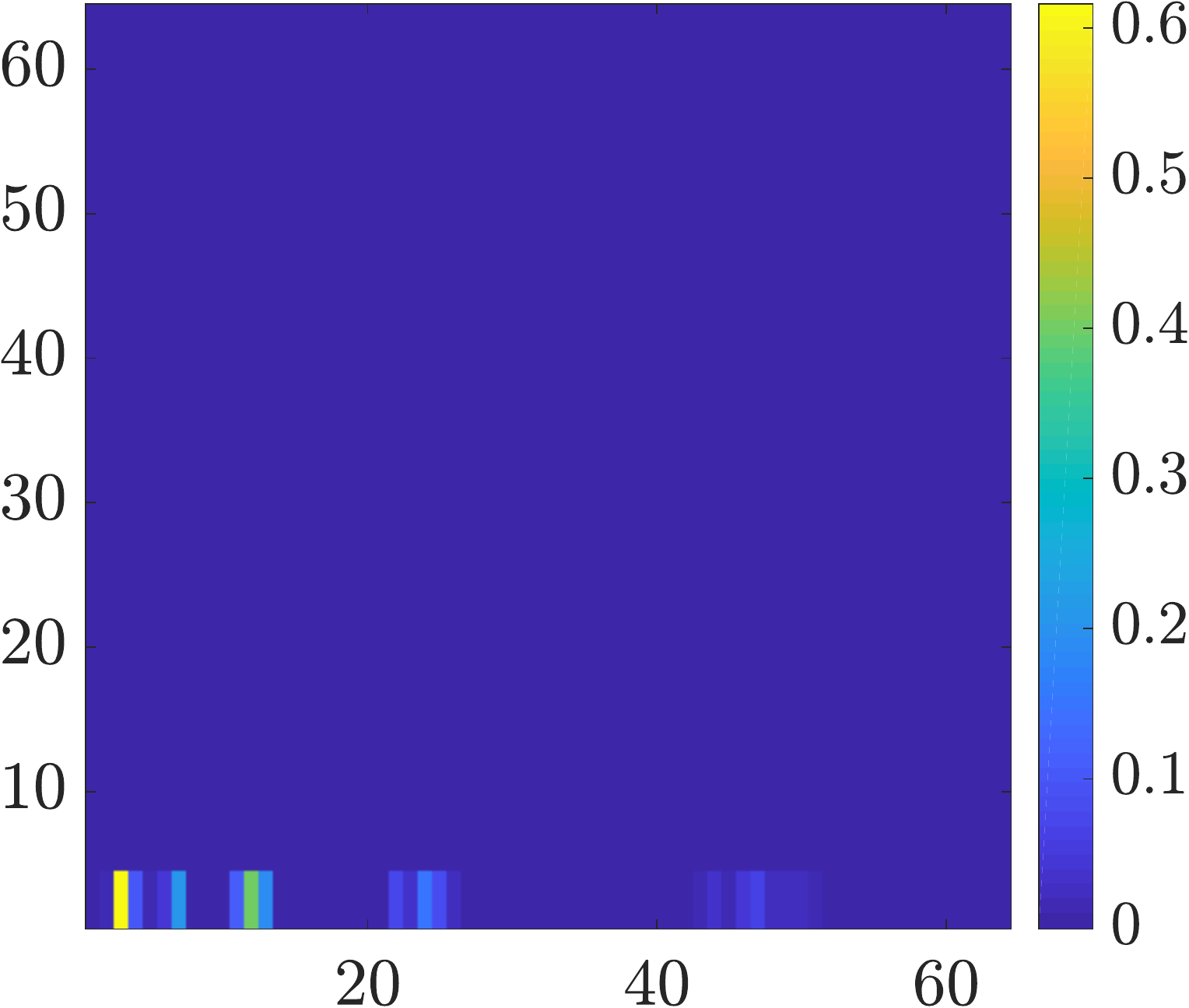} & 
\includegraphics[height = 3.7cm]{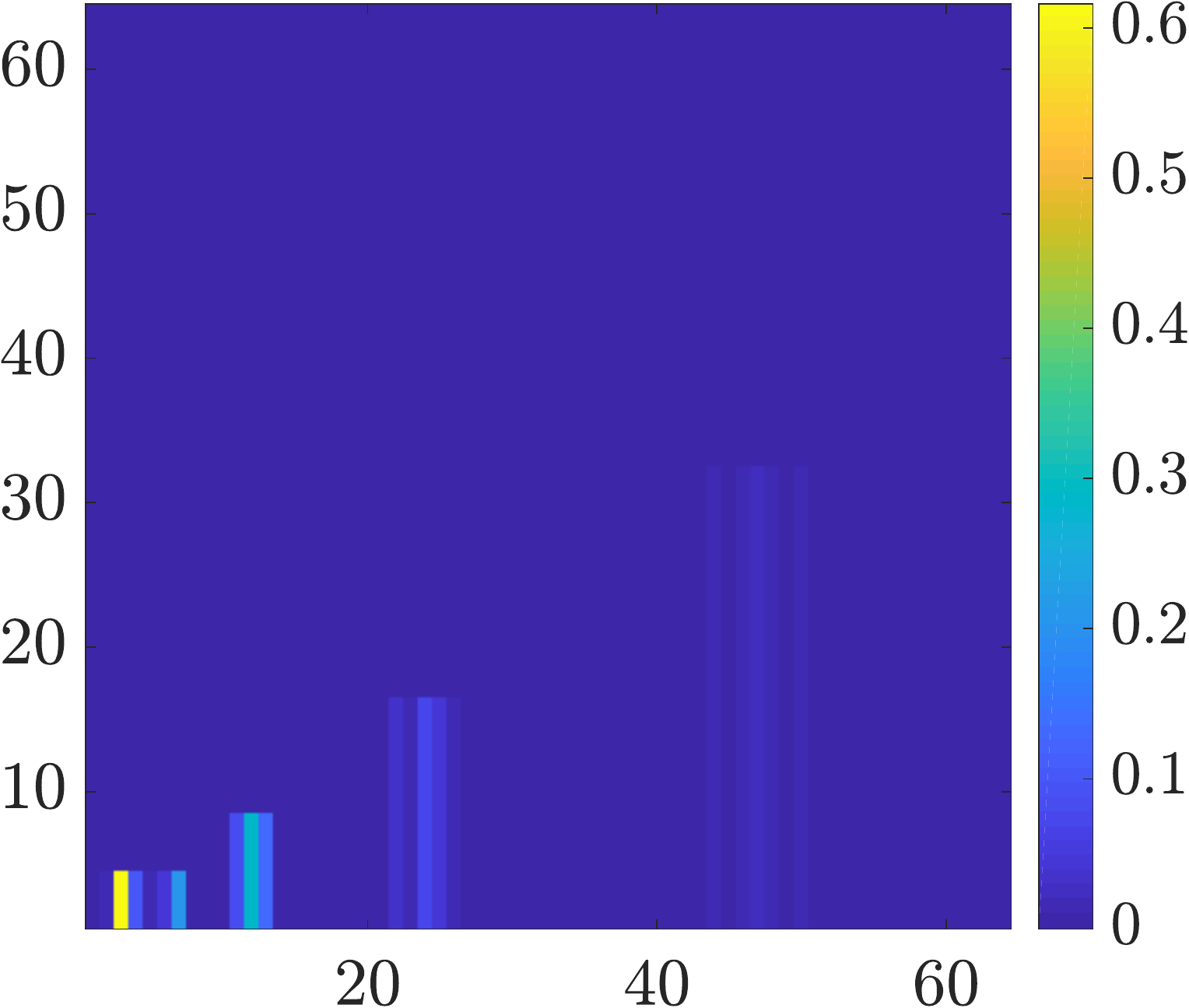} \\
\end{tabular}
\caption{\label{Figure10}(2D ADR problem) Contour plots (left) of functions $u_2$ (top) and $u_3$ (bottom) and corresponding wavelet coefficients with respect to anisotropic (center) and isotropic (right) tensor product wavelets.}
\end{figure}
Letting $s=100$, the relative best $s$-term approximation error with respect to the $H^1(\mathcal{D})$-norm of is $10^{-1}$ (anisotropic wavelets) and $6.6\cdot10^{-2}$ (isotropic wavelets) for $u_2$, and $8.2\cdot10^{-3}$ (anisotropic wavelets) and $1.2\cdot10^{-1}$ (isotropic wavelets) for $u_3$. As expected, anisotropic wavelets generate a very sparse representation of $u_3$. For $u_2$, the compression achieved by anisotropic and isotropic wavelets is comparable, in slight favor of isotropic wavelets.\footnote{Taking advantage of the norm equivalence property, the $H^1(\mathcal{D})$-norm is approximated using the $\ell^2$-norm of the wavelet coefficients as in \eqref{eq:H1errorapprox}.}

\paragraph{Sensitivity of the recovery error to the number of test functions.} 
We assess the performance of \corsing in the case of  anisotropic and isotropic wavelets and compare uniform random subsampling ($\bm{p} \propto \bm{1}$) with the nonuniform subsampling based on the local $a$-coherence upper bound 
$$
\nu_{\bm{q}} = \min\left\{1,\frac{\|\bf{q}\|_2^2}{\|\bm{q}\|_\infty^2 |\widehat{\bm{q}}|^{\bm{1}}}\right\},
$$ 
which is obtained from the upper bounds in Theorems~\ref{thm:mu_bound_multi_ani} and \ref{thm:mu_bound_multi_iso}. In particular, in \eqref{eq:mu_UB_multid} and \eqref{eq:mu_UB_multid_iso}, we consider the second argument of the minimum and use that $2^{-(n-\|\bm{q}\|_0)\ell_0} \leq 1$, while we use that $2^{-(2+n)\ell_0} \leq 1$ in \eqref{eq:mu_UB_multid_q0} and \eqref{eq:mu_UB_multid_q0_iso}.  We set $\ell_0 = 2$, $L=6$ (corresponding to $N =2^{2L} = 4096$). As for the test space, we fix $R = 2^L$, corresponding to $M = N$. Although issue (i) in Theorems~\ref{thm:CORSING_rec_multi_ani} and \ref{thm:CORSING_rec_multi_iso} suggests choosing $R \sim sN^{3-\frac{2}{n}}$, the choice $R = N$ turns out to be sufficient to have a well-conditioned Petrov-Galerkin discretization matrix $B$ in practice. We set $s = 100$ and let $m = 100, 200, 300, 400, 500$. For each value of $m$, we run $100$ tests of \corsing with uniform and nonuniform subsampling. We plot the relative recovery error measured with respect to the $H^1(\mathcal{D})$-norm as a function of the number of tests $m$ in Figures~\ref{Figure11} and \ref{Figure12} for $u_2$ and $u_3$, respectively.
\begin{figure}[t]
\centering
\includegraphics[height=4cm]{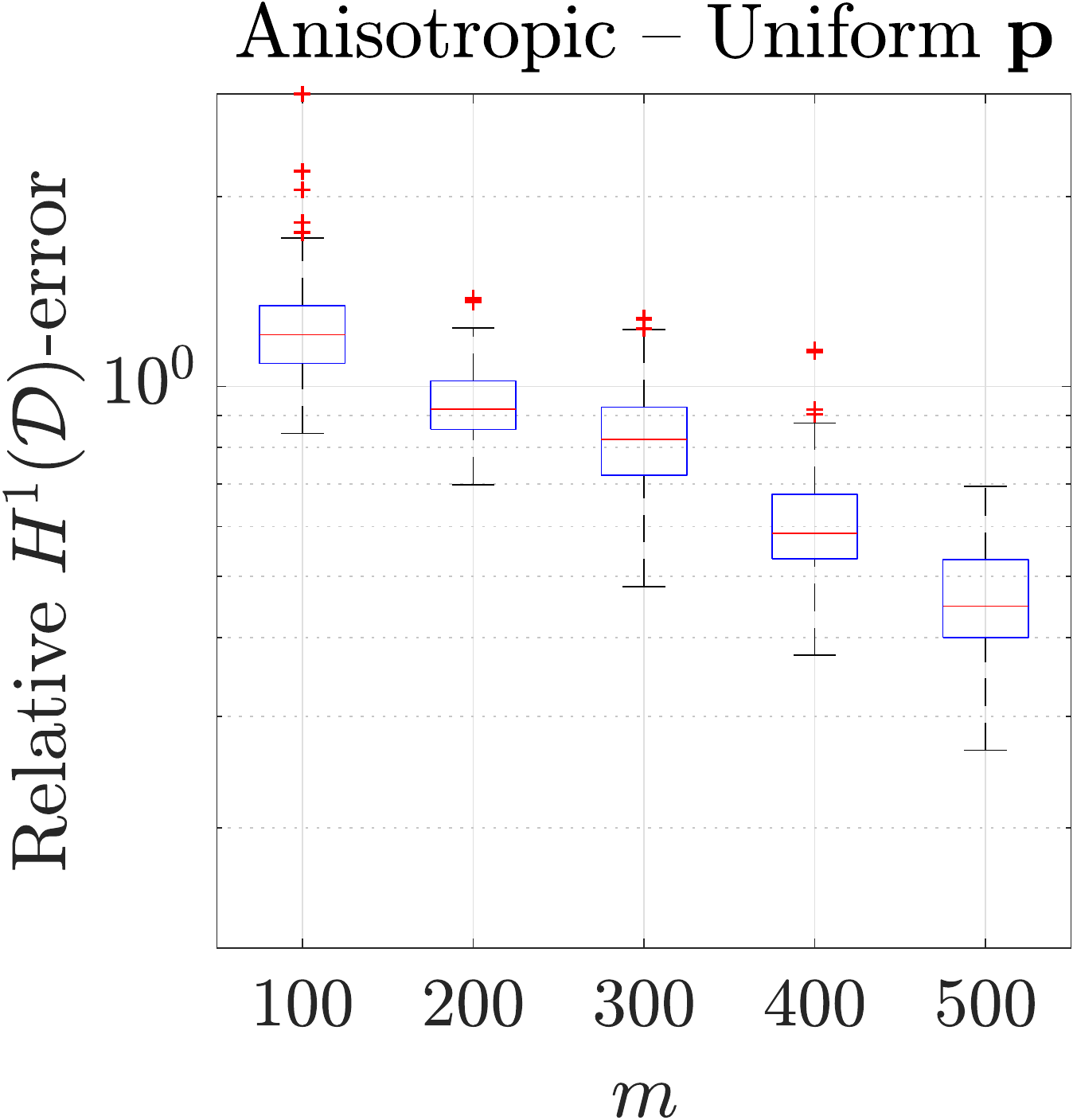} 
\includegraphics[height=4cm]{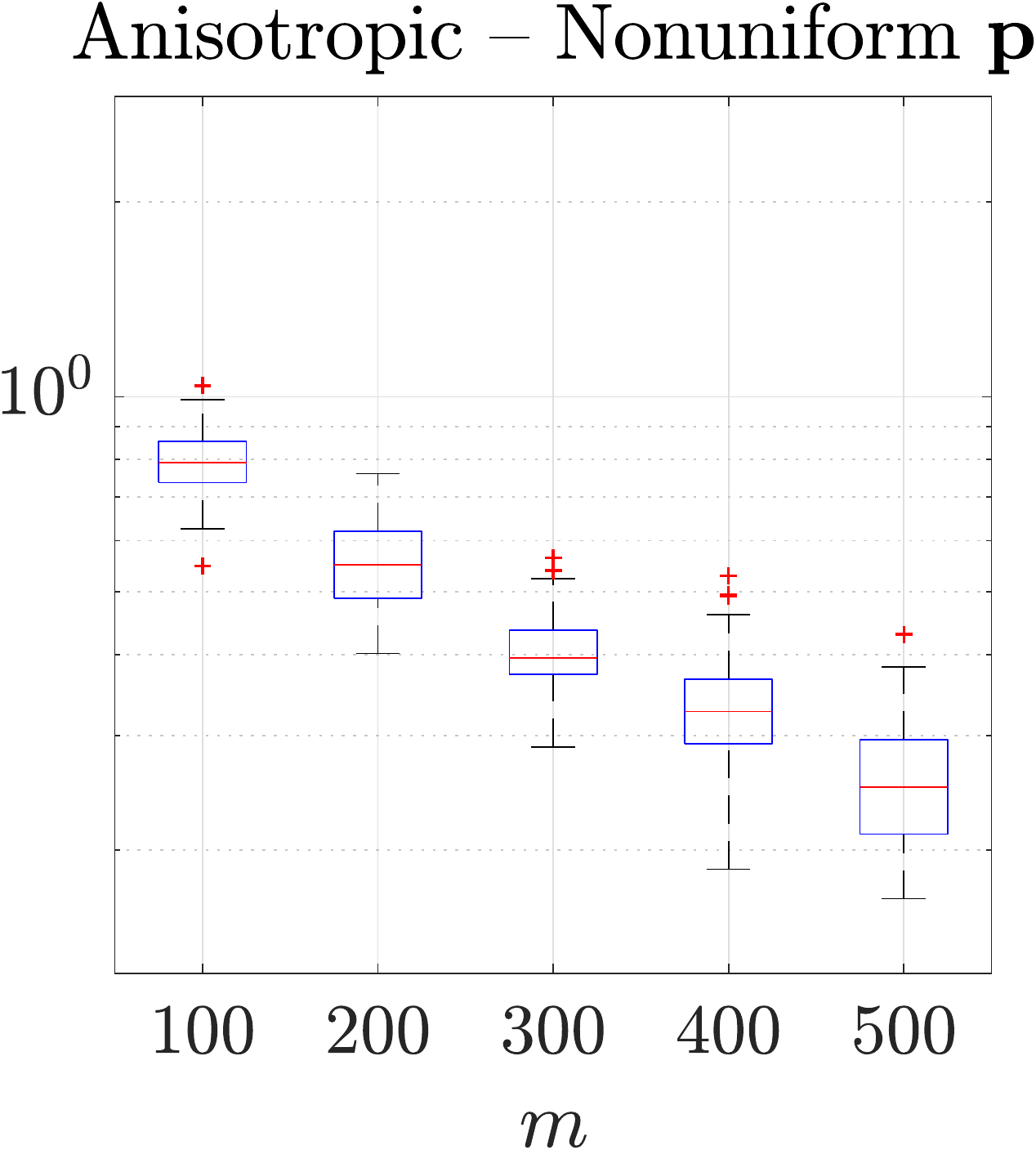} 
\includegraphics[height=4cm]{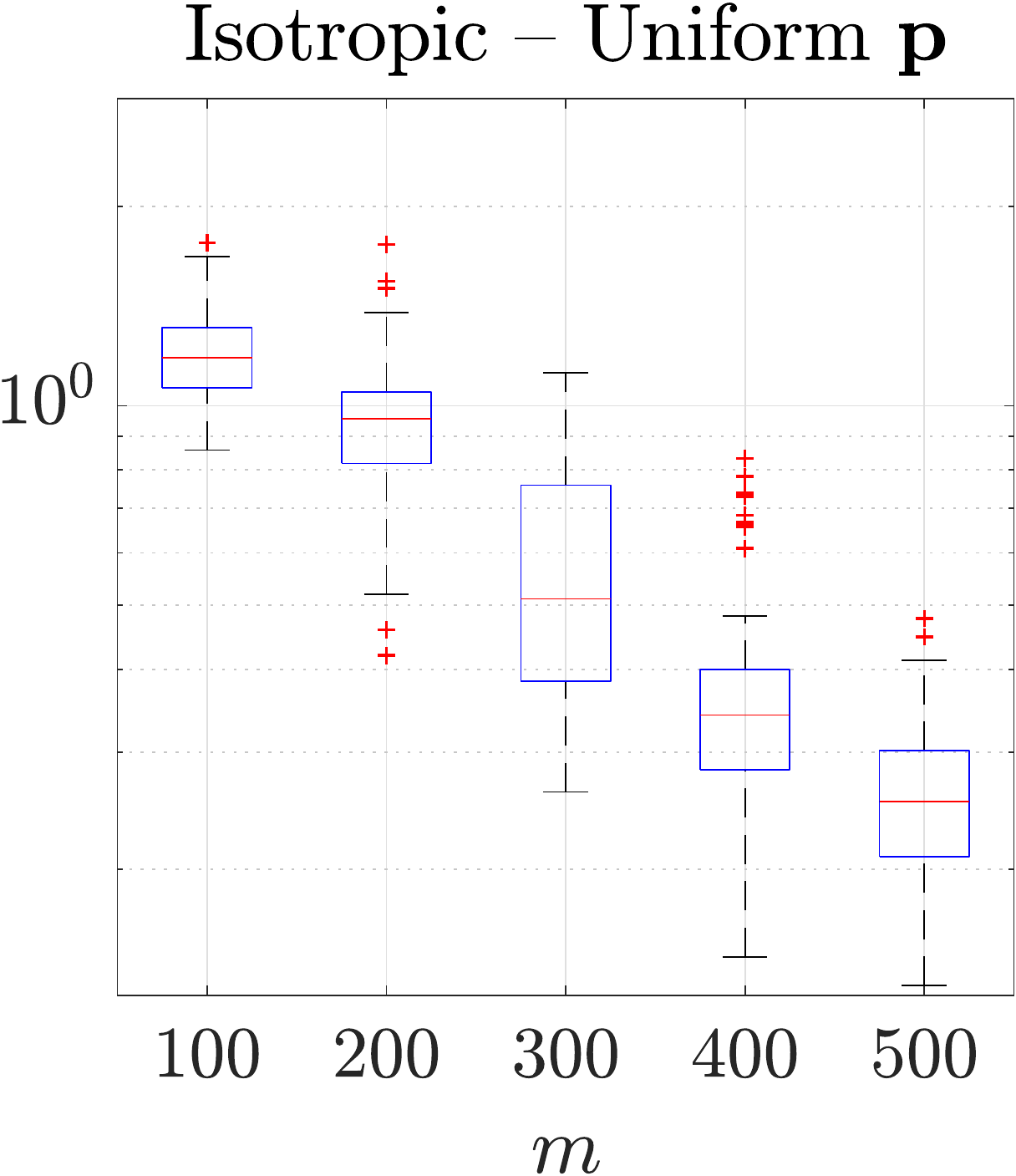} 
\includegraphics[height=4cm]{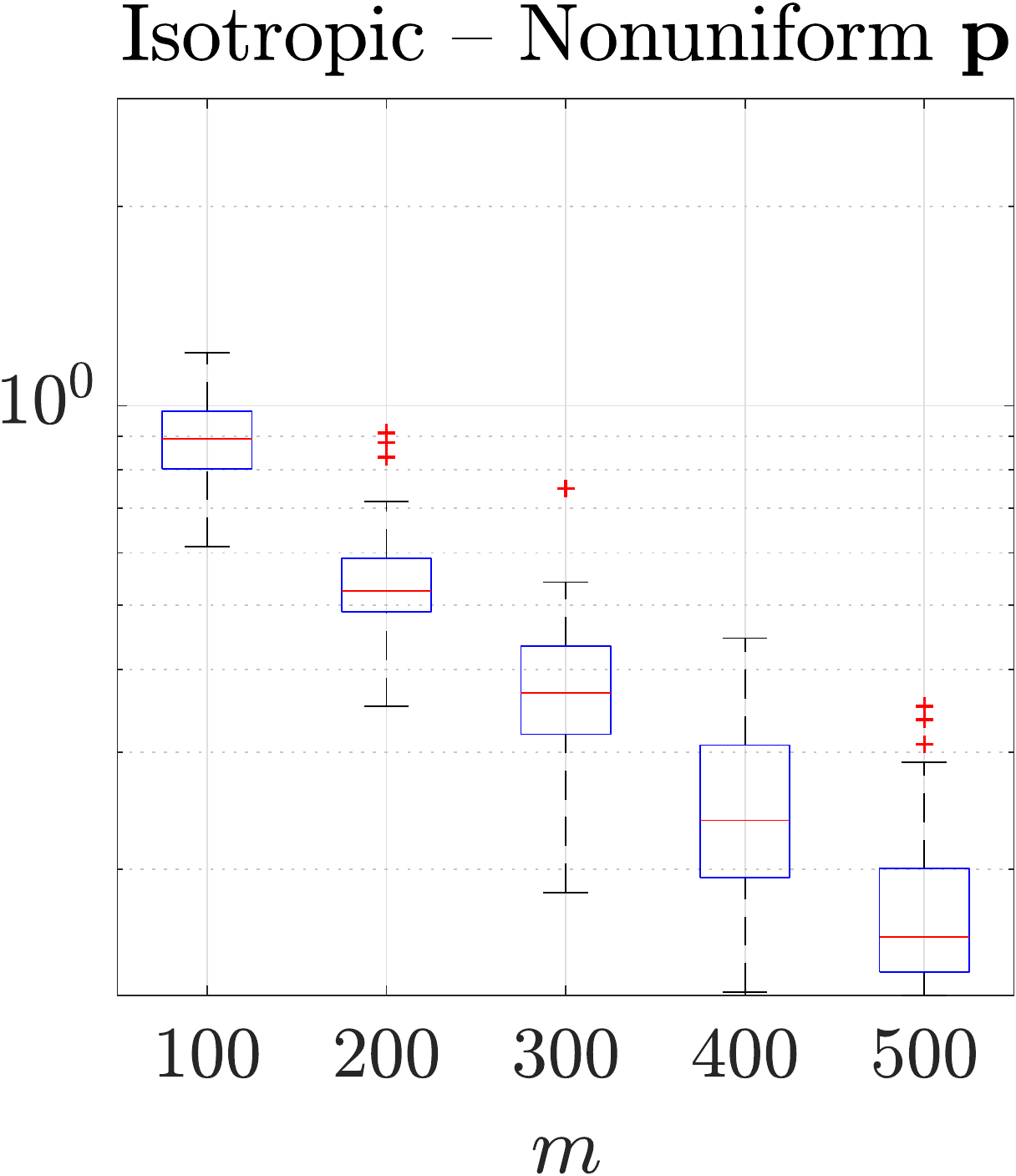} 
\caption{\label{Figure11}(2D ADR problem) Box plot of the relative recovery error for the function $u_2$ with respect to the $H^1(\mathcal{D})$-norm as a function of the number of tests $m$ with anisotropic and isotropic wavelets and uniform and nonuniform subsampling.}
\end{figure}
In the case of $u_2$, isotropic wavelets slightly outperform anisotropic wavelets. The benefit of nonuniform subsampling over uniform subsampling is evident both in terms of the probability of success (smaller boxes) and of accuracy. For $u_3$, anisotropic wavelets significantly outperform isotropic wavelets, thanks to the better compressibility of the solution. Moreover, uniform sampling fails to recover the solutions in both cases, whereas nonuniform sampling exhibits a convergent behavior. This experiment confirms the key role played by the local $a$-coherence for a successful implementation of the \corsing \WF method.
\begin{figure}[t]
\centering
\includegraphics[height=4cm]{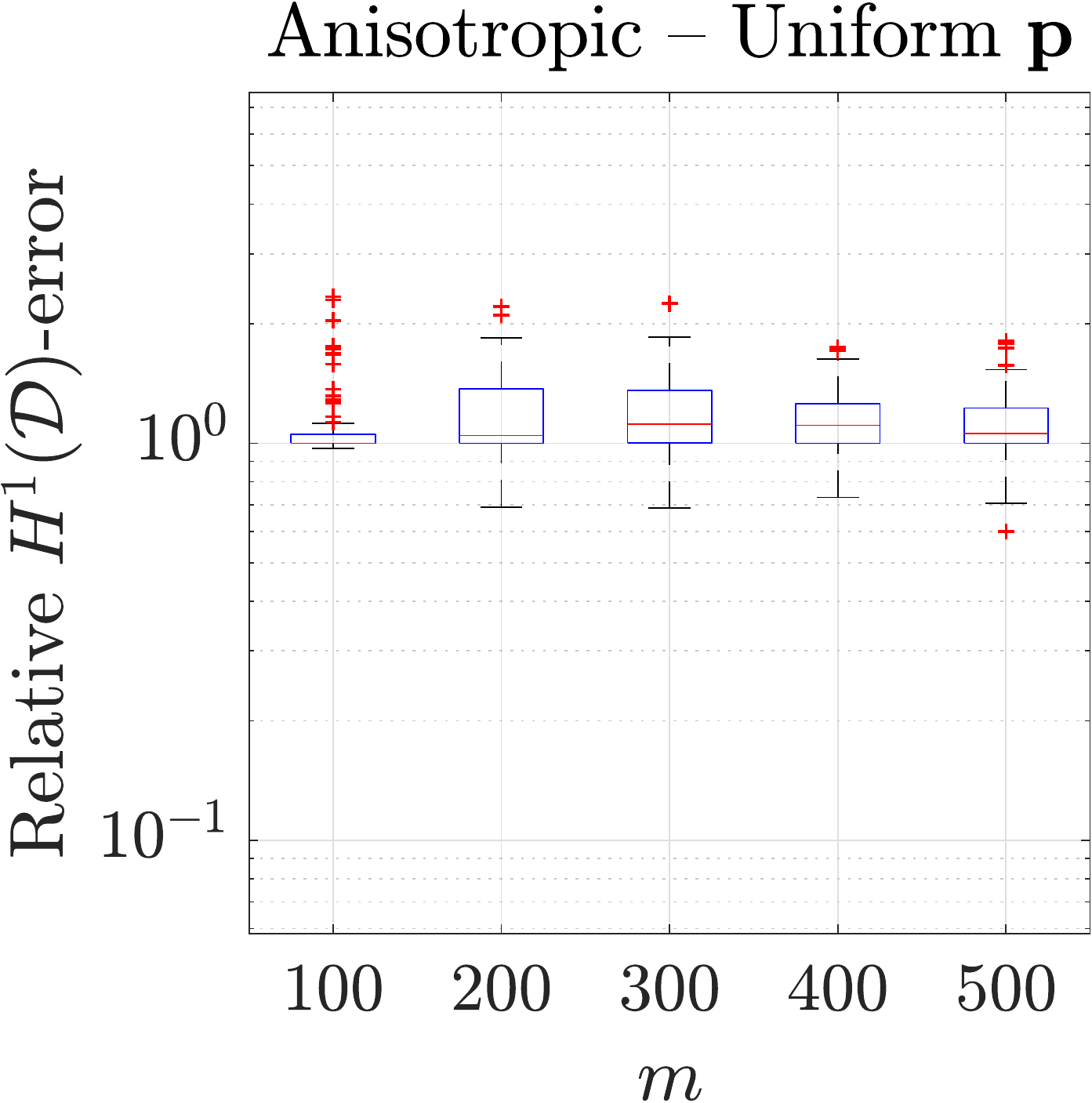} 
\includegraphics[height=4cm]{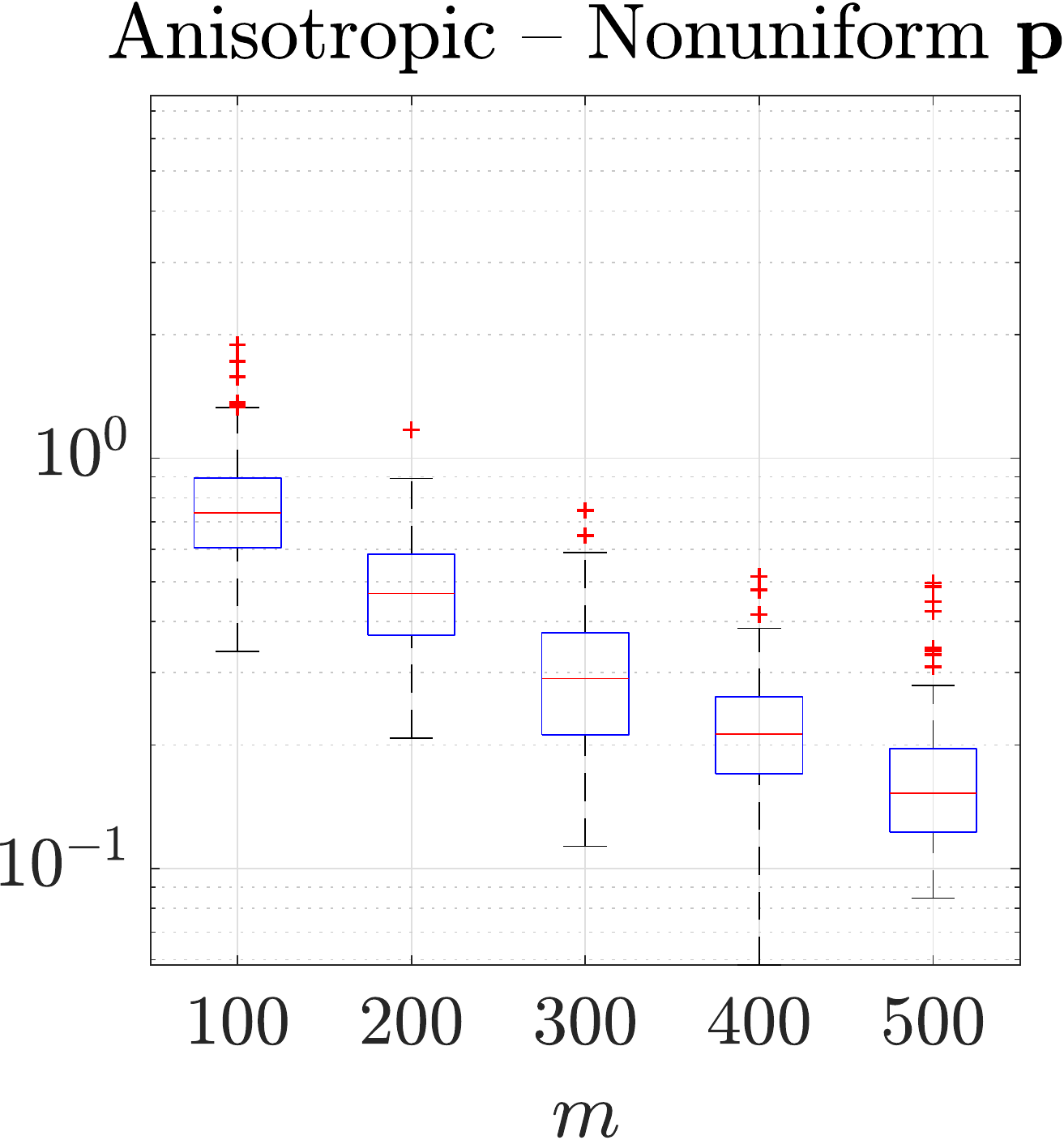} 
\includegraphics[height=4cm]{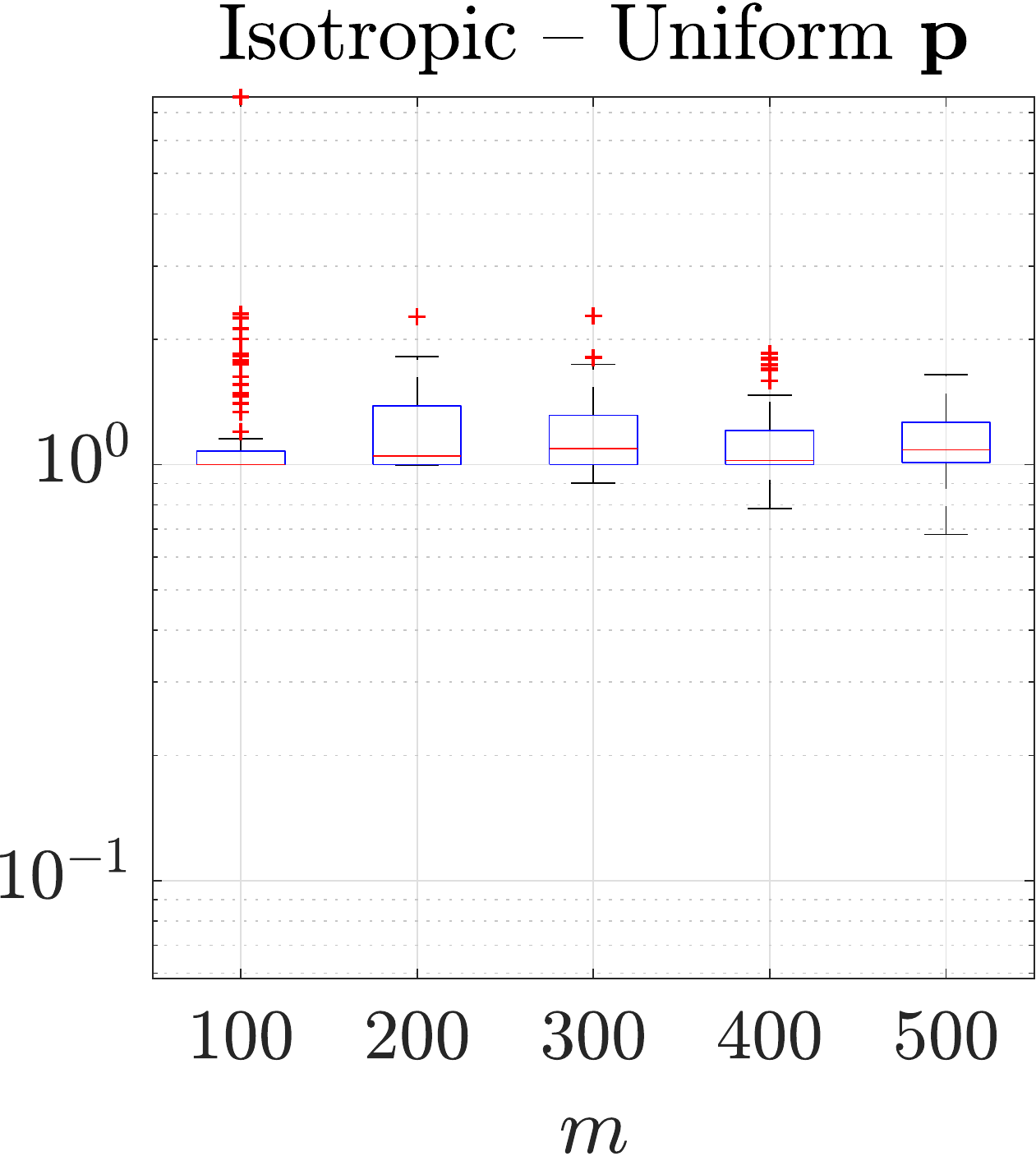} 
\includegraphics[height=4cm]{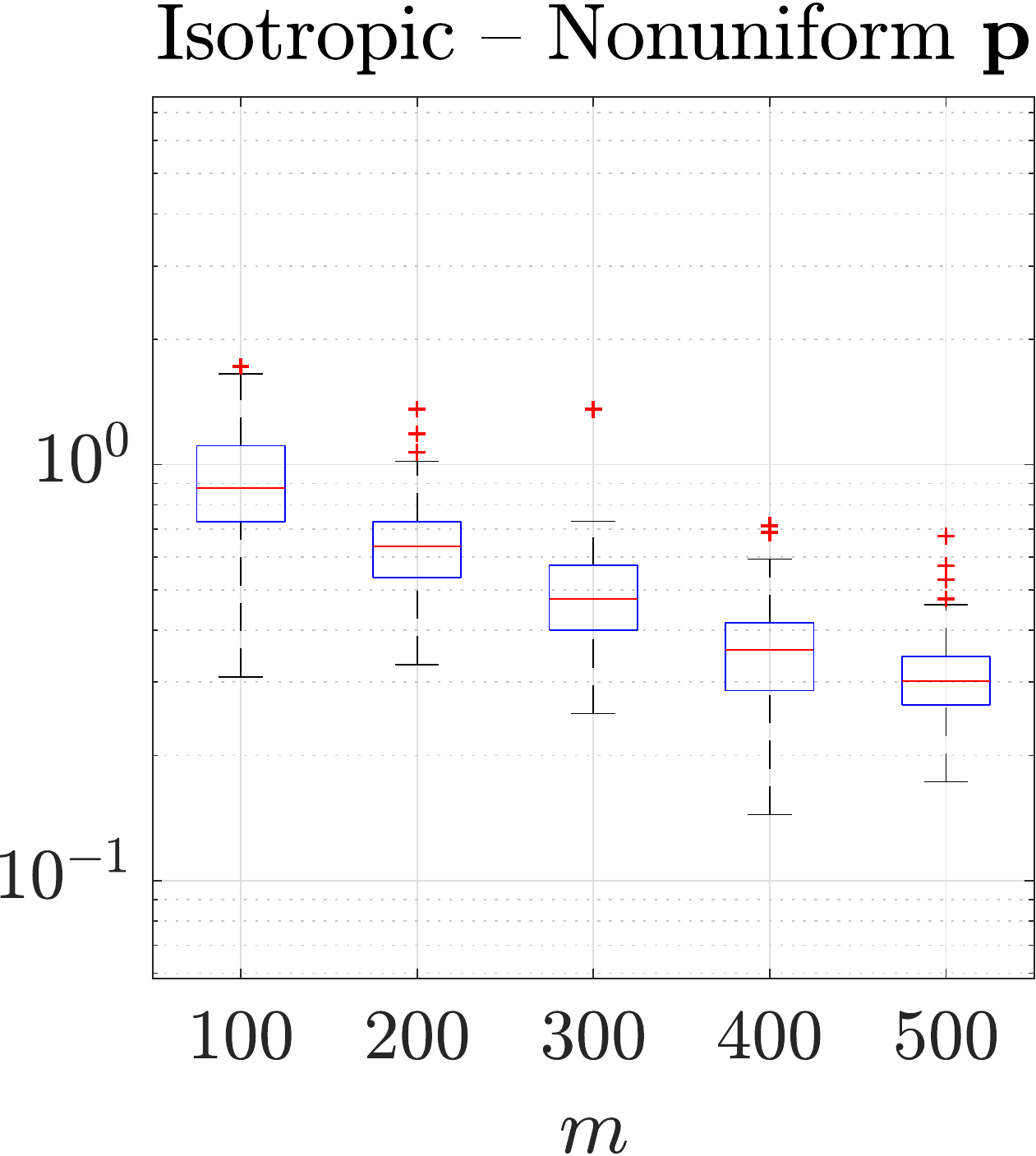}
\caption{\label{Figure12}(2D ADR problem) Box plot of the relative recovery error for the function $u_3$ with respect to the $H^1(\mathcal{D})$-norm as a function of the number of tests $m$ with anisotropic and isotropic wavelets and uniform and nonuniform subsampling.}
\end{figure}

\subsection{The 3D case}
\label{sec:3Dnum}
We validate the \corsing \WF method on a 3D ADR problem on $\mathcal{D}=(0,1)^3$ with constant coefficients $\mu = \rho = 1$, \and $\bm{b} = [1, 1, 1]^T$. We consider the exact solution
\begin{equation}
\label{eq:def_u4}
u_4(x_1,x_2,x_3) = 
\exp\left(-\frac{(x_1-0.4)^2}{0.005}\right)
\exp\left(-\frac{(x_2-0.5)^2}{0.0005}\right)
\exp\left(-\frac{(x_3-0.6)^2}{0.005}\right).
\end{equation}
The function $u_4$ exhibits an anisotropic Gaussian-shaped feature centered at the point $(0.4,0.5,0.6)$. We compare anisotropic and isotropic wavelets.

\paragraph{Wavelet coefficients and best $s$-term approximation.} 
We fix $\ell_0 = 2$, $L = 4$ (corresponding to a trial space of dimension $N = 2^{3L} =4096$) and $s = 200$. The wavelet coefficients and the best $s$-term approximation are shown in Figure~\ref{Figure13}.
\begin{figure}[t]
\centering
\includegraphics[height = 4.1cm]{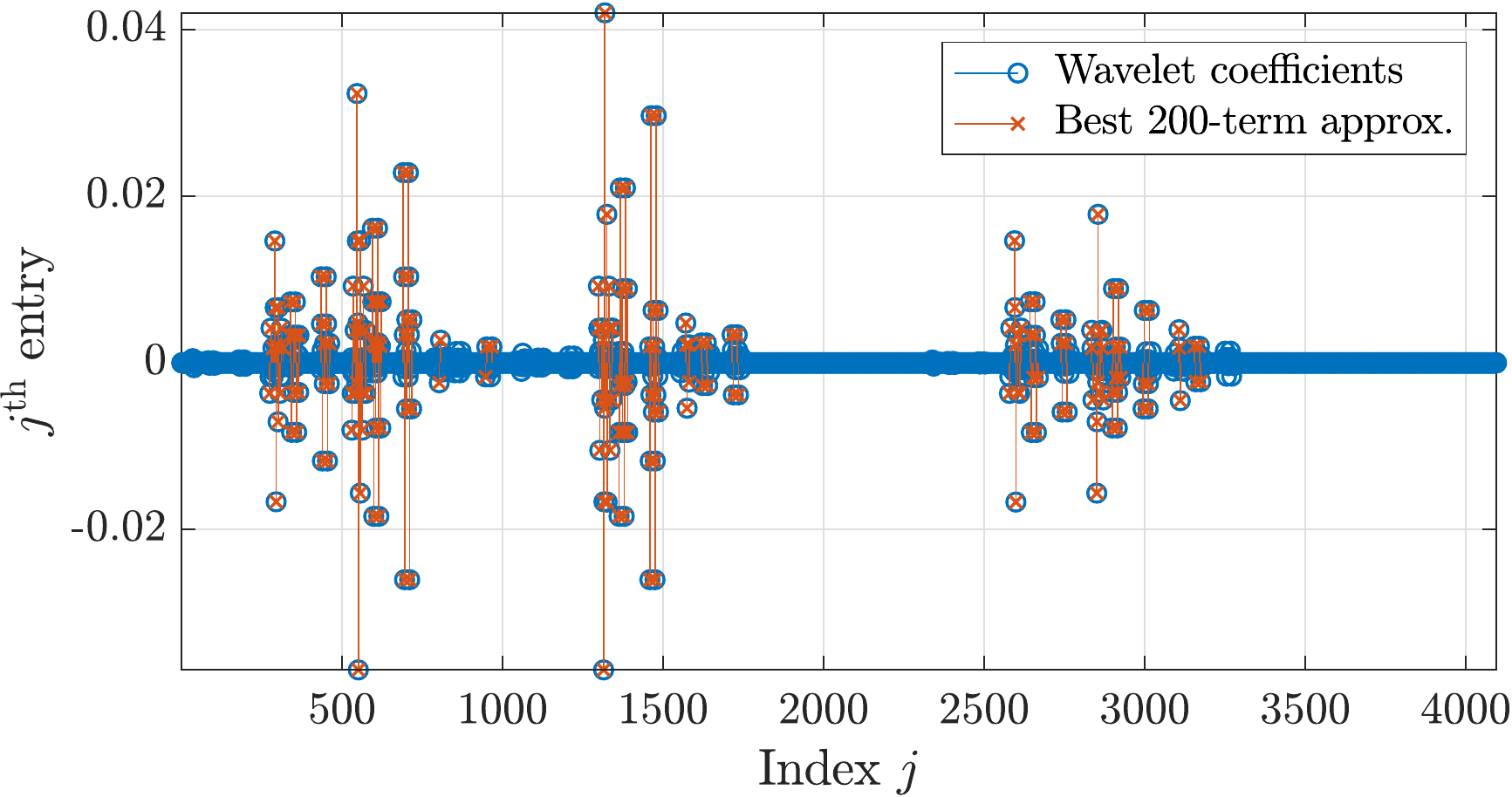} 
\includegraphics[height = 4.1cm]{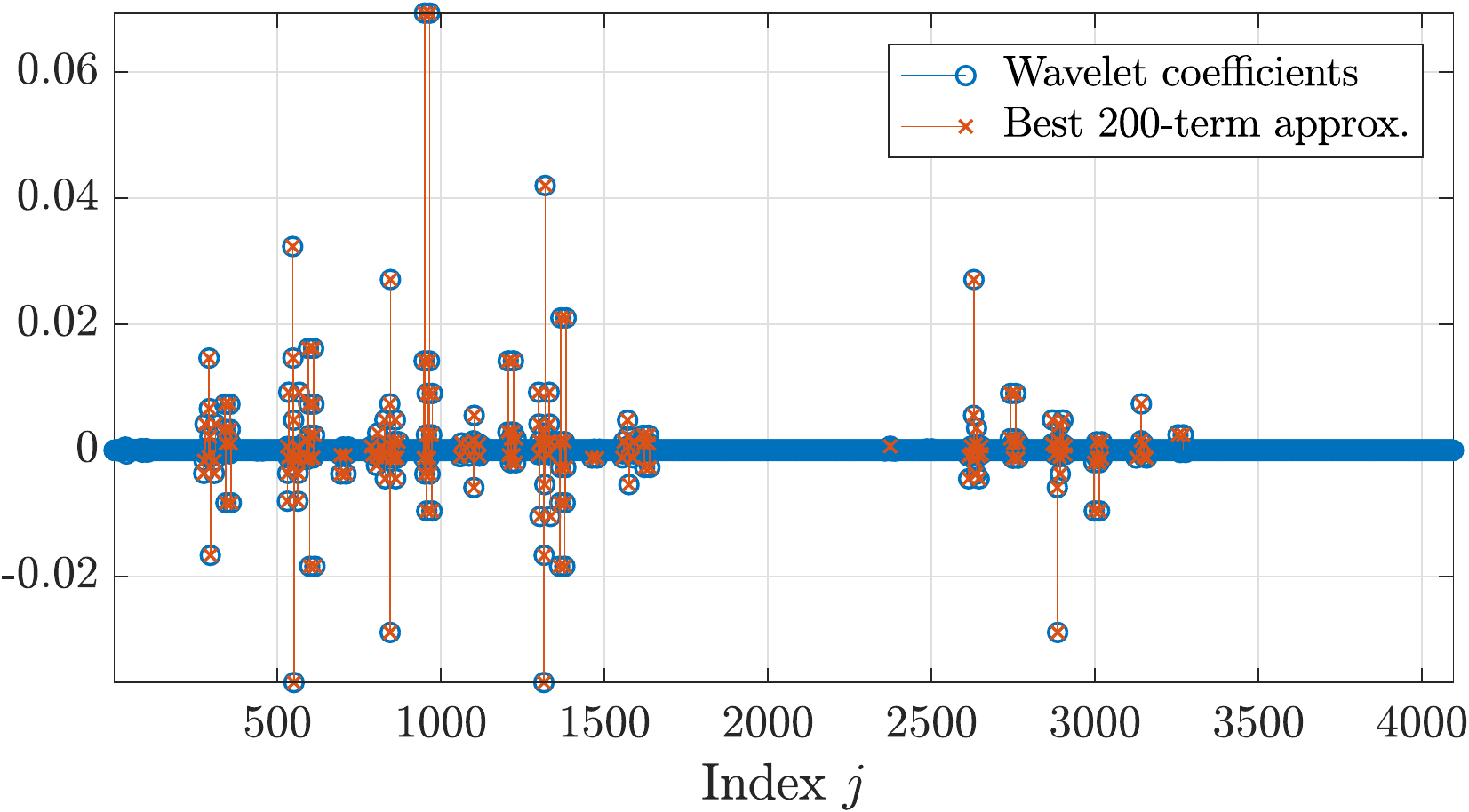} 
\caption{\label{Figure13}(3D ADR problem) Wavelet coefficients of the function $u_4$ and best $200$-term approximation for anisotropic (left) and isotropic (right) wavelets.}
\end{figure} 
The relative best $s$-term approximation error with respect to the $H^1(\mathcal{D})$-norm is $9.3\cdot 10^{-2}$ for the anisotropic wavelets and $2.5\cdot 10^{-2}$ for the isotropic wavelets. In this case, the isotropic tensorization is able to sparsify the function to a slightly better extent.

\paragraph{Sensitivity of the recovery error to the number of test functions.} 
We compare anisotropic and isotropic wavelets and uniform and nonuniform subsampling. We set $m = 200,300,400,500,600$. The box plots corresponding to $100$ runs of \corsing are shown in Figure~\ref{Figure14}.
\begin{figure}[t]
\centering
\includegraphics[height=4cm]{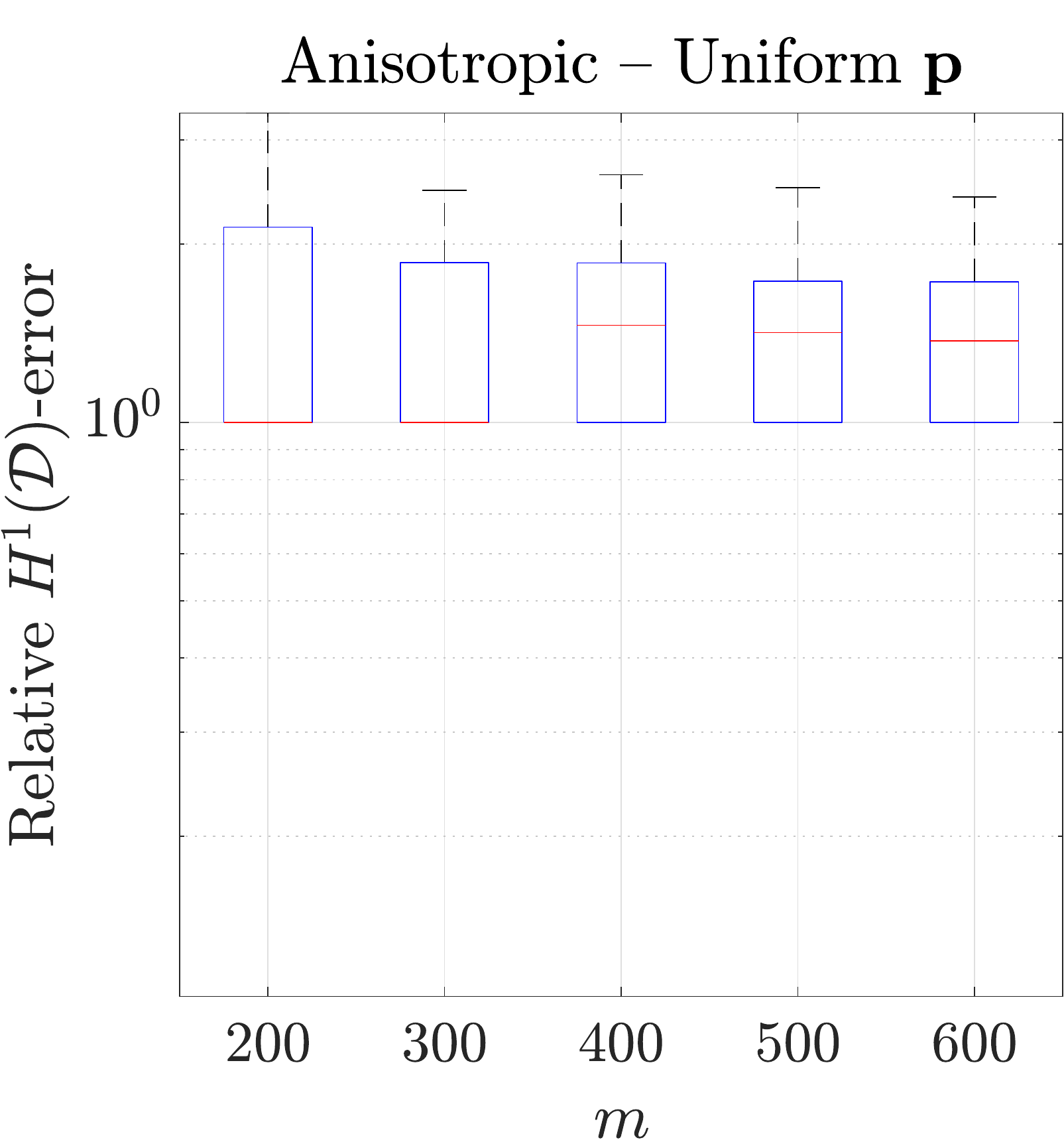} 
\includegraphics[height=4cm]{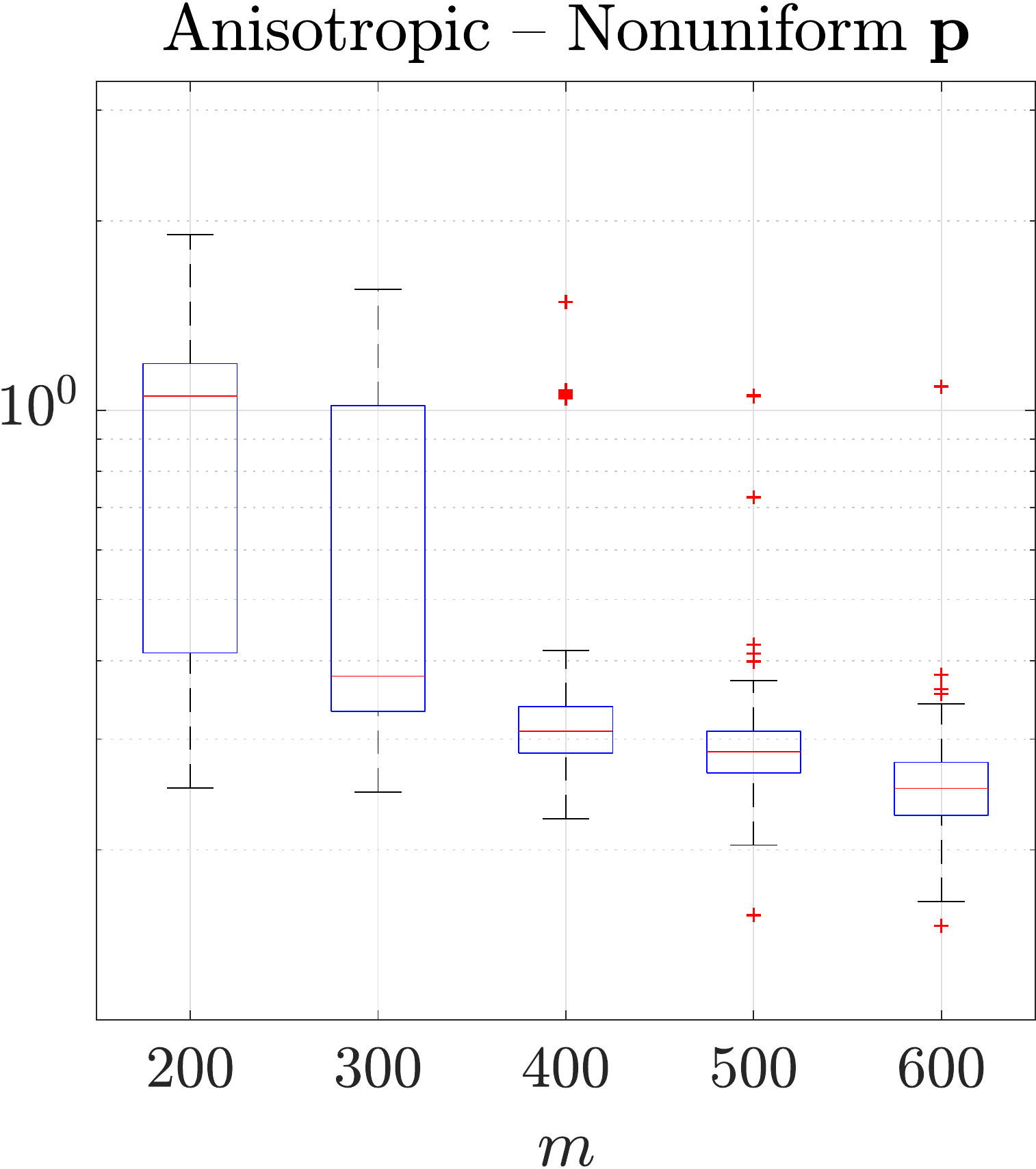} 
\includegraphics[height=4cm]{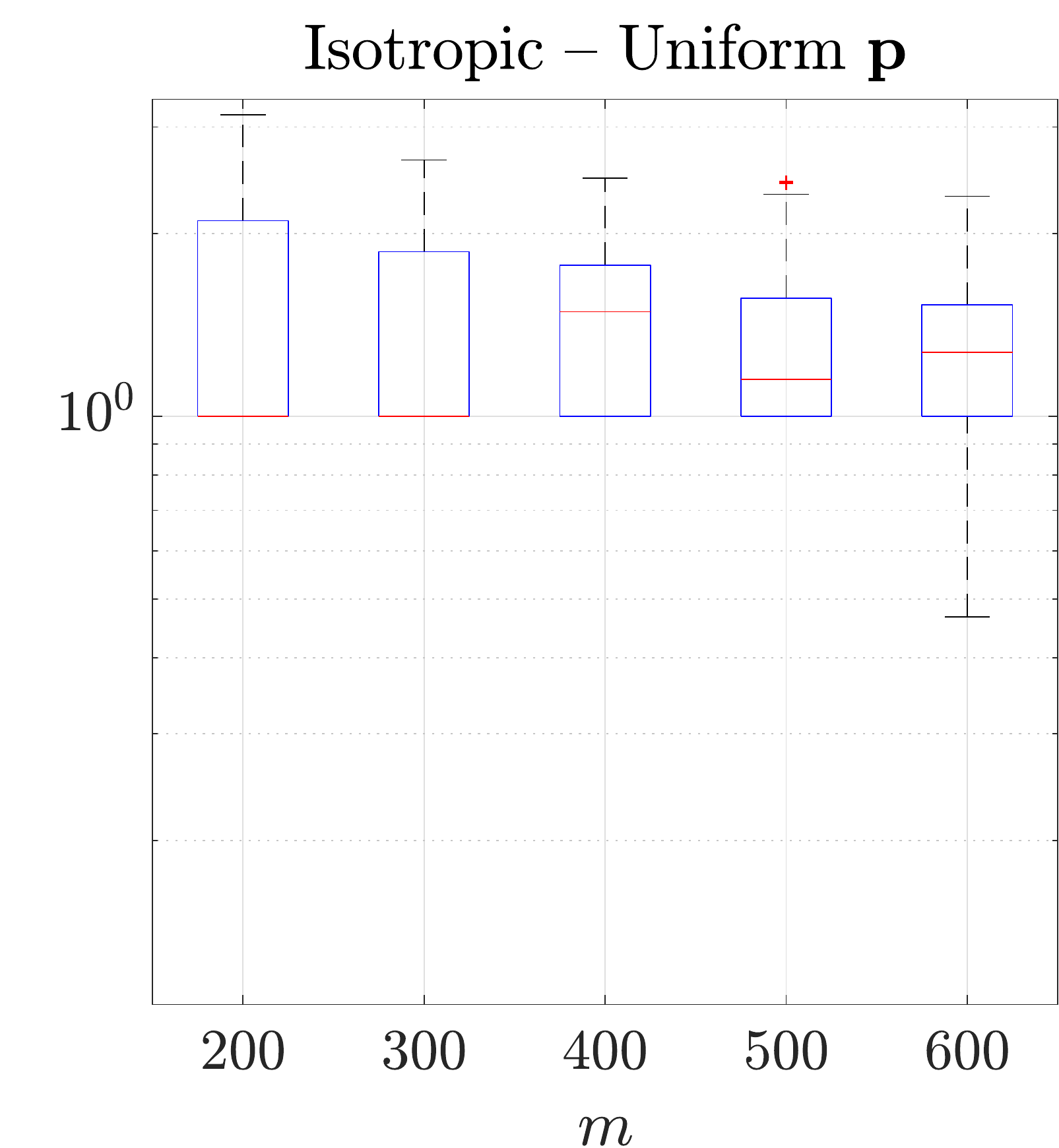} 
\includegraphics[height=4cm]{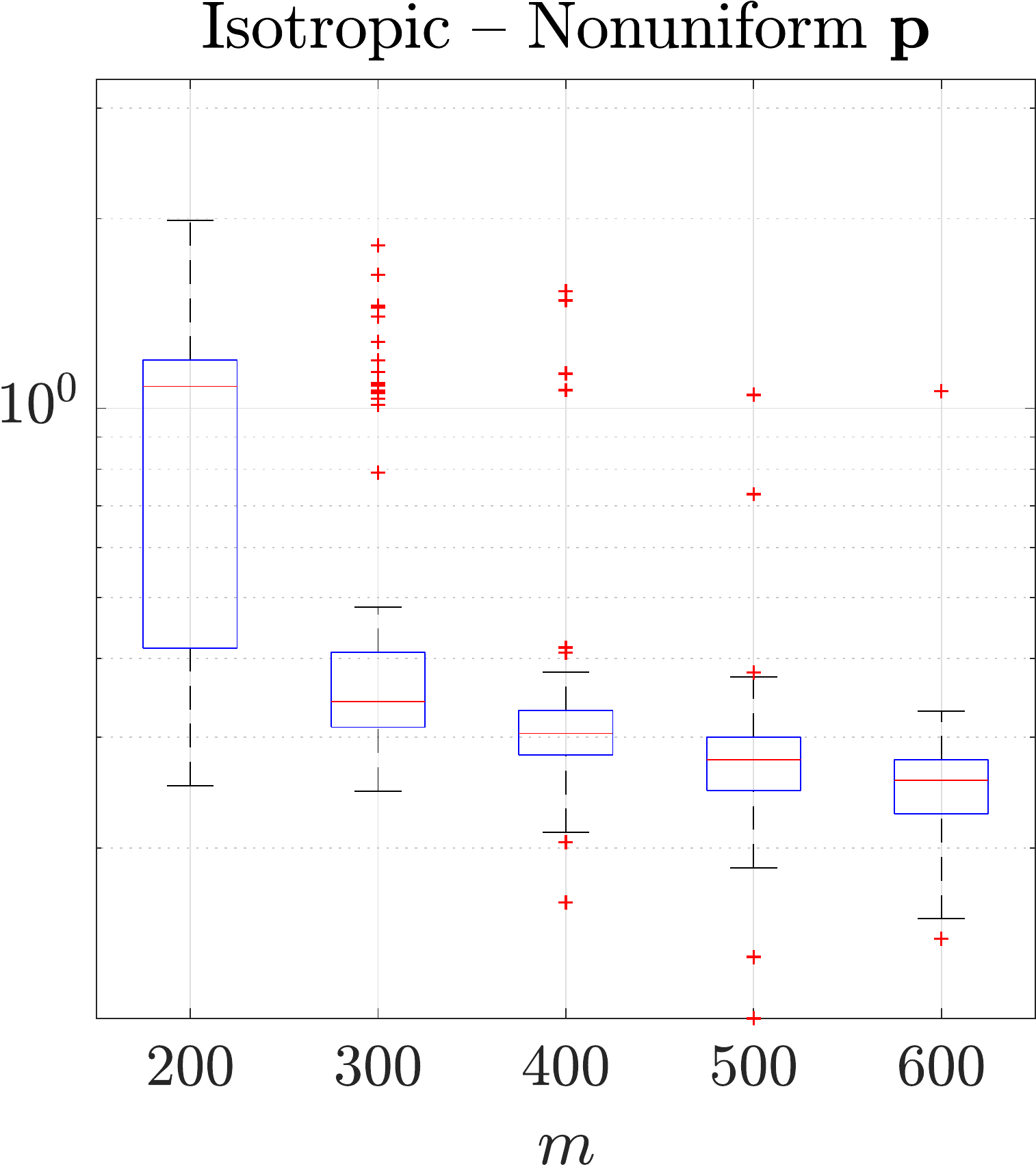}
\caption{\label{Figure14}(3D ADR problem) Box plot of the relative recovery error with respect to the $H^1(\mathcal{D})$-norm as a function of the number of tests $m$ with anisotropic and isotropic wavelets and for uniform and nonuniform subsampling.}
\end{figure}
The performance of anisotropic and isotropic wavelets is similar. However, uniform subsampling is not able to recover the solution at all. Comparing this results with those of the 2D case (Figures~\ref{Figure11} and \ref{Figure12}), we note how using a bad probability measure (i.e., the uniform) deteriorates the performance of the method more heavily as the dimension of the domain increases. Finally, this experiment confirms that the theoretical analysis carried out in Section~\ref{sec:localcoherence} turns out to be a useful tool for an effective implementation of the \corsing \WF method.

\section{Conclusions}
\label{sec:conclusion}

We presented a wavelet-Fourier discretization technique for ADR equations based on the Petrov-Galerkin method and on the compressed sensing paradigm, called \corsing \WF. We carried out a theoretical analysis of the method, which hinges on the concept of local $a$-coherence and provides practical recipes for a successful implementation of the method. Numerical experiments confirm the robustness and reliability of the \corsing \WF approach for $n$-dimensional ADR equations with $n = 1,2, 3$. 

In particular, we showed that the method achieves a recovery error comparable to the best $s$-term approximation error, and that the sampling measure based on the local $a$-coherence proposed here is able to successfully exploit the sparsity of the exact solution in the discretization (in contrast to other randomization strategies such as uniform random subsampling).

Several open issues still remain to be investigated. First, one needs to understand whether the sampling measure proposed in this paper is or is not the ``optimal'' one (in some sense to be specified). On the practical and computational side there is still a lot of work to be done. Although we compared the accuracy of the \textsf{CORSING} $\mathcal{WF}$ solution with the best $s$-term approximation error (Figure~\ref{fig:1D_s_vs_err}), the computational cost of the \textsf{CORSING} $\mathcal{WF}$ procedure with OMP reconstruction scales linearly in $N$ (i.e., the dimension of the trial basis of wavelet functions). Yet, adaptive wavelet methods can recover the best $s$-term approximation error accuracy with an optimal computational cost of $O(s)$ flops. This is a crucial issue to address in order to understand what is the real impact of the \textsf{CORSING} $\mathcal{WF}$. In this direction, a line of research currently under investigation is the use of techniques for sublinear-time compressed sensing recently proposed in \cite{choi2019sparse}. Finally, developing an effective and optimized implementation for \corsing \WF that takes advantage of the wavelet transform, the Fourier transform, and of the tensor product structure of the basis functions is still an open issue that has to be tackled to implement \corsing \WF in dimension $n > 3$.

\section*{Acknowledgements}

The first author acknowledges the support of the Postdoctoral Training Centre in Stochastics of the Pacifical Institute for the Mathematical Sciences (PIMS), the Centre for Advanced Modelling Science (CADMOS), and the Natural Sciences and Engineering Research Council of Canada through grant number 611675 for the financial support. The first author thanks Ben Adcock, Wolfgang Dahmen, and Holger Rauhut for very insightful discussions about approximation theory and compressed sensing. The fourth author acknowledges the research project GNCS-INdAM 2018 ``Tecniche di Riduzione di Modello per le Applicazioni Mediche'', which partially supported this research. The authors would also like to thank the two anonymous reviewers for their helpful and constructive comments.

% Bibliography
\small
\bibliography{library} 
\bibliographystyle{plain}

\end{document}

%% file: header.tex
\usepackage[utf8x]{inputenc}
\usepackage[T1]{fontenc}
\usepackage[english]{babel}              
\usepackage{graphicx}
\usepackage{amssymb,amsthm,amsfonts,amsmath}
\usepackage{mathtools}
\mathtoolsset{showonlyrefs=true,showmanualtags=false} % Numbers only the referenced equations
\usepackage{hyperref}
\hypersetup{
    colorlinks=true,
    linkcolor=black,          % color of internal links (change box color with linkbordercolor)
    citecolor=black,        % color of links to bibliography
    filecolor=black,      % color of file links
    urlcolor=black  
}
\usepackage{url}
\usepackage{subfigure}
\usepackage{enumerate}
\usepackage[normalem]{ulem}

\usepackage{geometry}

\usepackage{rotating} % to rotate words
\usepackage[section]{algorithm}
\usepackage{algorithmic}

\usepackage{xspace}

\usepackage[small,bf]{caption}

\usepackage{multirow} % righe multiple nelle tabelle
\usepackage{booktabs} % colonne multiple nelle tabelle

\theoremstyle{plain}

\newtheorem{thm}{Theorem}[section]

\newtheorem{lem}[thm]{Lemma}
\newtheorem{prop}[thm]{Proposition}
\newtheorem{setting}[thm]{Setting}

\theoremstyle{definition}
\newtheorem{defn}[thm]{Definition}

\theoremstyle{remark}
\newtheorem{rmrk}[thm]{Remark}
\numberwithin{equation}{section}

\newcommand{\stkout}[1]{\ifmmode\text{\sout{\ensuremath{#1}}}\else\sout{#1}\fi}

\DeclareMathOperator{\Span}{span}
\DeclareMathOperator{\supp}{supp}

\DeclareMathOperator{\Prob}{\mathbb{P}}

\DeclareMathOperator{\clos}{clos}
\DeclareMathOperator{\zero}{zero}

\DeclareMathOperator{\sign}{sign}

\newcommand{\Expe}{\mathbb{E}}

\newcommand{\per}{\textup{per}}

\newcommand{\rea}{\rho}
\newcommand{\dif}{\eta}

\newcommand{\gammavec}{\boldsymbol \gamma}

\newcommand{\onevec}{\mathbf{1}}

\newcommand{\iunit}{\text{i}}
\newcommand{\enum}{\text{e}}
\newcommand{\RR}{\mathbb{R}}

\newcommand{\NN}{\mathbb{N}}
\newcommand{\ZZ}{\mathbb{Z}}

\newcommand{\de}{\,\text{d}}
\DeclareMathOperator{\Real}{Re}
\DeclareMathOperator{\Imag}{Im}

\newcommand{\Qcal}{\mathcal{Q}}

\newcommand{\corsing} {\textsf{CORSING}\xspace}

\newcommand{\WF}      {$\mathcal{WF}$\xspace}

\usepackage[usenames,dvipsnames]{xcolor}

\newcommand{\bm}[1]{{\boldsymbol{#1}}}

\newcommand{\ani}{\textnormal{ani}}
\newcommand{\iso}{\textnormal{iso}}

\newcommand{\Dim}{n}

\DeclareMathOperator{\scal}{scal} % indices of scaling function in the multi-dimensional framekwork